\documentclass[reqno]{amsart}
\usepackage[utf8]{inputenc}
\usepackage{amssymb, amsthm}
\usepackage{mathrsfs}
\usepackage{bm,bbm}
\usepackage{stmaryrd}
\usepackage{euscript}
\usepackage{microtype} 
\usepackage{comment} 

\usepackage{amsfonts}
\DeclareMathAlphabet{\mathpgoth}{OT1}{pgoth}{m}{n}
\DeclareMathAlphabet{\mathesstixfrak}{U}{esstixfrak}{m}{n}
\DeclareMathAlphabet{\mathboondoxfrak}{U}{BOONDOX-frak}{m}{n}
\usepackage{amsmath,amssymb}
\usepackage[bbgreekl]{mathbbol}
\usepackage{yfonts,mathtools}

\usepackage{tikz-cd}
\usepackage[all,cmtip]{xy}

\numberwithin{equation}{section}
\usepackage{soul}

\usepackage{hyperref}
\hypersetup{colorlinks}
\definecolor{darkred}{rgb}{0.5,0,0}
\definecolor{darkgreen}{rgb}{0,0.5,0}
\definecolor{darkblue}{rgb}{0,0,0.5}
\hypersetup{colorlinks, linkcolor=darkblue, filecolor=darkgreen, urlcolor=darkred, citecolor=darkblue}
\makeatletter 
\@addtoreset{equation}{section}
\makeatother  

\numberwithin{equation}{section}

\newtheorem{thma}{Theorem}

\newtheorem{thm}{Theorem}[section]
\newtheorem{cor}[thm]{Corollary}

\newtheorem{conj}[thm]{Conjecture}
\newtheorem{prop}[thm]{Proposition}
 
\newtheorem{lemma}[thm]{Lemma}
\theoremstyle{definition}
\newtheorem{defn}[thm]{Definition}
\theoremstyle{remark}
\newtheorem{rem}[thm]{Remark}

\usepackage{xcolor}
\newtheorem{example}[thm]{Example}

\newtheorem{convention}[thm]{Convention}

\newcommand{\hh}{{\it HH}}
\newcommand{\CC}{{\it CC}}
\renewcommand{\L}{{\bm L}}
\newcommand{\e}{{\bf e}}
\newcommand{\f}{{\bf f}}

\newcommand{\vhf}{{\it VHF}}
\newcommand{\vcf}{{\it VCF}}
\newcommand{\qcf}{{\it QCF}}
\newcommand{\qhf}{{\it QHF}}
\newcommand{\glsm}{{\rm GLSM}}

\newcommand{\llangle}{{\langle\hspace{-0.2pt}\langle}}
\newcommand{\rrangle}{{\rangle\hspace{-0.2pt}\rangle}}

\newcommand{\beq}{\begin{equation}}
\newcommand{\eeq}{\end{equation}}
\newcommand{\beqn}{\begin{equation*}}
\newcommand{\eeqn}{\end{equation*}}
\newcommand{\ov}{\overline}
\newcommand{\mb}{\mathbb}
\newcommand{\mc}{\mathcal}
\newcommand{\mf}{\mathfrak}

\newcommand{\bb}{{\mf b}}

\newcommand{\wt}{\widetilde}
\newcommand{\wh}{\widehat}

\newcommand{\ham}{{\rm Ham}}
\newcommand{\tham}{{\rm H}{\widetilde{\rm am}}}

\newcommand{\uds}[1]{\underline{\smash{#1}}}

\newcommand{\ev}{{\rm ev}}

\title[Franks' dichotomy for toric manifolds]{Franks' dichotomy for toric manifolds, Hofer--Zehnder conjecture, and gauged linear sigma model}

\author{Shaoyun Bai}
\address{Columbia University, 2990 Broadway, New York, NY 10027, USA}
\email{sb4841@columbia.edu}

\author{Guangbo Xu}
\address{Department of Mathematics, Rutgers University, Hill Center--Busch Campus, 110 Frelinghuysen Road, Piscataway, NJ 08854-8019, USA}
\email{guangbo.xu@rutgers.edu}

\thanks{The second author is supported by NSF DMS-2345030.}

\date{\today}

\begin{document}

\maketitle

\begin{abstract}
We prove that for any compact toric symplectic manifold, if a Hamiltonian diffeomorphism admits more fixed points, counted homologically, than the total Betti number, then it has infinitely many simple periodic points. This provides a vast generalization of Franks' famous two or infinity dichotomy for periodic orbits of area-preserving diffeomorphisms on the two-sphere, and establishes a conjecture attributed to Hofer--Zehnder in the case of toric manifolds. The key novelty is the application of gauged linear sigma model and its bulk deformations to the study of Hamiltonian dynamics of symplectic quotients. 
\end{abstract}

\setcounter{tocdepth}{1}
\tableofcontents



\section{Introduction}

\subsection{Main result}

A celebrated theorem of Franks \cite{Franks_1, Franks_2} says that any area-preserving diffeomorphism on $S^2$ has either 2 or infinitely many periodic points. An extension of this dichotomy to all compact symplectic toric manifolds is obtained in this paper.


\begin{thma}\label{thm:main}
Let $X$ be a compact symplectic toric manifold and $\phi: X \to X$ be a Hamiltonian diffeomorphism with isolated fixed points. If the number of fixed points of $\phi$, counted homologically, is greater than the total rank of homology of $X$, 
then $\phi$ has infinitely many simple periodic points.
\end{thma}

Here a point $x \in X$ is called a simple periodic point (of period $k$) if $\phi^k(x) = x$ for some positive integer $k$ and $\phi^l(x) \neq x$ for all $l < k$. The homological count of the number of fixed points of $\phi$, denoted by $N(\phi)$, is defined as 
\beqn
N(\phi):= \sum_{x \in {\rm Fix}(\phi)} {\rm dim} HF^{\rm loc}(\phi, x)
\eeqn
where $HF^{\rm loc}(\phi, x)$ is the local Floer homology at a fixed point $x$. In particular, when $\phi$ is nondegenerate, meaning that for all $x \in {\rm Fix}(\phi)$, the linearization $\phi$ at $x$ has no eigenvalue $1$, then $N(\phi)$ is equal to the naive count of fixed points. 

Theorem \ref{thm:main} also resolves the following visionary conjecture set forth by Hofer and Zehnder \cite[Page 263]{hofer-zehnder-book} in the case of toric manifolds.

\begin{conj}[Hofer--Zehnder conjecture]
Let $\phi$ be a Hamiltonian diffeomorphism on a compact symplectic manifold. If $\phi$ has more fixed points than the lower bound provided by the Arnold conjecture, then $\phi$ has infinitely many simple periodic points.
\end{conj}

From a dynamical perspective, the Hofer--Zehnder conjecture is related to the problem of finding Hamiltonian diffeomorphisms with minimal number of invariant ergodic measures, i.e., pseudo-rotations. Such maps are usually very rare. For toric manifolds, such pseudo-rotations were recently constructed by Le Roux--Seyfaddini \cite{leroux_seyfaddini}  using the Anosov--Katok method \cite{Anosov_Katok}.

Our proof of Theorem \ref{thm:main} is built upon new connections between two distant stories:

\begin{enumerate}
    \item {\bf Quantitative Floer theory}. The energy filtration on Floer chain complexes leads to a powerful framework in studying Hamiltonian dynamics. Notions from topological data analysis, such as persistence module and barcode, have been incorporated with the quantitative symplectic geometry to prove many exciting results beyond the scope of the Arnold conjecture. Insights from \cite{Entov_Polterovich_1} and \cite{Shelukhin_2022} about the quantitative implication of the semisimplicity of quantum cohomology are particularly crucial ingredients.

\item {\bf Gauged linear sigma model (GLSM)}. The GLSM was introduced by Witten \cite{Witten_LGCY} in the physical context of 2-dimensional supersymmetric field theory, later becoming very impactful in enumerative algebraic geometry and mirror symmetry. The adaptation \cite{Xu_VHF}\cite{Wu_Xu} of GLSM to Hamiltonian Floer theory brings in three significant advantages in carrying out the proof of Theorem \ref{thm:main} along the lines of \cite{Shelukhin_2022}: a) all curve counts are integers even for general toric manifolds, b) transversality can be achieved without using the virtual technique, and c) the semisimplicity can be achieved using ``small'' bulk deformation, i.e., divisor classes.
\end{enumerate}
These new connections we build in this paper are perhaps equally interesting as our main result as a generalization of Franks' theorem and deserve further exploration.

\subsection{Pseudoholomorphic curves vs. vortices}

A crucial feature of our proof is the systematic adaptation of the GLSM to a symplecto-geometric problem. As mentioned above, this approach has tremendous advantages in various levels over the traditional, i.e., the nonlinear sigma model (NLSM) approach. Here we try to provide a brief sketch of the idea of GLSM and explain the advantageous effects on Floer theory.

\subsubsection{Pseudoholomorphic curves and the counting problem}

Since Gromov \cite{Gromov_1985} the theory of pseudoholomorphic curves has become the main infrastructure of symplectic geometry. A central problem about pseudoholomorphic curves, roughly speaking, is to define invariants by counting them. The difficulty to count lies in the fact that the moduli spaces are often very singular and that one needs to regularize the moduli spaces in order to define the contributions of the singularities. There are two main regularization methods which we recall here. 

The first one is to ``geometrically'' perturb the pseudoholomorphic curve equation, resembling the strategy of perturbing the Riemannian metric in Yang--Mills gauge theory. This method works for symplectic manifolds satisfying the topological condition called ``semipositivity,'' including the exact ($\omega = d\theta$), positively monotone ($c_1>0$), and Calabi--Yau ($c_1 = 0$) cases. Though it limits the  range of applications, due to its simplicity the method of geometric perturbations is still favored in many situations, with numerous distinguished works based upon.

The reason that the geometric perturbation method does not extend beyond the semipositive case is due to the issue of ``multiple covers of holomorphic spheres with negative Chern numbers.'' These multiple covers cause very wild singularities of the moduli spaces which cannot be perturbed away by geometric perturbations. The second regularization method, originally invented to deal with the general situation, is to embed the singular moduli spaces into regular spaces and apply ``abstract perturbations'' which no longer have concrete geometric meanings. The construction becomes substantially more involved and bulky. We mention one (but not the only) challenging task that is relevant in Floer theory. In the semipositive setting, as one can geometrically perturb the Floer equation to achieve transversality, moduli spaces with negative expected dimensions are automatically empty. As a result, one can work with individual moduli spaces. Beyond the semipositive setting, the regularization process must be taken simultaneously with all moduli spaces regardless of their expected dimensions. Greater care needs to be taken in order to obtain a coherent system of perturbations. 

Despite the technical complexities caused by the spherical multiple covers, such singularities of moduli spaces shall be regarded as a feature, not a bug, in Gromov--Witten theory. Numerically, these multiple covers can make certain GW invariants fractional numbers. For example, the famous Aspinwall--Morrison formula (see \cite{Aspinwall_Morrison_1993}\cite{Voison_Aspinwall_Morrison}) says that $k$-fold covers of a rigid holomorphic sphere in a Calabi--Yau threefold shall contribute $1/k^3$ to the genus zero Gromov--Witten invariant. From an abstract perspective, the non-integrality is due to the fact that moduli spaces are actually orbifolds/stacks rather than manifolds/spaces, while stacky points contribute fractionally to topological invariants such as the Euler number.

However, in Floer theory, regarding the multiple cover issue, at least the numerical non-integrality, the authors somehow view it as a bug rather than a feature. One way to justify this viewpoint is to look at Hamiltonian Floer theory. We know that Hamiltonian Floer homology should be isomorphic to the classical Morse homology which can be defined over integers. In the non-semipositive case, the spherical multiple covers, which appear in wrong dimensions, cannot be easily separated from Floer cylinders, although the latter can be made transverse by geometric perturbations. One must consider the fractional contributions from the stacky parts of moduli spaces. As a result, the Floer homology is only identified with the Morse homology over ${\mb Q}$, missing the information from the torsion part. This bug also makes the construction of interesting algebraic structures, such as equivariant pair-of-pants product \cite{Seidel_pants}\cite{Shelukhin-Zhao} and quantum Steenrod operation \cite{wilkins2020construction} impossible beyond the semipositive case. Recently the authors \cite{Bai_Xu_Arnold} constructed a Hamiltonian Floer chain complex over ${\mb Z}$ using the more refined abstract perturbation scheme discovered by Fukaya--Ono \cite{Fukaya_Ono_integer} and proved an integral version of the Arnold conjecture for general compact symplectic manifolds. This shows that the contributions from the ``stacky'' parts of moduli spaces can indeed be separated out in Floer theory. Still, the whole construction remains very involved as we need to work with infinitely many moduli spaces simultaneously, even if the recent invention of global Kuranishi charts by Abouzaid--McLean--Smith \cite{AMS, AMS2} (see also Hirschi--Swaminathan \cite{Hirschi_Swaminathan}) greatly reduced the technical burden.

\subsubsection{Mathematical theory of the GLSM}

It is clear that with the old infrastructure of pseudoholomorphic curves, the generality of applications and the simplicity of methodology hardly get along. The bug of numerical non-integrality is also hard, if not impossible, to remove. Remarkably, the new infrastructure of the gauged linear sigma model, which we systematically adopt in this paper, works for very general situations (i.e. beyond semipositive case) while only requiring the simple geometric perturbation method and provides an easy fix of the issue of spherical multiple covers.

Let us be slightly more technical. We would like to briefly discuss the {\bf vortex equation}, the GLSM counterpart of the pseudoholomorphic curve equation. We use the case of $X = \mb{CP}^n$ as an example. First, one views $\mb{CP}^n$, a nonlinear space, as the symplectic reduction of the linear space $V = {\mb C}^{n+1}$ by the group $U(1)$. The moment map is $\mu(x_0, \ldots, x_n) = |x_0|^2 + \cdots + |x_n|^2 - c$. Then, instead of considering maps from a Riemann surface $\Sigma$ to $\mb{CP}^n$ (the NLSM viewpoint), we consider the so-called ``gauged maps'' into the linear space $V$. They are nothing but triples $(L, A, u)$, where $L \to \Sigma$ is a Hermitian line bundle, $A$ is a Hermitian connection on $L$, and $u = (u_0, \ldots, u_n)$ is an $(n+1)$-tuple of sections of $L$. The vortex equation reads
\begin{align}\label{vortex1}
&\ \ov\partial_A u_i = 0,\ &\  F_A = 2\pi {\bf i} \mu(u) {\rm dvol}_\Sigma = 2\pi {\bf i} ( |u_0|^2 + \cdots + |u_n|^2 - c) {\rm dvol}_\Sigma.
\end{align}
Up to gauge transformations, this defines a holomorphic line bundle with $n+1$ holomorphic sections. Notice that the equation depends on a volume form ${\rm dvol}_\Sigma$ on the surface, losing the conformal invariance of pseudoholomorphic curves. But this is a very useful feature, as one can manipulate the volume form to adjust the position of the section $u$ relative to the moment map. For example, if one needs to define an evaluation map to the quotient space $\mb{CP}^n$ at a marked point, then one take a volume form which is cylindrical around the point. The finiteness of energy forces any solution to approach to the level set $\mu = 0$ at the cylindrical infinity. 

The reason that we can avoid sphere bubbles is simply because the target space $V$ has only constant holomorphic spheres. We can view this feature as giving a different way of compactifying the space of pseudoholomorphic curves downstairs. Notice that an $(n+1)$-tuple of holomorphic sections $(u_0, \ldots, u_n)$ defines a holomorphic map into $\mb{CP}^n$ precisely when these sections have no common zeroes. Sphere bubbles form in a sequential limit when some of their distinct zeroes collide. However, in the vortex equation, there is no restriction on where the sections can vanish.

The absence of sphere bubbles make the geometric perturbation method possible, even if the target space downstairs is not semipositive. For example, one can perturb the standard almost complex structure on the linear space $V$ as long as the perturbation is invariant under the gauge group action. As the solution intersect the region where the group action is free, the invariance restriction does not make the perturbation argument more difficult than the non-equivariant case.

We want to convince the reader that although the GLSM approach is only available for symplectic reductions, the drop of generality is small in a practical sense. Indeed, most target manifolds which can be explicitly defined have some ``linear'' feature, such as toric manifolds, flag manifolds, or subvarieties inside them. They are the targets for which explicit calculations can be expected.

The mathematical theory of the GLSM develops in two different directions. The first one is constructing gauged Gromov--Witten type invariants, relating them with classical GW invariants, and verifying physicists' predictions such as classical mirror symmetry or LG/CY correspondence. Although there are substantial developments on the symplectic side (see \cite{Mundet_thesis, Mundet_2003, Mundet_Tian_2009} \cite{Gaio_Salamon_2005} \cite{Ziltener_thesis, Ziltener_book} \cite{Tian_Xu, Tian_Xu_2, Tian_Xu_3, Tian_Xu_geometric}), the progress has been far more successful on the algebraic side (\cite{CKM_quasimap}\cite{FJR_GLSM}\cite{Chen_Janda_Ruan_2021}\cite{CLLL_16, CLLL_2019}\cite{CGLL_2021}) with significant applications \cite{CK_2020} \cite{CGL_2018, CGL_2021}\cite{Guo_Janda_Ruan_2018} etc.

The second direction, which only exists in the symplectic setting, is Floer-theoretic applications of the GLSM technique, including \cite{Frauenfelder_thesis, Frauenfelder_2004}\cite{Woodward_toric}\cite{Xu_VHF}\cite{Wu_Xu}. The key novelty is to take advantage of the simple topology upstairs to reduce the cost of regularization. Meanwhile, the moduli spaces can sometimes be explicitly identified, making certain computations such as the disk potential calculation available.

\subsubsection{Exact relation between upstairs and downstairs}

To see that the GLSM approach is an ideal alternative of the pseudoholomorphic curve approach, but not anything distantly different, we explain a concrete relation between their numerical or algebraic outputs. Physically, this relation is about the renormalization process in GLSM. Similar phenomena show up in different situations, such as the Atiyah--Floer conjecture \cite{Atiyah_Floer_conjecture}.  Mathematically, this relation can be revealed by considering the limit of the vortex equation when we blow up the volume form on the surface. In the context of \eqref{vortex1}, consider the $\epsilon \to 0$ limit of the following variant:
\begin{align*}
&\ \ov\partial_A u_i = 0,\ &\ \epsilon^2 F_A = 2\pi {\bf i} \mu(u) {\rm dvol}_\Sigma.
\end{align*}
Since the vortex theory is supposed to be independent of the parameters, then as $\epsilon \to 0$, we shall obtain invariant countings, chain homotopy equivalent complexes, etc. On the other hand, if $\epsilon \to 0$, the equation indicates that $\mu(u) \to 0$. Modulo gauge symmetry, in the limit the vortices become pseudoholomorphic curves in $X = \mu^{-1}(0)/K$. However, by looking more carefully at the $\epsilon \to 0$ limit, one can find that vortices do not converge to pseudoholomorphic curves entirely, as certain gauge-theoretic bubbles may appear. Consequently, one should expect
\beq\label{eqn12}
{\rm GLSM} \approx {\rm NLSM} + {\rm correction}.
\eeq
The correction term is intimately related to mirror symmetry. For example, when $X$ is Calabi--Yau, the correction can be viewed as a coordinate change on the K\"aher moduli on the A-side, which should match the coordinates on the complex moduli on the B-side. 

To finish, we remark that although the work conducted in this paper does not rely on a rigorous verification of \eqref{eqn12} in the Floer setting, this picture shows that the GLSM approach of the Floer theory is equivalent to the ordinary one up to certain deformation.

\subsection{Outline of the proof}\label{subsec:intro-2}

Our proof is greatly indebted to the argument of Shelukhin \cite{Shelukhin_2022}, which boils down to two key elements: 1) the uniform upper bound of the boundary depth under the semisimplicity condition, and 2) the linear growth of the total bar length under prime iterations. The reason why this argument can be applied to the toric case is because the quantum cohomology of toric manifolds are ``generically semisimple,'' a fact which can be easily deduced from mirror symmetry (see below). 
For 2), it is crucial for us to use the GLSM approach which leads to Floer chain complexes over ${\mb Z}$, hence all their mod $p$ reductions.


\subsubsection{Hamiltonian Floer theory via GLSM}

(See Section \ref{section4} for more details.) Let $X$ be a compact toric manifold. For any nondegenerate 1-periodic family of Hamiltonians $H_t: X \to {\mb R}$ ($t \in S^1$) 
in \cite{Xu_VHF} the second named author constructed a Floer chain complex over ${\mb Z}$ 
following the suggestion of \cite{Cieliebak_Gaio_Salamon_2000} using the GIT presentation of the toric manifold. Basically, using the moment polytope $P\subset {\mb R}^n$ one can realize $X$ as the symplectic reduction of a vector space $V \cong {\mb C}^N$ by a Hamiltonian action of $T^{N-n}$. Upon choosing a $T^{N-n}$-invariant lift $\wh H_t: V \to {\mb R}$ of $H_t$ and $T^{N-n}$-invariant almost complex structure $\wh J_t$, one can define the {\bf vortex Floer complex}
\beqn
\vcf_\bullet( \wh H, \wh J)
\eeqn
whose generators are still 1-periodic orbits ``downstairs'' in $X$, but with differentials counting gauge equivalence classes of Hamiltonian perturbed vortices ``upstairs'' in $V$. As the target manifold $V$ is contractible, the compactness problem about the moduli spaces of vortices becomes very simple, similary to the symplectically aspherical case in ordinary Hamiltonian Floer theory. As a result, one can achieve transversality by just perturbing $(\wh H, \wh J)$; moreover, the complex can be defined over the Novikov ring with integer coefficients:
\beqn
\Lambda_{\mb Z}:= \left\{ \sum_{i=1}^\infty a_i T^{\lambda_i}\ |\ a_i \in {\mb Z},\ \lambda_i \to +\infty \right\}.
\eeqn
The continuation maps and pair-of-pants product can also be constructed, resulting in a quantum algebra $\vhf_\bullet(V)$ over $\Lambda_{\mb Z}$ which is independent of the data $(\wh H, \wh J)$.\footnote{The Piunikhin--Salamon--Schwarz (PSS) construction for the vortex Floer theory, however, do not work over ${\mb Z}$. However, by the special topology of toric manifolds, the homology group $\vhf_\bullet(V)$ is still isomorphic to the ordinary homology of $X$.}

The filtered Floer theory can also be developed using the vortex equation. Indeed, an equivariant action functional makes the vortex Floer complex a Floer--Novikov complex in the sense of Usher \cite{Usher_2008}, hence leads to spectral invariants as defined in \cite{Wu_Xu}. By working over a Novikov field $\Lambda_{\mb K}$, one can also associate to each Hamiltonian $H_t$ a {\bf barcode} in the sense of Usher--Zhang \cite{Usher_Zhang_2016} and an infinite-dimensional persistence module.

\subsubsection{Mirror symmetry via GLSM and generic semisimplicity}

We recall certain folklore knowledge about toric mirror symmetry. Instead of classical mirror symmetry for Calabi--Yau manifolds \cite{COGP}, the mirror object of a compact toric manifold $X$ is a {\bf Landau--Ginzburg model} $(Y, W)$, where $Y$ is an algebraic torus with coordinates $y = (y_1, \ldots, y_n)$ and $W$ is a holomorphic function in $y$. This mirror correspondence implies the ring isomorphism 
\beq\label{eqn11}
QH^\bullet(X) \cong {\rm Jac}(W)
\eeq
where the right hand side, the Jacobian ring, is roughly the ring of functions on the critical locus of $W$. When $W$ is a Morse function, mirror symmetry implies that $QH^\bullet(X)$ is semisimple. 

The work of Fukaya--Oh--Ohta--Ono \cite{FOOO_toric_1, FOOO_toric_2, FOOO_mirror} provides a geometric description of the mirror superpotential and the isomorphism \eqref{eqn11}. Indeed, $W$ can be viewed as the count of (stable) holomorphic disks with boundary in Lagrangian tori and the isomorphism \eqref{eqn11} is induced from the {\bf closed-open map}. The mirror superpotential $W$ can be Morsified by certain deformations, which can be realized by {\bf bulk deformations}. Roughly speaking, cohomology classes ${\mf b}$ in $X$ can deform the ring $QH^\bullet(X)$ to a family $QH^\bullet_{\mf b}(X)$, deform the mirror superpotential $W$ to a family $W_{\mf b}$, and deform the closed-open map to a family of isomorphisms
\beq\label{eqn12}
QH^\bullet_{\mf b}(X) \cong {\rm Jac}(W_{\mf b}).
\eeq
Moreover, for a generic ${\mf b}$, the superpotential $W_{\mf b}$ a Morse function, implying the semisimplicity of $QH^\bullet_{\mf b}(X)$. Notice that, the rigorous mathematical proof of \eqref{eqn12} by Fukaya {\it et.al.} relies heavily on the machinery of (equivariant) Kuranishi structures.

Using the GLSM, one can reproduce the above picture with significantly reduced technicality. Indeed, viewing $X$ as the GIT quotient of $V$, there is a GLSM version of disk potential, firstly discussed by Givental \cite{Givental_potential} and Hori--Vafa \cite{Hori_Vafa} and rigorously defined by Woodward \cite{Woodward_toric}. We denote temporarily the GLSM version of the disk potential by $W^{\rm GLSM}$ and the previous disk potential defined by Fukaya--Oh--Ohta--Ono by $W^{\rm NLSM}$. We remind the reader that when $X$ is Fano, $W^{\rm GLSM} = W^{\rm NLSM}$ (see \cite{Cho}\cite{Cho_Oh}). In the general situation, schematically, one shall have the following commutative diagram
\beqn
\vcenter{ \xymatrix{ \vhf_\bullet^{\mf b}(V) \ar[r] \ar[d] & {\rm Jac}( W_{\mf b}^{\rm GLSM})  \ar[d]\\
           QH^\bullet_{{\mf b}'}(X) \ar[r] & {\rm Jac}( W_{{\mf b}'}^{\rm NLSM})} }.
\eeqn
Here the left vertical arrow and the correspondence ${\mf b} \mapsto {\mf b}'$, often called the {\bf quantum Kirwan map}, has been discussed in \cite{Gaio_Salamon_2005}\cite{Ziltener_thesis, Ziltener_book}\cite{Woodward_15}, the right vertical arrow, which is realized by a coordinate change on the variables $y = (y_1, \ldots, y_n)$, has been discussed in various cases in \cite{FOOO_mirror}\cite{CLLT}\cite{Woodward_Xu}. This picture allows us to work entirely upstairs.

\begin{thma}\label{thm:main-2}
There exists a bulk deformation of the form 
\beq\label{small_bulk}
{\mf b} = \sum_{j=1}^N \log c_j V_j
\eeq
with $c_j \in {\mb Z}[{\bf i}] \cong {\mb Z} \oplus {\bf i} {\mb Z}$ such that the quantum algebra $\vhf_\bullet^{\mf b}(V)$ is semisimple over $\Lambda_{\ov{\mb Q}}$. 
\end{thma}

Using purely algebraic argument one can derive the semisimplicity in finite characteristics.

\begin{cor}
Let ${\mf b}$ be given as in Theorem \ref{thm:main-2}. Then for any sufficiently large prime $p$, the quantum algebra $\vhf_\bullet^{\mf b}(V)$ is semisimple over the Novikov field $\Lambda_{\ov{\mb F}_p}$, where $\ov{\mb F}_p$ is an algebraic closure of ${\mb F}_p$. 
\end{cor}

\subsubsection{Boundary depth estimate and linear growth of total bar length}

The bulk-deformed vortex Floer theory still provides us powerful quantitative invariants of Hamiltonian diffeomorphisms. Basically, over a Novikov field $\Lambda_{\mb K}$ the bulk-deformed complex $\vcf_\bullet^{\mf b}(\wh H, \wh J)$ induces a barcode, whose longest finite bar has length equal to the {\bf boundary depth} introduced by Usher \cite{Usher_boundary_depth}. The boundary depth only depends on the induced Hamiltonian diffeomorphism $\phi$ on $X$ and we denote it by $\beta_{\mb K}^{\mf b}(\phi)$; it also extends to possibly degenerate Hamiltonian diffeomorphisms. We prove the following uniform upper bound on boundary depth, similar to that of \cite[Theorem B]{Shelukhin_2022}.

\begin{thma}\label{thmc}
Let ${\mf b}$ be given as in Theorem \ref{thm:main-2}. Then there exists $C>0$ such that for all sufficiently large prime $p$, for all Hamiltonian diffeomorphism $\phi$ on $X$, one has 
\beq\label{eqn:uniform-bound}
\beta_{\ov{\mb F}_p}^{\mf b}(\phi) \leq C.
\eeq
\end{thma}

On the other hand, the total bar length of the bulk-deformed vortex Floer complex can be extended to Hamiltonian diffeomorphisms with isolated fixed points. Denote it (over $\Lambda_{\mb K}$) by
\beqn
\tau_{\mb K}^{{\mf b}}(\phi) \in {\mb R}.
\eeqn
Following \cite{Shelukhin-Zhao} and \cite{Shelukhin_2022}, via the algebraic Tate construction and geometrically constructed equivariant pair-of-pants product (extending Seidel's construction \cite{Seidel_pants} in characteristic 2), one obtains the following analogue of \cite[Theorem D]{Shelukhin_2022}.

\begin{thma}\label{thm_total_length_growth}
For any bulk ${\mf b}$ for the form \eqref{small_bulk} and any odd prime $p$, we have the inequality
\beq\label{eqn:linear-grow}
\tau_{\ov{\mb F}_p}^{\mf b} (\phi^p) \geq p \cdot \tau_{\ov{\mb F}_p}^{\mf b} (\phi).
\eeq
\end{thma}


Finally, Theorem \ref{thm:main} follows by a simple contradiction between Theorem \ref{thmc} and Theorem \ref{thm_total_length_growth}, if we assume the number of fixed points is great than the total rank of homology (see Section \ref{section_proof}).

\subsection{Outlook and speculations}

We are very surprised to find out that classical considerations from mirror symmetry can be quite useful for investigations in Hamiltonian dynamics. We expect such a connection could open up new avenues for future research. As mentioned above, GLSM can more generally be used to study symplectic topology and Hamiltonian dynamics of other symplectic/GIT quotients or complete intersections in them, the latter of which requires studing the gauged Witten equation (see \cite{Tian_Xu, Tian_Xu_geometric}) with Hamiltonian perturbations. It is conceivable that one could resolve the Hofer--Zehnder conjecture for a broader class of symplectic quotients, provided that certain form of closed string mirror symmetry can be established.

On a different note, except for deploying tools like GLSM, as mentioned earlier, there are some recent advances \cite{Bai_Xu_2022, Bai_Xu_Arnold} on defining Hamiltonian Floer theory over integers for general symplectic manifolds. The methods from \emph{loc. cit.} are general enough for us to expect that a version of symplectic Smith-type inequality should hold using such a theory. Deriving dynamical applications using such a toolkit, including proving the Hofer--Zehnder conjecture in more general settings, is another topic for future research.

\subsection{Outline of the paper}
The following provides an outline of this paper.

\begin{itemize}

\item Basic notions related to toric manifolds are recalled in Section \ref{sec:prelim}, which also includes an introduction to the symplectic vortex equations arising from GLSM.

\item In Section \ref{sec:alg-prelim}, various algebraic preliminaries relevant for our purpose, including semisimple algebras over Novikov rings defined over fields with possibly positive characteristics, abstract setups for filtered Floer theories, persistence modules, and $A_\infty$ algebras and their Hochschild cohomology, are recalled systematically.

\item A (filtered) Hamiltonian Floer theory package in the vortex setting is recorded in Section \ref{sec:ham-package}. Most notably, we introduce bulk deformations in vortex Hamiltonian Floer theory which allow us to incorporate ideas from generic semisimplicity of quantum homology to derive applications in quantitative symplectic topology.

\item In Section \ref{sec:local-Floer}, we introduce local Floer theory in the vortex setting in order to establish Theorem \ref{thm:main} for Hamiltonian diffeomorphisms with isolated but degenerate fixed points.

\item The main purpose of Section \ref{sec:beta} is to prove Theorem \ref{thm_boundary_depth} = Theorem \ref{thmc}, which ensures a uniform upper bound on the boundary depth of the bulk-deformed vortex Hamiltonian Floer persistence module of any Hamiltonian diffeomorphism provided that the bulk-deformed vortex quantum homology is semisimple.

\item In Section \ref{sec:equiv}, we develop ${\mb Z}/p$-equivariant vortex Hamiltonian Floer theory by adapting the work \cite{Seidel_pants, Shelukhin-Zhao} in the GLSM setting. Theorem \ref{thm_total_length_growth} = Theorem \ref{thm:smith} is proven as a consequence by appealing to the work of Shelukhin \cite{Shelukhin_2022}.

\item In Section \ref{section8} we wrap up the preparations and give the proof of Theorem \ref{thm:main}.

\item We turn our attention to Lagrangian Floer theory in Section \ref{section9}. The key result is to demonstrate the existence of  a ``convenient" bulk deformation (cf. Definition \ref{defn:convenient}) whose associated Fukaya category (in the GLSM setting) takes a very simple form, such that its Hochschild cohomology is a semisimple algebra.

\item Lastly, in Section \ref{section10}, Theorem \ref{thm_CO} = Theorem \ref{thm:main-2} is proven by showing that the closed-open string map is a unital ring isomorphism.

\end{itemize}

\subsection*{Acknowledgements}: 

We thank Marcelo Atallah, Hiroshi Iritani, Han Lou, Egor Shelukhin, Nick Sheridan, Michael Usher, and Chris Woodward for useful discussions and email correspondences. The first-named author is grateful to the Simons Center for Geometry and Physics for its warm hospitality during Spring 2023.

During the preparation of our paper we learned the then ongoing work of Atallah--Lou \cite{Atallah_Lou} on the Hofer--Zehnder conjecture for semipositive symplectic manifolds; we also communicated with them on methods of deriving semipositivity from characteristic zero to positive characteristics, although we have developed a completely different approach using the quantum multiplication by the first Chern class.

\section{Geometric preliminaries}\label{sec:prelim}

We recall basic notions about toric symplectic manifolds and symplectic vortex equations.

\subsection{Toric manifolds as symplectic quotients}\label{subsection_toric}

We briefly recall the notion of symplectic reduction/quotients. Let $K$ be a compact Lie group with Lie algebra ${\mf k}$. Let $(V, \omega_V)$ be a symplectic manifold having a Hamiltonian $K$-action with moment map
\beqn
\mu: V \to {\mf k}^*.
\eeqn

The {\bf symplectic reduction} of $V$ (with respect to the $K$-action and the moment map) is
\beqn
X:= \mu^{-1}(0)/K.
\eeqn
We always assume that $0$ is a regular value of $\mu$ and the $K$-action on $\mu^{-1}(0)$ is free. This assumption implies that $X$ is a smooth manifold with a naturally induced symplectic form $\omega_X$. 

When $V$ has a $K$-invariant integrable almost complex structure $J_V$, the $K$-action can be extended to its complexification $K^{\mb C}$ as a holomorphic action. When this is the case, (under certain extra conditions), the Kempf--Ness theorem says that the symplectic reduction can be identified with the geometric invariant theory (GIT) quotient.

Symplectic toric manifolds can be realized as symplectic quotients of a vector space. We provide a minimal description of symplectic toric manifolds necessary for this paper. An $2n$-dimensional compact symplectic toric manifold $X$ is described by a convex polytope $P \subset {\mb R}^n$ satisfying the following conditions.
\begin{enumerate}
    \item For each face $\partial_j P$ of $P$, there are ${\bf v}_j \in {\mb Z}^n$ and $\lambda_j\in {\mb R}$ such that ${\bf v}_j$ is an inward normal vector and the face is defined by 
    \beqn
    \partial_j P = \{ {\bf u}\in {\mb R}^n\ |\ \langle {\bf v}_j, {\bf u}\rangle = \lambda_j \}.
    \eeqn

    \item For each vertex of $P$, the normal vectors ${\bf v}_{j_1}, \ldots, {\bf v}_{j_n}$ of all adjacent faces form a ${\mb Z}$-basis of ${\mb Z}^n$. 
\end{enumerate}

In this paper, denote by $N$ the number of faces of $P$. We can realize $X$ as the symplectic quotient of ${\mb C}^N$ (with the standard symplectic form) by the $N-n$ dimensional torus $K = T^{N-n} = (S^1)^{N-n}$. The collection of vectors ${\bf v}_1, \ldots, {\bf v}_N$ defines a linear map 
\beqn
\tilde \pi_P: {\mb R}^N \to {\mb R}^n
\eeqn
which sends ${\mb Z}^N$ onto ${\mb Z}^n$. Hence it induces a surjective group homomorphism
\beqn
\pi_P: T^N \to T^n.
\eeqn
Let $K$ be the kernel of $\pi_P$. Hence $K$ acts on ${\mb C}^N$ as a subgroup of $T^N$. Notice that for the standard $\wh K = T^N$-action, the moment map can be written as 
\beqn
\wh \mu(x_1, \ldots, x_N) = \left( \pi |x_1|^2 - \lambda_1, \ldots, \pi |x_N|^2 - \lambda_N \right) \in {\mb R}^N \cong \wh {\mf k}^*.
\eeqn
Then the moment map of the $K$-action is simply the composition
\beqn
\xymatrix{ V \ar[r]^{\wh \mu} & \wh{\mf k}^* \ar[r] & {\mf k}^*}.
\eeqn
On the other hand, there is a residual torus action on $X$ by the quotient $T^n$ which is the torus action usually appear in the discussion of toric manifolds. Denote the associated moment map by
\beq
\pi_X: X \to {\mb R}^n
\eeq
whose range is actually the moment polytope $P$.

Notice that if one translates the moment polytope $P$ in ${\mb R}^N$ by a vector in ${\mb R}^n$, then it does not change the moment map $\mu$ and hence the symplectic form on $X$.

\subsection{Symplectic vortex equation}

The symplectic vortex equation was originally introduced by Cieliebak--Gaio--Salamon \cite{Cieliebak_Gaio_Salamon_2000} and Mundet \cite{Mundet_thesis}. It is a generalization of the pseudoholomorphic curve equation to the equivariant setting. Here we briefly recall its setup and some basic analytical result. We restrict to the toric setting described above. 

\subsubsection{Gauged maps and vortex equation}

Vortices are special ``gauged maps'' in a similar way as holomorphic curves are special ordinary maps. Let $\Sigma$ be a Riemann surface. A {\bf gauged map} from $\Sigma$ to the vector space $V$ is a triple ${\mf u} = (P, A, u)$ where $P\to \Sigma$ is a principal $K$-bundle, $A\in {\mc A}(P)$ is a connection on $P$, and $u$ is a section of the associated vector bundle $P(V):= P\times_K V$. The group of gauge transformations ${\mc G}(P)$, which, in the abelian case, is the group of smooth maps 
\beqn
g: \Sigma \to K,
\eeqn
which acts on gauged maps by 
\beqn
g^* {\mf u} = g^* (P, A, u) = (P, g^* A, g^* u) = (P, A + g^{-1} dg, g^{-1} u).
\eeqn

We need three quantities to define the vortex equation. First the covariant derivative of $u$ is a section 
\beqn
d_A u \in \Omega^1(P, u^* TV)
\eeqn
which descends to an element in $\Omega^1(\Sigma, u^* TV / K)$.
There are also the curvature and the moment potential
\begin{align*}
    &\ F_A \in \Omega^2(\Sigma, {\rm ad} P),\ &\ \mu(u) \in \Omega^0(\Sigma, {\rm ad}P^* ).
\end{align*}
By choosing an invariant inner product on the Lie algebra ${\mf k}$ one can identify ${\rm ad}P\cong {\rm ad} P^*$; by choosing a volume form $\nu_\Sigma$ one can identify $\Omega^2 \cong \Omega^0$. The gauged map ${\mf u}$ is called a {\bf vortex} if 
\begin{align}\label{vortex_equation}
&\ \ov\partial_A u = 0,\ &\ *F_A + \mu(u) = 0.
\end{align}
Here $\ov\partial_A u$ is the $(0,1)$-part of the covariant derivative $d_A u$. Both equations are invariant under gauge transformations. The energy of a vortex is defined to be 
\beqn
E({\mf u}) = \frac{1}{2} \int_\Sigma \left( \| d_A u \|^2 + \| F_A \|^2 + \|\mu(u)\|^2 \right) \nu_\Sigma.
\eeqn

Analogous to pseudoholomorphic curves, vortices satisfy an energy identity. Suppose $\Sigma$ is closed. Then each gauged map ${\mf u}$ represents an equivariant homology class $[{\mf u}] \in H_2^K(V; {\mb Z})$ defined as follows. The section $u: \Sigma \to P(V)$ can be identified with a $K$-equivariant map $\tilde u: P \to V$. Let $EK \to BK$ be the universal $K$-bundle. The classifying map of $P \to \Sigma$ is a map $\iota: \Sigma \to BK$ which is covered by a bundle map $\tilde \iota: P \to EK$. Then the equivariant map $(\tilde \iota, \tilde u): P \to EK \times V$ descends to a continuous map from $\Sigma$ to $(EK \times V)/K$, which represents a class $[{\mf u}] \in H_2^K(V; {\mb Z})$. In the toric case, this class is just the degree of the principal bundle $P\to \Sigma$. Then for any gauged map ${\mf u} = (P, A, u)$, one has 
\beqn
E({\mf u}) = \langle \omega^K, [{\mf u}] \rangle + \| \ov\partial_A u \|_{L^2(\Sigma)}^2 + \| * F_A + \mu(u) \|_{L^2(\Sigma)}^2. 
\eeqn
Here $\omega^K\in H_K^2(V; {\mb R})$ is the equivariant class represented by the equivariant 2-form $\omega - \mu$ (see\cite[Proposition 3.1]{Cieliebak_Gaio_Salamon_2000} and \cite[Lemma 14]{Mundet_2003}). 

\begin{rem}
An important feature of the symplectic vortex equation in the toric setting is that no bubbling happens as the space $V$ is symplectically aspherical. In general, energy concentration could cause bubbling of holomorphic spheres as shown in \cite{Mundet_thesis, Mundet_2003, Cieliebak_Gaio_Mundet_Salamon_2002, Ott_compactness}.
\end{rem}

One can introduce Hamiltonian perturbations. Given a 1-form
\beqn
{\mc H} \in \Omega^1(\Sigma, C^\infty(V)^K)
\eeqn
with coefficients in the space of $K$-invariant smooth functions on $V$, we can define a family of Hamiltonian vector fields
\beqn
X_{\mc H} \in \Gamma( \Sigma \times V, \pi_{\Sigma}^* T^* \Sigma \otimes TV)
\eeqn
which is $K$-invariant, where $\pi_{\Sigma}: \Sigma \times V \to \Sigma$ is the projection to the first factor. Hence for any principal $K$-bundle $\pi_P: P \to \Sigma$, the vector field $X_{\mc H}$ induces a section on the total space  of the vector bundle $\pi_{P(V)}: P(V)\to \Sigma$
\beqn
X_{\mc H} \in \Gamma( P(V), \pi_{P(V)}^* T^* \Sigma \otimes P(TV)),
\eeqn
where $P(TV) := P \times_K u^* TV$.
The perturbed symplectic vortex equation is 
\begin{align}
&\ \ov\partial_{A, {\mc H}} u = 0,\ &\ * F_A + \mu(u) = 0.
\end{align}
where 
\beqn
\ov\partial_{A, {\mc H}} u = (d_A u)^{0,1} + (X_{\mc H}(u))^{0,1}.
\eeqn
For our applications, ${\mc H}$ is obtained by extending the pullback of Hamiltonian connections in $\Omega^1(\Sigma, C^{\infty}(X)) = \Omega^1(\Sigma, C^{\infty}(\mu^{-1}(0) / K))$.

\subsubsection{Compactness}

Although in aspherical targets vortices cannot bubble off holomorphic spheres, in general holomorphic curves can bubble off. It is the case when one considers Lagrangian boundary conditions. Let $L \subset V$ be a $K$-invariant Lagrangian submanifold. One can impose the Lagrangian boundary condition for gauged maps ${\mf u} = (P, A, u)$ from $\Sigma$ to $V$ with $u|_{\partial \Sigma} \subset P(L)$. Given a sequence of solutions ${\mf u}_i$ to the vortex equation on $\Sigma$ subject to the Lagrangian boundary condition, if $u_i$ has uniformly bounded image and $E({\mf u}_i)$ is uniformly bounded, the energy density could blow up near a boundary point. The boundedness of the images of $u_i$ implies that the curvatures $F_{A_i}$ do not blowup. Moreover, if one scales by the rate of energy concentration, the sequence of connections $A_i$ converge subsequentially (up to gauge transformation) to a flat connection. All Hamiltonian perturbations and variations of almost complex structures will also be scaled off. Hence a subsequece can bubble off a (stable) holomorphic disk in $V$ with boundary in $L$ with respect to a fixed almost complex structure. See details in \cite{Wang_Xu}.

\section{Algebraic preliminaries}\label{sec:alg-prelim}

\subsection{Novikov rings}

We set up the notations for our coefficient rings. In this paper, $R$ always denotes a commutative ring with a unit, hence comes with a canonical ring map 
\beqn
{\mb Z} \to R.
\eeqn
Let $\Lambda = \Lambda_R$ be the (upward) {\bf Novikov ring}
\beqn
\Lambda_R = \Big\{ \sum_{i=1}^\infty a_i T^{g_i}\ |\ g_i \in {\mb R},\ a_i \in R,\  \lim_{i \to \infty} g_i = +\infty \Big\}
\eeqn
The {\bf valuation} on ${\mf v}: \Lambda \to {\mb R} \cup \{+\infty\}$ is defined by 
\beqn
{\mf v}\left( \sum_{i=1}^\infty a_i T^{g_i} \right) = \inf \big\{ g_i\ |\ a_i \neq 0 \big\}\ \ \ {\rm and}\ \ \ {\mf v}(0) = +\infty.
\eeqn
We will also need the following version
\beqn
\Lambda_{0, R} = \Big\{ \sum_{i=1}^\infty a_i T^{g_i}\ |\ g_i \in {\mb R}_{\geq 0},\ a_i \in R,\  \lim_{i \to \infty} g_i = +\infty \Big\},
\eeqn
which also comes with a valuation by restricting the above valuation. When $R$ is a field, $\Lambda_R$ is also a field, and it is the field of fraction of $\Lambda_{0, R}$.

In many cases we can restrict to a Novikov ring of series $\sum a_i T^{g_i}$ where $g_i$ are restricted to a finitely generated additive group $\Gamma \subsetneq {\mb R}$. In this paper $\Gamma$ is fixed and actually determined by the GIT presentation of a toric manifold. Indeed, the discrete monoid $\Gamma$ associated with the toric manifold $X(\Sigma)$ is defined to be the image of effective $1$-cycles in ${\mb R}$ defined from pairing with the cohomology class represented by the symplectic form, see Subsection \ref{subsection_toric}. Denote 
\beqn
\Lambda_R^\Gamma:= \Big\{ \sum_{i=1}^\infty a_i T^{g_i}\in \Lambda_R \ |\ g_i \in \Gamma \Big\}
\eeqn
and 
\beqn
\Lambda_{0, R}^\Gamma:= \Lambda_{0, R} \cap \Lambda_R^\Gamma.
\eeqn
However $\Lambda_R^\Gamma$ does not enjoy certain algebraic properties of $\Lambda_R$. For example, when $R = {\mb K}$ is an algebraically closed field, $\Lambda_{\mb K}$ is algebraically closed but $\Lambda_{\mb K}^\Gamma$ is not.

\subsubsection{Modules and algebras over Novikov rings}

\begin{defn}\label{defn_nonarchimedean_module}
A {\bf non-Archimedean normed free module} over $\Lambda_R$ is a pair $(C, \ell)$ where $C$ is a free $\Lambda_R$-module endowed with a function $\ell: C \to {\mb R} \cup \{-\infty\}$ satisfying 
    \begin{enumerate}
        \item {\bf (Nondegeneracy)} $\ell(x) = -\infty$ if and only if $x = 0$.
        \item {\bf (Homogeneity)} For all $\lambda \in \Lambda_R$ and $x \in C$, $\ell(\lambda x) = \ell(x) - {\mf v}(\lambda)$.

        \item {\bf (Subadditivity)} For $x, y \in C$, $\ell(x +y ) \leq \max \{ \ell(x), \ell(y)\}$; if $\ell(x) \neq \ell(y)$, then $\ell(x + y) = \max \{ \ell(x), \ell(y)\}$.\footnote{The second part is a consequence of the first, see \cite[Proposition 2.3]{Usher_Zhang_2016}.}
        \end{enumerate}
\end{defn}

Now suppose ${\mb K}$ is an $R$-module which is also a field. Then one can extend the function $\ell$ to $C \otimes_{\Lambda_R} \Lambda_{\mb K}$ via {\bf (Homogeneity)}. Then one obtains a non-Archimedean normed vector space in the sense of \cite[Definition 2.2]{Usher_Zhang_2016} (except that the coefficient field was $\Lambda_{\mb K}^\Gamma$ rather than $\Lambda_{\mb K}$).

We also need to consider multiplicative structures compatible with the non-Archimedean norm. 

\begin{defn}
A {\bf non-Archimedean normed algebra} over $\Lambda_R$ is a non-Archimedean normed free module $(C, \ell)$ together with a commutative $\Lambda_R$-algebra structure (with multiplication denoted by $x \ast y$) satisfying
\begin{itemize}
    \item {\bf (Triangle inequality)} For all $x, y \in C$, 
    \beqn
    \ell( x \ast y) \leq \ell(x) + \ell(y).
    \eeqn
\end{itemize}
\end{defn}

\subsubsection{Specific coefficients and mod $p$ reductions}

In this paper we need to use certain non-traditional coefficient ring or fields. Here we briefly summarize them and set up the notations. First, let $\ov{\mb Q}$ be the algebraic closure of ${\mb Q}$, which is viewed as a subfield of ${\mb C}$. Inside $\ov{\mb Q}$ there is the subring of algebraic integers $\ov{\mb Z}$, which is the set of algebraic numbers which are solutions to monic polynomials with integer coefficients. Further, in characteristic $p$ (where $p$ is an odd prime), let ${\mb F}_p\cong {\mb Z}/ p {\mb Z}$ be the smallest field with characteristic $p$. Let $\ov{\mb F}_p$ be the algebraic closure of ${\mb F}_p$, which is only well-defined up to isomorphism of field extensions.

Notice that the notion of non-Archimedean normed algebras can be transferred between different coefficient rings via tensor products. A crucial feature of the geometric construction of this paper is that, as long as one has a counting theory over ${\mb Z}$ (or $\ov{\mb Z}$), it automatically induces a theory over any ring $R$ (or $\ov{\mb Z}$-algebra). In particular, one needs to perform the ``mod $p$ reduction'' which is roughly associated to the ring map ${\mb Z} \to {\mb F}_p$. In our situation, one needs the corresponding extension to the algebraic closure of ${\mb F}_p$.

\begin{lemma}
For each prime $p$, there exists a unital ring map 
\beqn
\ov\pi_p: \ov{\mb Z} \to \ov{\mb F}_p.
\eeqn
\end{lemma}

\begin{proof}
(Following the mathoverflow post \cite{Fp_closure}) $p\in {\mb Z} \subset \ov{\mb Z}$ is not an invertible element. Hence there exists a maximal ideal ${\mf m} \subset \ov{\mb Z}$ containing $p$. Consider the quotient field $\ov{\mb Z}/ {\mf m}$, which is an extension of ${\mb F}_p$. We prove that it is an algebraic closure of ${\mb F}_p$. First, each element $x\in \ov{\mb Z} / {\mf m}$ has a lift $\tilde x\in \ov{\mb Z}$ which is the solution to an algebraic equation with integer coefficients. Hence $x$ is an algebraic element over ${\mb F}_p$. Therefore $\ov{\mb Z}/{\mf m}$ is algebraic over ${\mb F}_p$. Second, for any monic polynomial $f$ with ${\mb F}_p$-coefficients, one can find a monic integral lift $\tilde f$ and all roots of $\tilde f$ are in $\ov{\mb Z}$. Hence $f$ has roots in the field $\ov{\mb Z}/ {\mf m}$. Therefore $\ov{\mb Z}/ {\mf m}$ is algebraically closed. Hence $\ov{\mb Z}/ {\mf m}$ is an algebraic closure of ${\mb F}_p$. Therefore, $\ov{\mb Z}/ {\mf m} \cong \ov{\mb F}_p$ as fields. Precomposing with $\ov{\mb Z} \to \ov{\mb Z}/ {\mf m}$ one obtains the desired ring map.
\end{proof}

\begin{rem}
The map $\ov\pi_p$ is not unique. But we fix one for each prime $p$. 
\end{rem}

The ``mod $p$ reduction'' also makes sense for any rational number as long as $p$ is greater than the denominator. We need the following lemma to extend this simple fact to the algebraic closure.

\begin{lemma}\label{lemma46}
For each $y \in \ov{\mb Q}$, there exists $m \in {\mb Z}\setminus \{ 0\}$ such that $my \in \ov{\mb Z}$. 
\end{lemma}

\begin{proof}
If $y\in \ov{\mb Q}$ is an algebraic number, then it is a root of a polynomial with integer coefficients $P = a_k u^k + \cdots + a_1 u + a_0$ with $a_k \neq 0$. Then 
\beqn
0 = a_k^{k-1} P(y) = (a_k y)^k + a_{k-1} (a_k y )^{k-1} + a_{k-2} a_k (a_k y)^{k-2} + \cdots + a_1 a_k^{k-2} (a_k u) + a_0 a_k^{k-1}.
\eeqn
So $a_k y$ is a root of a monic polynomial with integer coefficients, hence is in $\ov{\mb Z}$.
\end{proof}

\subsection{Semisimple algebras over Novikov fields}

In this paper we use a more restrictive notion of semisimplicity of algebras over Novikov fields. 

\begin{defn}\label{defn_semisimple}
Let ${\mb F}$ be a field. A unital ${\mb F}$-algebra $(A, \ast)$ is called {\bf semisimple} if it splits as a direct sum of rings
\beqn
A = F_1 \oplus \cdots \oplus F_k
\eeqn
where $F_i \cong {\mb F}$ as a ring. Each summand $F_i$ is called an {\bf idempotent summand} of $A$ and the splitting is called the {\bf idempotent splitting}.
\end{defn}

\begin{rem}
In many papers such as \cite{Entov_Polterovich_1, Shelukhin_2022, FOOO_toric_1}, the meaning of semisimplicity is more general: for example, each summand $F_i$ is allowed to be a finite extension of the field ${\mb F}$. The number of idempotent summands also depends on the choice of the field. In our situation, one can achieve the above stronger semisimplicity of a version of quantum cohomology algebra by turning on bulk deformations and taking a sufficiently large field.
\end{rem}

Suppose $A$ is semisimple. Then for each idempotent summand $F_i$, there is a unique generator $e_i \in F_i$ such that $e_i \ast e_i = e_i$. We call $e_i$ the {\bf idempotent generator}. Then $(e_1, \ldots, e_k)$ is a basis of $A$. Given any element $\alpha = \lambda_1 e_1 + \cdots + \lambda_k e_k$, one can see that the linear map 
\beqn
\alpha \ast: A \to A
\eeqn
has eigenspace decomposition $F_1 \oplus \cdots \oplus F_k$ with eigenvalues $\lambda_1, \ldots, \lambda_k$. The following statement shows that the converse also holds under additional assumptions.

\begin{lemma}\label{semisimple_criterion}
Let $A$ be a $k$-dimensional commutative unital ${\mb F}$-algebra and $\alpha \in A$. Suppose $\alpha\ast:A \to A$ has $k$ distinct \emph{nonzero} eigenvalues $\lambda_1, \ldots, \lambda_k$. Then $A$ is semisimple.
\end{lemma}

\begin{proof}
Let $(\varepsilon_1, \ldots, \varepsilon_k)$ be an eigen-basis of $\alpha \ast$ and write $\alpha = \mu_1 \varepsilon_1 + \cdots + \mu_k \varepsilon_k$. Then we see 
\beqn
\alpha \ast ( \varepsilon_i \ast \varepsilon_j ) = \lambda_i \varepsilon_i \ast \varepsilon_j = \lambda_j \varepsilon_i \ast \varepsilon_j. 
\eeqn
As $\lambda_i$ are all distinct, one has $\varepsilon_i \ast \varepsilon_j = 0$ whenever $i \neq j$. Then one obtains
\beqn
\alpha \ast \varepsilon_i = \mu_i \varepsilon_i \ast\varepsilon_i = \lambda_i \varepsilon_i.
\eeqn
As $\lambda_i \neq 0$, one can see $\mu_i \neq 0$. Define $e_i = \lambda_i^{-1} \mu_i \varepsilon_i$. Then 
\beqn
e_i \ast e_i = (\lambda_i^{-1} \mu_i)^2 \varepsilon_i \ast \varepsilon_i = \lambda_i^{-1} \mu_i \varepsilon_i = e_i.
\eeqn
Hence $A$ is semisimple.
\end{proof}

\subsubsection{Semi-simplicity and different characteristics}

Here we prove a useful algebraic fact which allows us to derive semi-simplicity in finite characteristics from semi-simplicity in characteristic zero. We set up the problem as follows. Let $(A, \ell)$ be a non-Archimedean normed (free) algebra over the Novikov ring $\Lambda_{\ov{\mb Z}}$. Denote 
\beqn
A_{(0)}:= A \otimes_{\Lambda_{\ov{\mb Z}}} \Lambda_{\ov{\mb Q}}
\eeqn
and for each prime $p$
\beqn
A_{(p)}:= A \otimes_{\Lambda_{\ov{\mb Z}}} \Lambda_{\ov{\mb F}_p}.
\eeqn
Denote the induced valuations by 
\begin{align*}
    &\ \ell_0: A_{(0)} \to {\mb R} \cup \{-\infty\},\ &\ \ell_p: A_{(p)} \to {\mb R} \cup \{ -\infty \}.
\end{align*}
Moreover, let ${\mc U} \in A$ be a distinguished nonzero element (which will be the first Chern class in quantum homology in our later discussions), and let ${\mc U}_{(0)} \in A_{(0)}$, ${\mc U}_{(p)} \in A_{(p)}$ be the corresponding induced element. They induce linear operators 
\beqn
E_{(m)}: A_{(m)} \to A_{(m)},\ x \mapsto {\mc U}_{(m)} \ast x,\ m = 0, p.
\eeqn

\begin{thm}\label{thm_semisimple_finite_characteristic}
Suppose $A_{(0)}$ is semisimple over $\Lambda_{\ov{\mb Q}}$ and all eigenvalues of $E_{(0)}$ are nonzero and distinct. Then there exist $p_0>0$ and $C>0$ such that for all prime $p \geq p_0$, the following conditions hold.
\begin{enumerate}
    \item $A_{(p)}$ is semisimple over $\Lambda_{\ov{\mb F}_p}$.

    \item If $e_{1, p}, \ldots, e_{m, p}$ are idempotent generators of $A_{(p)}$, then 
    \beqn
    \ell(e_{l, p}) \leq C.
    \eeqn
\end{enumerate}
\end{thm}

\begin{proof}[Proof of Theorem \ref{thm_semisimple_finite_characteristic} (1)]
Consider the operator $E: A \to A$ defined by the multiplication with ${\mc U}$ and its characteristic polynomial $f_E$. Notice that $\Lambda_{\ov{\mb Q}}$ is the field of fraction of $\Lambda_{\ov{\mb Z}}$. Hence $f_E$ has $m$ distinct roots in $\Lambda_{\ov{\mb Q}}$ and so the discriminant of $f_E$, denoted by $D(f_E)\in \Lambda_{\ov{\mb Z}}$, is nonzero. Hence for sufficiently large prime $p$, the discriminant of $f_{E_{(p)}}$, which is the mod $p$ reduction of $D(f_E)$, is nonzero. It follows that $E_{(p)}$ also has $m$ distinct eigenvalues. Moreover, as all eigenvalues of $E_{(0)}$ are nonzero, $f_E(0) \neq 0$. Hence $f_{E_{(p)}}(0) \neq 0$ when $p$ is sufficiently large. Hence $E_{(p)}$ is invertible and has no zero eigenvalue. By Lemma \ref{semisimple_criterion}, $A_{(p)}$ is semisimple for sufficiently large $p$. 
\end{proof}

\subsubsection{Proof of Theorem \ref{thm_semisimple_finite_characteristic} (2)}

To prove the quantitative statement of Theorem \ref{thm_semisimple_finite_characteristic}, we introduce the notion of truncation. First, given an element 
\beqn
\lambda = \sum_{i =1}^\infty a_i T^{g_i} \in \Lambda_{\ov{\mb Q}}, 
\eeqn
and $Z \in {\mb R}$, define the $Z$-truncation of $\lambda$ to be the element 
\beqn
\lambda^Z:= \sum_{g_i \leq Z} a_i T^{g_i},
\eeqn
which has only finitely many terms. Then it follows easily 
\beq\label{truncation_1}
{\mf v} ( \lambda - \lambda^Z) \geq Z.
\eeq
For an element in a module over $\Lambda_{\ov{\mb Q}}$, its truncations are not canonically defined. We fix, throughout the proof, a basis ${\mf x}_1, \ldots, {\mf x}_m$, of the $\Lambda_{\ov{\mb Z}}$-module $A$. Without loss of generality, we can choose the basis such that 
\beqn
\ell( {\mf x}_1) = \cdots = \ell({\mf x}_m) = 0.
\eeqn
By abuse of notations, denote the induced basis of $A_{(0)}$ and $A_{(p)}$ still by ${\mf x}_1, \ldots, {\mf x}_m$. Then for each $\alpha \in A_{(0)}$, we can write 
\beqn
\alpha = \sum_{j=1}^m \alpha_j {\mf x}_j
\eeqn
where $\alpha_j \in \Lambda_{\ov{\mb Q}}$. Then define the $Z$-truncation
\beqn
\alpha^Z = \sum_{j=1}^B \alpha_j^Z {\mf x}_j.
\eeqn
Then by \eqref{truncation_1} we have the estimate
\begin{multline}\label{truncation_2}
\ell_0 ( \alpha - \alpha^Z) = \ell_0 \left( \sum_{l=1}^m (\alpha_l - \alpha_l^Z) {\mf x}_l \right) \leq \max_{1 \leq l \leq m} \ell_0 \left( (\alpha_l - \alpha_l^Z) {\mf x}_l \right) \\
= \max_{1\leq l \leq m} \left( \ell_0 ({\mf x}_l) - {\mf v}(\alpha_l - \alpha_l^Z) \right) \leq \max_{1 \leq j \leq m} \ell ({\mf x}_j) - Z = - Z.
\end{multline}

{\bf Running convention.} Within this proof, $Z$ is a large real number which can be fixed from the beginning. The lower bound of $p_0$ which is valid for the statement of Theorem \ref{thm_semisimple_finite_characteristic} depends on the choice of $Z$. The letter $C>0$ denotes a real number which is independent of $Z$ and $p\geq p_0$ but whose value is allowed to change from line to line. 

\begin{lemma}
Suppose $e_{1,(0)}, \dots, e_{m, (0)}$ constitute the idempotent generators of $A_{(0)}$, with valuations $\lambda_{1, (0)}, \dots, \lambda_{m, (0)}$. Then for $Z$ sufficiently large, $e_{1,(0)}^Z, \dots, e_{m, (0)}^Z$ form a basis of $A_{(0)}$, and $\lambda_{1, (0)}^Z, \dots, \lambda_{m, (0)}^Z$ are all nonzero and distinct. 
\end{lemma}

\begin{proof}
With respect to the basis $({\mf x}_{1, (0)}, \ldots, {\mf x}_{m, (0)})$ of $A_{(0)}$, we identify $e_{l, (0)}$ with its coordinate vector in $(\Lambda_{\ov{\mb Q}})^m$. Then the $m \times m$ matrix with columns $e_{l, (0)}$ is invertible with a nonzero determinant. Then when $Z$ is sufficiently large, the corresponding determinant with $e_{l, (0)}$ replaced by $e_{l, (0)}^Z$ is also nonzero. On the other hand, as all $\lambda_{l,(0)}$ are nonzero, $\lambda_{l, (0)}^Z\neq 0$ when $Z$ is large. 
\end{proof}

We would like to construct, for large primes $p$, eigenvectors and eigenvalues of ${\mb E}_{(p)}$ over the field $\Lambda_{\ov{\mb F}_p}$. The basic idea is to take some truncation $e_{l, (0)}^Z$ of the idempotent generators and their mod $p$ reductions as an appropriate eigenbasis and then to apply certain corrections. 

By Lemma \ref{lemma46}, for each $Z\in {\mb R}$, there exists $m^Z \in {\mb Z}$ such that 
\begin{align*}
&\ m^Z \lambda_{l, (0)}^Z \in \Lambda_{\ov{\mb Z}},\ &  m^Z e_{l, (0)}^Z \in A_{\ov{\mb Z}}.
\end{align*}
This allows us to define the ``mod $p$ reduction'' of $\lambda_{l, (0)}^Z$ and $e_{l, (0)}^Z$ as follows. Fixing $m^Z$, by choosing a sufficiently large $p$ so that it cannot divide $m^Z$, the quantity $m^Z$ has a nonzero reduction $[m^Z]_p \in {\mb F}_p$. Moreover, $m^Z \lambda_{l, (0)}^Z$ has a mod $p$ reduction $[m^Z \lambda_{l, (0)}^Z]_p \in \Lambda_{\ov{\mb F}_p}$ and $m^Z e_{l, (0)}^Z$ has a mod $p$ reduction $[m^Z e_{l, (0)}^Z]_p \in A_{(p)}$ (defined via the integral basis ${\mf x}_1, \ldots, {\mf x}_m$). Then define 
\begin{align*}
&\ \lambda_{l, (p)}^Z:= [m^Z]_p^{-1} [m^Z \lambda_{l, (0)}^Z]_p,\ &\ e_{l, (p)}^Z:= [m^Z]_p^{-1} [m^Z e_{l, (0)}^Z]_p.
\end{align*}

\begin{lemma}
There exists $C>0$ such that for any sufficiently large $Z$, upon choosing $m^Z$ as above, there exists $p^Z > 0$ such that whenever $p \geq p^Z$, $e_{l, (p)}^Z$ is a basis of $A_{(p)}$ and all $\lambda_{l, (p)}^Z$ are nonzero and distinct. Moreover, for some constant $C>0$ one has
\begin{align*}
&\ \ell_p ( e_{l, (p)}^Z) \geq  - C,\ &\ {\mf v}(\lambda_{l, (p)}^Z) \leq C.
\end{align*}
Moreover, for all $k \neq l$
\beqn
{\mf v}( \lambda_{l, (p)}^Z - \lambda_{k, (p)}^Z) \leq C.
\eeqn
\end{lemma}

\begin{proof}
Straightforward.
\end{proof}

\begin{prop}
There exists $C>0$ such that given any sufficiently large $Z$, for all sufficiently large prime $p$, there exist eigenvectors $\varepsilon_{l, (p)}$ of $E_{(p)}$ with corresponding distinct eigenvalues $\lambda_{l, (p)} \in \Lambda_{\ov{\mb F}_p}$ such that
\beqn
\ell_p ( e_{l, (p)}^Z - \varepsilon_{l, (p)} ) \leq  -Z + C
\eeqn
and 
\beqn
{\mf v} ( \lambda_{l,(p)}^Z - \lambda_{l, (p)} ) \geq Z - C.
\eeqn
\end{prop}

\begin{proof}
In $A_{(0)}$, one has
\beqn
(m^Z)^{-1} E_{(0)} ( m^Z e_{l, (0)}) = ((m^Z)^{-1} \lambda_{l, (0)})(m^Z e_{l, (0)}).
\eeqn
Using \eqref{truncation_2}, it follows that
\beq\label{approx}
\begin{split}
&\ \ell_0 \Big( (m^Z)^{-1} E_{(0)} (m^Z e_{l, (0)}^Z) - ((m^Z)^{-1} \lambda_{l, (0)}^Z) (m^Z e_{l, (0)}^Z) \Big)\\
= &\ \ell_0 \Big( E_{(0)} (e_{l, (0)}^Z) - \lambda_{l, (0)}^Z e_{l, (0)}^Z \Big)\\
= &\ \ell_0 \Big( E_{(0)} \big(e_{l, (0)}^Z - e_{l, (0)} \big) +  ( \lambda_{l, (0)} - \lambda_{l, (0)}^Z) e_{l, (0)} + \lambda_{l, (0)}^Z (e_{l, (0)} - e_{l, (0)}^Z) \Big) \\
\leq &\ \max \Big\{ \ell_0 ( c_{1, (0)} \ast (e_{l, (0)}^Z - e_{l, (0)})),\ \ell_0 ( (\lambda_{l, (0)} - \lambda_{l, (0)}^Z) e_{l, (0)} ),\ \ell_0 (\lambda_{l, (0)}^Z (e_{l, (0)} - e_{l, (0)}^Z) ) \Big\}\\
\leq &\ \max \Big\{ \ell_0 (c_{1, (0)}) + C - Z,\ \ell_0(e_{l, (0)}) - Z,\ - {\mf v}(\lambda_{l, (0)}) + C - Z \Big\}\\
 \leq &\ C- Z.
\end{split}
\eeq

Now take a sufficiently large $p$. Notice that the left hand side of \eqref{approx} is the valuation of an integral element, hence descends to $\ov{\mb F}_p$. Then \eqref{approx} implies 
\beq\label{approximate_eigenvalue}
\ell_p \left( E_{(p)} (e_{l, (p)}^Z) - \lambda_{l, (p)}^Z e_{l, (p)}^Z \right) \leq C-Z.
\eeq
We would like to correct $e_{l, (p)}^Z$ by adding higher order terms. 

With respect to the basis $(e_{1, (p)}^Z, \ldots, e_{m, (p)}^Z)$, let the matrix of $E_{(p)}$ be $T_{(p)}$. For a matrix $S = (S_{ij})$ with entries in $\Lambda_{\mb K}$, we write 
\beqn
{\mf v}(S) = \min_{i, j} {\mf v}(S_{ij}) \in {\mb R}.
\eeqn

\begin{lemma}\label{lemma_matrix_valuation}
One has 
\beqn
{\mf v} \left( T_{(p)} - \left[ \begin{array}{cccc} \lambda_{1, (p)}^Z & 0 & \cdots & 0 \\
0 & \lambda_{2, (p)}^Z & \cdots & 0 \\
\vdots & \vdots & \ddots & \vdots \\
0 & 0 & \cdots & \lambda_{m, (p)}^Z \end{array} \right] \right) \geq Z - C.
\eeqn
\end{lemma}

\begin{proof}
We can write 
\beqn
E_{(p)}(e_{l, (p)}^Z) - \lambda_{l, (p)}^Z e_{l, (p)}^Z = \sum_{l=1}^m a_l e_{l, (p)}^Z.
\eeqn
By \eqref{approximate_eigenvalue} and the subadditivity property of the function $\ell$ (see Definition \ref{defn_nonarchimedean_module}), one has ${\mf v}(a_l) \geq Z- C$ for some appropriate $C$.
\end{proof}

For each $l$, consider the following equation for 
\begin{align*}
&\ x = \sum_{k \neq l} x_k e_{k, (p)}^Z,\ &\ \delta \in \Lambda_{\ov{\mb F}_p}
\end{align*}
which is 
\beqn
E_{(p)} \big( e_{l, (p)}^Z + x \big) = \big( \lambda_{l,(p)}^Z + \delta \big) \big(  e_{l, (p)}^Z + x \big)
\eeqn
which is equivalent to 
\beqn
E_{(p)} \left( \sum_{k \neq l} x_k e_{k, (p)}^Z \right) - \sum_{k\neq l} \lambda_{l, (p)}^Z x_k e_{k, (p)}^Z - \delta e_{l, (p)}^Z = \delta \sum_{k \neq l} x_k e_{k, (p)}^Z + \rho_l.
\eeqn
Here $\rho_l$ is the error term with ${\mf v}(\rho_l) \leq -Z + C$ by \eqref{approximate_eigenvalue}. To simplify notations, assume $l = 1$. Then using the basis $e_{1, (p)}^Z, \ldots, e_{m, (p)}^Z$, this equation is equivalent to the linear system 
\beqn
\left( \left[ \begin{array}{cccc} 1 & 0 & \cdots  & 0 \\
                                   & & T_{(p)}' & \end{array}\right] - \left[ \begin{array}{cccc} 0 & 0 & \cdots & 0 \\
0 & \lambda_{1, (p)}^Z & \cdots & 0 \\
\vdots & \vdots & \ddots & \vdots \\
0 & 0 & \cdots & \lambda_{1, (p)}^Z \end{array} \right] \right) \left[ \begin{array}{c} \delta \\ x_2 \\ \vdots \\ x_m \end{array} \right] = Q( \delta, x_2, \ldots, x_m) + \rho_1.
                                   \eeqn
Here the left hand side is linear and $Q$ is quadratic. Let the matrix on the left hand side be $F_1$. Lemma \ref{lemma_matrix_valuation} implies that $F_1$ is invertible with 
\beqn
{\mf v}(F_1) \leq C
\eeqn
where $C$ is independent of $Z$ and $p$. Then one can use an iteration argument to solve the equation term by term. The correction term has valuation at least $Z - C$ for some constant $C$. So the theorem follows. 
\end{proof}

We continue the proof of (2) of Theorem \ref{thm_semisimple_finite_characteristic}. By the proof of Lemma \ref{semisimple_criterion}, each idempotent generator of $A_{(p)}$ is a multiple of $\varepsilon_{l, (p)}$. Indeed, if 
\beqn
\varepsilon_{l, (p)} \ast \varepsilon_{l, (p)} = \mu_l \varepsilon_{l, (p)}
\eeqn
then the corresponding idempotent generator is
\beqn
e_{l, (p)} = \mu_l^{-1} \varepsilon_{l, (p)}.
\eeqn
So we need to estimate the valuation of $\mu_l$. In characteristic zero one has 
\beqn
e_{l, (0)} \ast e_{l, (0)} = e_{l, (0)}.
\eeqn
Taking truncation at $Z$ one has 
\beqn
\ell_0 \big( e_{l, (0)}^Z \ast e_{l, (0)}^Z - e_{l, (0)}^Z \big) \leq C-Z.
\eeqn
Taking mod $p$ reduction, one obtains
\beqn
\ell_p \big( e_{l, (p)}^Z \ast e_{l, (p)}^Z - e_{l, (p)}^Z \big) \leq C-Z.
\eeqn
Then 
\begin{multline*}
\ell_p \big( \mu_l \varepsilon_{l, (p)} - e_{l, (p)}^Z \big) = \ell_p \big( \varepsilon_{l, (p)} \ast \varepsilon_{l, (p)} - e_{l, (p)}^Z \ast e_{l, (p)}^Z + e_{l, (p)}^Z \ast e_{l, (p)}^Z - e_{l, (p)}^Z \big) \\
\leq \max \Big\{ \ell_p ( \varepsilon_{l, (p)} - e_{l, (p)}^Z ) + \ell_p (\varepsilon_{l, (p)} + e_{l, (p)}^Z), C-Z \Big\} \leq C-Z.
\end{multline*}
As we have $\ell_p( e_{l, (p)}^Z ) \geq -C$, it follows that $\ell_p( \mu_l \varepsilon_{l, (p)}) = \ell_p (e_{l, (p)}^Z) = \ell_p( \varepsilon_{l, (p)})$. Hence ${\mf v}(\mu_l) = 0$ and hence 
\beqn
\ell_p( e_{l, (p)}) = \ell_p ( \mu_l^{-1} \varepsilon_{l, (p)}) = \ell_p (\varepsilon_{l, (p)}) + {\mf v}(\mu_l) = \ell_p (e_{l,(p)}^Z) = \ell_0 (e_{l, (0)})
\eeqn
which is independent of $p$. This finishes the proof of Theorem \ref{thm_semisimple_finite_characteristic}.

\subsection{Floer--Novikov complexes}

Let $\Gamma \subsetneq {\mb R}$ be a proper additive subgroup. 

\begin{defn}[Floer--Novikov complex] (cf. \cite[Definition 1.1]{Usher_2008})\label{defn_fn} A {\bf ${\mb Z}_2$-graded filtered Floer--Novikov package} over a commutative unital ring $R$ consists of data 
\beqn
{\mf c} = \Big(  P, {\mc A}, gr, n_R \Big)
\eeqn
where 
\begin{enumerate}
    \item $P$ is a $\Gamma$-torsor with $P/ \Gamma$ finite.

    \item ${\mc A}: P \to {\mb R}$ is the ``action functional'' and $gr: P \to {\mb Z}_2$ is the ``grading.''

    \item For $p \in P$ and $g \in \Gamma$, one has 
    \begin{align*}
        &\ {\mc A}(g p) = {\mc A}(p) - g,\ &\ gr(g p) = gr (p)
    \end{align*}

    \item $n_R: P \times P \to R$ is a function such that
    \begin{itemize}
    \item $n_R (p, q) \neq 0 \Longrightarrow gr(p) = gr(q) + 1,\ {\mc A}(p) > {\mc A}(q)$;
    \item for all $p \in P$ and $C\in {\mb R}$, the set
    \beqn
    \{ q \in P\ |\ n_R (p, q) \neq 0,\ {\mc A}(q) \geq C \}
    \eeqn
    is finite;
    \item for any $g \in \Gamma$, we have $n_R (gp, gq) = n_R (p,q)$;
    \item the $\Lambda_R$-linear map $\partial$ defined in \eqref{eqn:boundary} satisfies $\partial^2 = 0$.
    \end{itemize}
\end{enumerate}
\end{defn}

Given a Floer--Novikov package one can construct the associated Floer chain complex. Define 
\beqn
CF_\bullet ({\mf c}) = \Big\{ \sum_{p \in P} a_p p\ |\ a_p \in R,\ \forall C\in {\mb R}, \# \{ p \in P \ |\ a_p \neq 0,\ {\mc A}(p)\geq C \} < \infty \Big\}
\eeqn
which is ${\mb Z}_2$-graded. The $\Lambda_R^\Gamma$-module structure is induced from the $\Gamma$-action on $P$. Define 
\beqn
\partial: CF_\bullet({\mf c}) \to CF_{\bullet -1} ({\mf c})
\eeqn
by 
\beq\label{eqn:boundary}
\partial \left( \sum_{p \in P} a_p p \right) = \sum_{q\in P} \left( \sum_{ p \in P} a_p n_R (p, q)\right) q.
\eeq
We also define the function 
\beq\label{valuation_function}
\ell: CF_\bullet({\mf c}) \to {\mb R}\cup \{-\infty\},\ \ell\left( \sum_{p \in P} a_p p \right) = \sup \big\{ {\mc A}(p)\ |\ a_p \neq 0 \Big\}. 
\eeq

Given a Floer--Novikov package ${\mf c}$ over $R$, if $\iota: R \to \wt R$ is a ring map, then one can extend ${\mf c}$ to a Floer--Novikov package ${\mf c}\otimes_R \wt R$ by simply defining $n_{\wt R}:= \iota \circ n: P \times P \to \wt R$

\begin{prop}
If $R = {\mb K}$ is a field, the triple $(CF_\bullet({\mf c}), \partial, \ell)$ is a Floer-type complex over $\Lambda_{\mb K}^\Gamma$ in the sense of \cite[Definition 4.1]{Usher_Zhang_2016}.
\end{prop}

\begin{proof}
It follows directly from the definitions of Floer-type complexes. The proof serves rather as a brief clarification about this concept. First, for each $k \in {\mb Z}_2$, the pair $(CF_k({\mf c}), \ell|_{CF_k({\mf c})})$ is a non-Archimedean normed vector space over $\Lambda_{\mb K}^\Gamma$ (see \cite[Definition 2.2]{Usher_Zhang_2016}. In addition it is an orthogonalizable $\Lambda_{\mb K}^\Gamma$-space (see \cite[Definition 2.7]{Usher_Zhang_2016}. The last requirement for being a Floer-type complex is the inequality
\beqn
\ell(\partial(x)) \leq \ell(x)\ \forall x \in CF_\bullet({\mf c}),
\eeqn
which is a consequence of the property of the function $n_{\mb K}$ in the data ${\mf c}$. 
\end{proof}

\subsubsection{Spectral invariants}

Following Usher \cite{Usher_2008}, one can also define spectral invariants in an abstract way. First, define the ``energy filtration'' on the complex $CF_\bullet({\mf c})$: for each $\tau \in {\mb R}$, define 
\beqn
CF_\bullet^{\leq \tau}({\mf c}):= \left\{ \sum_{p \in P} a_p p \in CF_\bullet ({\mf c})\ |\ a_p \neq 0 \Longrightarrow {\mc A}(p) \leq \tau \right\}.
\eeqn
Then since the differential decreases the action, it is a subcomplex with homology 
\beqn
HF_\bullet^{\leq \tau}({\mf c}) 
\eeqn
and natural maps when $\tau \leq \kappa$
\beq\label{eqn:iota-map}
\iota^{\tau, \kappa}: HF_\bullet^{\leq \tau}({\mf c}) \to HF_\bullet^{\leq \kappa}({\mf c}).
\eeq
For $\alpha \in HF_\bullet ({\mf c})$, define 
\beqn
\rho (\alpha):= \inf \Big\{ \tau \in {\mb R}\ |\ \alpha \in {\rm Im} \left( \iota^\tau: HF_\bullet^{\leq \tau}({\mf c}) \to HF_\bullet ({\mf c}) \right) \Big\}\in {\mb R} \cup \{-\infty\}
\eeqn

\begin{thm}\cite[Theorem 1.3, 1.4]{Usher_2008}\label{Usher_thm}
Given a Floer--Novikov package ${\mf c}$ (over a Noetherian ring $R$) and $\alpha \in HF({\mf c}) \setminus \{0\}$, $\rho (\alpha) > -\infty$ and $\alpha \in {\rm Im} (\iota_{\rho(\alpha)})$. 
\end{thm}

\subsubsection{Boundary depth}

\begin{defn}\cite{Usher_boundary_depth}\label{defn_boundary_depth_1} 
Let ${\mf c}$ be a Floer--Novikov package and let $CF_\bullet^{\leq \lambda}({\mf c})$ be the associated filtered Floer--Novikov complex over $\Lambda_{\mb K}^\Gamma$. Then the {\bf boundary depth} of the filtered complex is the infimum of $\beta > 0$ such that for all $\lambda \in {\mb R}$
\beqn
CF^{\leq \lambda} ({\mf c}) \cap {\rm Im} \partial \subset \partial( CF^{\leq \lambda + \beta} ({\mf c}) ).
\eeqn
\end{defn}

\begin{thm}\cite[Theorem 1.3]{Usher_2008} Given a Floer--Novikov package ${\mf c}$, the boundary depth of the associated Floer--Novikov complex is finite.    
\end{thm}

\subsubsection{Quasiequivalence distance}

We rephrase the notion of quasiequivalences between Floer--Novikov complexes, which was originally introduced in \cite{Usher_Zhang_2016} for the more general situation of Floer-type complexes. 

\begin{defn}(cf. \cite[Definition 1.3]{Usher_Zhang_2016}) \label{defn_quasiequivalence}
Let $(CF_\bullet({\mf c}_i), \partial_i)$, $i = 1, 2$, be two Floer--Novikov complexes associated to Floer--Novikov packages ${\mf c}_i$ over a field ${\mb K}$. Let $\ell_i$ be the valuation function on the two complexes defined by \eqref{valuation_function}. Let $\delta \geq 0$. A {\bf $\delta$-quasiequivalence} between $CF_\bullet({\mf c}_1)$ and $CF_\bullet({\mf c}_2)$ is a quadruple $(\Phi, \Psi, K_C, K_D)$ where 
\begin{enumerate}
    \item $\Phi: CF_\bullet({\mf c}_1) \to CF_\bullet({\mf c}_2)$ and $\Psi: CF_\bullet({\mf c}_2) \to CF_\bullet({\mf c}_1)$ are chain maps with 
    \begin{align*}
        &\ \ell_2(\Phi(x_1)) \leq \ell_1(x_1) +\delta,\ &\ \ell_1(\Psi(x_2)) \leq \ell_2(x_2) + \delta
    \end{align*}
    for all $x_1 \in CF_\bullet({\mf c}_1)$ and $x_2 \in CF_\bullet({\mf c}_2)$.

    \item $K_i: CF_\bullet({\mf c}_i) \to CF_{\bullet+1}({\mf c}_i)$, $i = 1, 2$, obey the homotopy equations
    \begin{align*}
    &\ \Psi \circ \Phi - {\rm Id}_{CF_\bullet({\mf c}_1)} = \partial_1 K_1 + K_1 \partial_1,\ &\ \Phi \circ \Psi - {\rm Id}_{CF_\bullet({\mf c}_2)} = \partial_2 K_2 + K_2 \partial_2
    \end{align*}
    and for all $x_i \in CF_\bullet({\mf c}_i)$, $i = 1, 2$, one has 
    \beqn
    \ell_i(K_i(x_i)) \leq \ell_i(x_i) + 2\delta.
    \eeqn
    \end{enumerate}
The {\bf quasiequivalence distance} between $CF_\bullet({\mf c}_1)$ and $CF_\bullet({\mf c}_2)$, denoted by $d_Q(CF_\bullet({\mf c}_1), CF_\bullet({\mf c}_2))$, is the infimum of $\delta$ such that there exists a $\delta$-quasiequivalence between them.    
\end{defn}

\subsection{Persistence modules and stability of boundary depth}\label{subsection_persistence}

\begin{defn}\label{defn_persistence}
Let ${\mb K}$ be a field. 
\begin{enumerate}

\item A {\bf persistence module} ${\bm V}$ is a family of ${\mb K}$-vector spaces 
\beqn
{\bm V} = (V^s)_{s \in {\mb R}}
\eeqn
together with linear maps (called the {\bf structural maps} of ${\bm V}$)
\beqn
\iota^{s, t}:= \iota^{s, t}_{\bm V}: V^s \to V^t\ \forall s \leq t
\eeqn
such that $\iota^{s, s} = {\rm Id}_{V^s}$ for all $s$ and $\iota^{t, r} \circ \iota^{s, t} = \iota^{s, r}$ for all $s \leq t \leq r$. 

\item Let ${\bm V}$ be a persistence module and $\delta\in {\mb R}$. The $\delta$-shift of ${\bm V}$ is the persistence module ${\bm V}[\delta]$ with $V[\delta]^s = V^{s + \delta}$ and $\iota[\delta]^{s, t} = \iota^{s + \delta, t + \delta}$.

\item Let ${\bm V}$ and ${\bm W}$ be two persistence modules. A morphism from ${\bm V}$ to ${\bm W}$ is a collection of linear maps ${\bm f} = (f^s: V^s \to W^s)_{s \in{\mb R}}$ such that for all $s \leq t$ the following diagram commutes.
\beqn
\xymatrix{ V^s \ar[r]^{f^s} \ar[d]_{\iota^{s, t}_{\bm V}} & W^s \ar[d]^{\iota^{s, t}_{\bm W}} \\
           V^t  \ar[r]_{f^t}               & W^t   }
\eeqn

\item The direct sum of persistence modules is defined in a natural way.
\end{enumerate}
\end{defn}

\begin{rem}
The notion of persistence modules first appeared in topological data analysis and then was adopted to symplectic topology (see \cite{Polterovich_Shelukhin_2016, Usher_Zhang_2016}). In many papers, notably in the symplectically aspherical or monotone setting such as \cite{Polterovich_Shelukhin_2016, Shelukhin_2022},the Floer persistence modules are usually finite-dimensional and hence generate {\it barcodes} as in the original situation of topological data analysis. However, in more general situations such as \cite{Usher_Zhang_2016} where the Floer persistence modules are infinite-dimensional, the notion of barcodes becomes more complicated. See Subsection \ref{subsection_barcodes} for more details.
\end{rem}

\subsubsection{Interleaving distance}

\begin{defn}
Let $\delta \geq 0$. Two persistence modules ${\bm V}$, ${\bm W}$ are called {\bf $\delta$-interleaved} if there are ${\mb K}$-linear maps 
\begin{align*}
&\ f^s: V^s \to W^{s+\delta},\ &\ g^s: W^s \to V^{s+\delta}
\end{align*}
for all $s \in {\mb R}$ such that for all $s \leq t$ the following diagram commutes. 
\beqn
\xymatrix{ V^{s-\delta} \ar[r]_{f^{s-\delta}} \ar[d]  \ar@/^2.0pc/[rr]^{\iota^{s-\delta, s+\delta}_{\bm V}} & W^s \ar[r]_{g^s} \ar[d] & V^{s+\delta} \ar[r]_{f^{s + \delta}} \ar[d] & W^{s + 2\delta} \ar[d]\\
           V^{t-\delta} \ar[r]_{f^{t-\delta}} &        W^t \ar[r]_{g^{t}} \ar@/_2.0pc/[rr]_{\iota^{t, t+2\delta}_{\bm W}}  & V^{t+\delta} \ar[r]_{f^{t+\delta}} & W^{ t+ 2\delta} }
\eeqn
Here all vertical arrows are the structural maps in the persistence modules.

Define the {\bf interleaving distance} between ${\bm V}$ and ${\bm W}$ to be the infimum of all $\delta\geq 0$ such that ${\bm V}$ and ${\bm W}$ are $\delta$-interleaved; if such $\delta$ does not exist, define the interleaving distance to be $+\infty$. Here ends this definition.
\end{defn}

\subsubsection{Boundary depth of persistence modules and stability}

We observe that one can generalize the notion of boundary depth to persistence modules. 

\begin{defn}\label{defn_boundary_depth_2}
Let ${\bm V}$ be a persistence module over ${\mb K}$. The {\bf boundary depth} of ${\bm V}$, denoted by $\beta({\bm V})$, is the infimum of $\beta>0$ such that for all $s \in {\mb R}$, $x \in V^s$, if $\iota^{s, t}(x) = 0$ for some $t > s$, then $\iota^{s, s+\beta}(x) = 0$. 
\end{defn}

As we allow persistent modules to be infinite-dimensional, we reprove the stability result of boundary depth. 

\begin{prop}\label{prop_boundary_depth_stability}
Suppose ${\bm V}$, ${\bm W}$ are $\delta$-interleaved persistence modules. Suppose ${\bm V}$ has finite boundary depth. Then ${\bm W}$ has finite boundary depth and 
\beqn
\beta({\bm W}) \leq \beta({\bm V}) + 2\delta.
\eeqn
\end{prop}

\begin{proof}
Suppose on the contrary that $\beta({\bm W}) \geq \beta({\bm V}) + 2\delta + 2\epsilon$ for some $\epsilon>0$. Then there exist $s \in {\mb R}$ and $x \in W_s$ such that $\iota^{s, s + \beta({\bm V}) + 2\delta +  \epsilon}(x) \neq 0$. Then by the definition of $\delta$-interleaving, one has $y: = f^{s, s+ \delta}(x) \neq 0$ and 
\beqn
\iota^{s+\delta, s+\delta + \beta({\bm V}) + \epsilon}(y) \neq 0
\eeqn
but $y$ cannot survive eventually. This contradicts the definition of $\beta({\bm V})$.
\end{proof}

\subsubsection{Persistence modules associated to filtered Floer--Novikov complexes}\label{subsubsec:floer-persistent}

Fix a field ${\mb K}$. Let ${\mf c}$ be a Floer--Novikov package (see Definition \ref{defn_fn}) and $CF_\bullet({\mf c})$ be the associated filtered Floer--Novikov complex. Then the collection of homology groups
\beqn
V^s({\mf c}):= HF_\bullet^{\leq s}({\mf c}; \Lambda_{\mb K}^\Gamma)
\eeqn
together with the natural maps $\iota^{s, t}$ (cf. Equation \eqref{eqn:iota-map}) is a persistence module over ${\mb K}$, denoted by ${\bm V}({\mf c})$.

It is easy to derive from definitions the following stability results of persistence modules coming from Floer--Novikov complexes.

\begin{prop}
Let $CF_\bullet({\mf c}_i)$, $i = 1, 2$ be two Floer--Novikov complexes over a field ${\mb K}$ and ${\bm V}({\mf c}_i)$ be the associated persistence module. Then the interleaving distance between ${\bm V}({\mf c}_1)$ and ${\bm V}({\mf c}_2)$ is no greater than the quasiequivalence distance between $CF_\bullet({\mf c}_1)$ and $CF_\bullet({\mf c}_2)$.
\end{prop}

Moreover, the two notions of boundary depths (Definition \ref{defn_boundary_depth_1} and Definition \ref{defn_boundary_depth_2}) agree.

\begin{prop}\label{prop_boundary_depth_equivalence}
Let ${\mf c}$ be a Floer--Novikov package over ${\mb K}$. Then the boundary depth of the filtered Floer--Novikov complex $CF_\bullet({\mf c})$ and the boundary depth of the persistence module ${\bm V}({\mf c})$ coincide. 
\end{prop}

\begin{proof}
Let $\beta_1$ be the boundary depth of $CF_\bullet({\mf c})$ and $\beta_2$ be the boundary depth of ${\bm V}({\mf c})$. Suppose $[x]\in HF^{\leq s}_\bullet({\mf c})$ which does not survive eventually. Let $x\in CF^{\leq s}_\bullet({\mf c})$ be a representative. Then $x$ is exact. Then by Definition \ref{defn_boundary_depth_1}, for all $\epsilon>0$, one has
\beqn
x \in \partial ( CF^{\leq s+\beta_1 + \epsilon}_\bullet({\mf c})).
\eeqn
Hence $\iota^{s, s+\beta_1 + \epsilon}([x]) = 0$. As $\epsilon$ is arbitrary, this implies that $\beta_2 \leq \beta_1$. 

On the other hand, for all $s \in {\mb R}$ and all exact $x \in CF_\bullet^{\leq s} ({\mf c})$, the class $[x]\in HF_\bullet^{\leq s}({\mf c})$ does not survive eventually. Then by Definition \ref{defn_boundary_depth_2}, for any $\epsilon>0$, one has $\iota^{s, s+\beta_2 + \epsilon}([x]) = 0$. This implies that 
\beqn
x \in \partial ( CF^{\leq s + \beta_2 + \epsilon}_\bullet({\mf c})).
\eeqn
It follows that $\beta_1 \leq \beta_2$. Hence $\beta_1 = \beta_2$.
\end{proof}

\subsection{Reduced barcodes}\label{subsection_barcodes}

Barcodes, originally used in topological data analysis, have become a very convenient tool in quantitative symplectic geometry. In this paper, as we only need to use the lengths of bars rather than the positions of the bars, we define the following notion of reduced barcodes.

\begin{defn}\label{defn:bar}(cf. \cite[Definition 8.13, 8.14]{Usher_Zhang_2016})
\begin{enumerate}


\item A {\bf reduced barcode} is a finite multiset ${\mc B}$ of elements of $(0, +\infty]$.\footnote{To match with the traditional notion of barcode, one can regard a member of ${\mc B}$ the length of either a finite interval or a semi-infinite interval.} Although ${\mc B}$ is not a set in general and a member $L$ of ${\mc B}$ may appear multiple times, we still use the same notations as if $L$ is an element of a set ${\mc B}$, such as $L \in {\mc B}$, without confusion. Let ${\mc B}_{\rm finite}\subset {\mc B}$ denote the submultiset of finite bars, i.e., those with $L < +\infty$. 

\item The {\bf total bar length} of a reduced barcode ${\mc B}$ is 
\beqn
\tau ({\mc B} ):= \sum_{L_i \in {\mc B}_{\rm finite}} L_i.
\eeqn



\item The {\bf reduced bottleneck distance} between two reduced barcodes ${\mc B}$ and ${\mc B}'$, denoted by $d_B({\mc B}, {\mc B}')$, is the infimum of $\delta>0$ such that, after removing certain submultisets ${\mc B}_{\rm short}\subset {\mc B}$ and ${\mc B}_{\rm short}' \subset {\mc B}'$ whose members all have length at most $2\delta$, there is a bijection between ${\mc B}\setminus {\mc B}_{\rm short}$ and ${\mc B}' \setminus {\mc B}_{\rm short}'$ such that the differences of the corresponding bar lengths are all bounded by $\delta$.

\end{enumerate}
\end{defn}

The bottleneck distance is symmetric and satisfies the triangle inequality. It is not a metric in the usual sense as it may take infinite value. Indeed, $d_B({\mc B}, {\mc B}')< \infty$ if and only if ${\mc B}$ and ${\mc B}'$ has the same number of infinite bars. 





\begin{prop}(cf. \cite[Proposition 20]{Roux_Seyfaddini_Viterbo_2021}) 
For any $k  \geq 0$, the completion of the set of reduced barcodes having $k$ infinite bars is the set of possibly infinite reduced barcodes (with $k$ infinite bars) such that for all $\epsilon>0$, the number of finite bars with length greater than $\epsilon$ is finite. 
\end{prop}

\subsubsection{Reduced barcodes associated to Floer--Novikov complexes}

Usher--Zhang \cite{Usher_Zhang_2016} defined for each ${\mb Z}_2$-graded Floer-type complexes over $\Lambda_{\mb K}^\Gamma$ and each $k \in {\mb Z}_2$ the associated barcodes (which allows bars of length zero). As Floer--Novikov complexes are all Floer-type complexes, one has an associated reduced barcode. Let the reduced barcode associated to a Floer--Novikov complex $CF_\bullet({\mf c})$ be ${\mc B}({\mf c})$. As the differential strictly decreases the action, there are no bars of length zero (which was allowed in the abstract setting of \cite{Usher_Zhang_2016}). We do not recall the detail of the definition here.

\begin{prop}\label{prop:beta-equal}
Let ${\bm V}({\mf c})$ be the persistence module induced from a filtered Floer--Novikov complex $CF_\bullet({\mf c})$ over $\Lambda_{\mb K}^\Gamma$. Then the boundary depth of ${\bm V}({\mf c})$ (see Definition \ref{defn_boundary_depth_1} and Definition \ref{defn_boundary_depth_2}) coincides with the length of the longest finite bar in ${\mc B}({\mf c})$. In particular, the boundary depth is zero if and only if ${\mc B}({\mf c})$ has no finite bar.
\end{prop}

\begin{proof}
It follows from the definitions of boundary depth and barcodes (via singular value decompositions, see \cite{Usher_Zhang_2016}). The details are left to the reader.
\end{proof}

We also have the following stability result about reduced barcode.

\begin{thm}\cite[Theorem 8.17]{Usher_Zhang_2016}\label{thm:bottle-neck-Hofer}
Let $(CF_\bullet({\mf c}_1), \partial_1, \ell_1)$ and $(CF_\bullet({\mf c}_2), \partial_2, \ell_2)$ be two Floer--Novikov complexes associated to Floer--Novikov data ${\mf c}_1$, ${\mf c}_2$ over a field ${\mb K}$. Suppose the quasiequivalence distance between $CF_\bullet({\mf c}_1)$ and $CF_\bullet({\mf c}_2)$ is finite. Then 
\beqn
d_B ({\mc B}({\mf c}_1), {\mc B} ({\mf c}_2)) \leq 2 d_Q ( CF_\bullet({\mf c}_1), CF_\bullet({\mf c}_2)).
\eeqn
\end{thm}

\subsection{$A_\infty$ algebras and Hochschild cohomology}

Let ${\mb K}$ be a field of characteristic zero. We recall the notion of ${\mb Z}_2$-graded $A_\infty$ algebras over the Novikov field $\Lambda_{\mb K}$. 

\begin{defn}[Curved $A_\infty$ algebra]
\begin{enumerate}

\item A ${\mb Z}_2$-graded {\bf curved $A_\infty$ algebra} over $\Lambda_{\mb K}$ consists of a ${\mb Z}_2$-graded $\Lambda_{\mb K}$-vector space ${\mc A}$ (the degree of a homogeneous element $a$ is denoted by $|a|$) and for all positive integers $k \geq 0$ {\bf higher composition maps}
\beqn
m_k: {\mc A}^{\otimes k} \to {\mc A}\ ({\rm where}\ m_0: \Lambda_{\mb K} \to A)
\eeqn
(which are $\Lambda_{\mb K}$-linear and have degree $k\ {\rm mod}\ 2$). The higher composition maps need to satisfy the following $A_\infty$ composition law: for all $k \geq 1$ and $a_k, \ldots, a_1 \in {\mc A}$,\footnote{There are two different conventions: the variables are either ordered as $a_1, \ldots, a_k$ or ordered as $a_k, \ldots, a_1$.}
\beqn
\sum_{i=0}^k \sum_{j=0}^{k-i} (-1)^{\maltese_1^j} m_{k-i+1} \left( a_k, \ldots, a_{i+j+1}, m_i ( a_{i+j}, \ldots, a_{j+1}), a_j, \ldots, a_1 \right) = 0
\eeqn
where the symbol $\maltese_a^b$ for all $a < b$ is defined as 
\beq\label{maltese}
\maltese_a^b = \sum_{a \leq i \leq b} \| a_i \| \quad {\rm where}\ \| a_i \| = |a_i| + 1.
\eeq

\item The {\bf curvature} of a curved $A_\infty$ algebra is the element 
\beqn
m_0(1) \in {\mc A}.
\eeqn
If $m_0 = 0$, then we say that the $A_\infty$ algebra is {\bf flat}.

\item Given a (curved or flat) $A_\infty$ algebra ${\mc A}$, a {\bf cohomological unit} is an even element $e\in {\mc A}$ such that $m_1(e) = 0$ and that for all homogeneous $x\in {\mc A}$
\beqn
(-1)^{|x|} m_2(e, x) = m_2(x, e) = x.
\eeqn
$e$ is called a {\bf strict unit} if in addition
\beqn
m_k(\ldots, e, \ldots) = 0\ \forall k \geq 3.
\eeqn
In these two cases we call $({\mc A}, e)$ a cohomologically unital (resp. strictly unital) $A_\infty$ algebra.

\item When ${\mc A}$ is flat, $A_\infty$ composition law implies that $m_1\circ m_1 = 0$. The {\bf cohomology algebra} of $A$, denoted by $H^\bullet({\mc A})$, is the ${\mb Z}_2$-graded associative $\Lambda_{\mb K}$ algebra whose underlying space is $H^\bullet({\mc A}) = {\rm ker} m_1/ {\rm Im} m_1$ and whose multiplication is induced from $m_2$.
\end{enumerate}
\end{defn}

Because of bubbling of holomorphic disks, $A_\infty$ algebras associated to a Lagrangian brane is generally curved. There is a way to turn certain curved $A_\infty$ algebras to flat ones. 

\begin{defn}
Let $({\mc A}, e)$ be a strictly unital $A_\infty$ algebra. A {\bf weakly bounding cochain} of $({\mc A}, e)$ is an odd element $b \in {\mc A}^{\rm odd}$ such that 
\beqn
m(b):= \sum_{k \geq 0} m_k( b, \ldots, b) = W(b) e\ {\rm where}\ W(b) \in \Lambda_{\mb K}.
\eeqn
\end{defn}

Suppose $b$ is a weakly bounding cochain of $({\mc A}, e)$. Then define ${\mc A}^\flat$ (which depends on the weakly bounding cochain $b$) to be the flat $A_\infty$ algebra whose underlying space is the same as ${\mc A}$ and whose composition maps $m_k^\flat$ is defined by 
\beqn
m_k^\flat( x_k, \ldots, x_1) := \sum_{l_0, \ldots, l_k \geq 0} m_{k + l_0 + \cdots + l_k} (\underbrace{b, \ldots, b}_{l_k}, x_k, \cdots, x_1, \underbrace{b,\ldots, b}_{l_0} ).
\eeqn

\begin{lemma}
${\mc A}^\flat$ is a flat $A_\infty$ algebra. \qed
\end{lemma}

\subsubsection{Hochschild cohomology for associative algebras}

Let $A$ be a ${\mb Z}_2$-graded associative algebra over ${\mb K}$. Hochschild cohomology $\hh^\bullet(A, M)$ can be defined for all ${\mb Z}_2$-graded bimodules $M$ of $A$. Here we only consider the case when $M = A$. The Hochschild cochain complex (with coefficients in $A$ itself) is defined by
\beqn
\CC^{\bullet, n}(A):= \CC^{\bullet, n}(A, A):= {\rm Hom}_{\Lambda_{\mb K}}^\bullet (A^{\otimes n}, A[n]).
\eeqn
Here the bullet is the ${\mb Z}_2$-grading on linear maps and $A[n]$ is the ${\mb Z}_2$-graded vector space $A$ with the ${\mb Z}_2$-grading shifted by $n$ (modulo $2$). Denote the ${\mb Z}_2$-degree of a homogeneous element $\phi \in \CC^{\bullet, \bullet}(A)$ by $|\phi| \in {\mb Z}_2$ and the {\it reduced} degree by 
\beqn
\| \phi \|:= |\phi| + 1 \in {\mb Z}_2.
\eeqn

A Hochschild cochain is represented by a sequence $\tau = (\tau_n)_{n \geq 0}$ of such multi-linear maps. The differential $\delta_{\CC}$, which raises the length grading $n$ by $1$, is defined by 
\begin{multline}\label{dcc}
(\delta_{\CC}(\tau))(x_{n+1}, \ldots, x_1)  =  x_{n+1} \tau_n (x_n, \ldots, x_1) + (-1)^{\| \tau \| \| x_1 \|} \tau_n(x_{n+1}, \ldots, x_2)x_1\\
- \sum_{0 \leq i < n} (-1)^{\| \tau \| + \maltese_1^i } \tau_n(x_{n+1}, \ldots, x_{i+2} x_{i+1}, x_i, \ldots, x_1).
\end{multline}
The cohomology defined by $\delta_{\CC}$ is called the {\bf Hochschild cohomology} of $A$ (with coefficients in $A$). As the simplest example, via a straightforward calculation one obtains (for $A = {\mb K}$ trivially graded)
\beqn
\hh^{\bullet, n} ({\mb K}) = \left\{ \begin{array}{cc} {\mb K},\ &\ n = 0 \text{ and } n \text{ even},\\
                               0,\ &\ \text{ otherwise}, \end{array}\right.
\eeqn
where the superscript $n$ comes from the length filtration of Hochschild cochains.

\begin{rem}
The formula \eqref{dcc} differs from the usual version of Hochschild differential, see for example \cite[(1.5.1.1)]{Loday_cyclic}. Indeed, suppose $A$ is ungraded, i.e. all elements are even. Then the ${\mb Z}_2$-grading of a length $n$ cochain is $n$ mod $2$. In this case \eqref{dcc} reduces to 
\begin{multline*}
(\delta_{\CC}(\tau))(x_{n+1}, \ldots, x_1) = x_{n+1} \tau_n(x_n, \ldots, x_1) + (-1)^{n+1} \tau_n(x_{n+1}, \ldots, x_2) x_1 \\
+ (-1)^{n+i} \tau_n(x_{n+1}, \ldots, x_{i+2} x_{i+1}, x_i, \ldots, x_1).
\end{multline*}
If we replace $A$ by the opposite algebra $A^{\rm op}$ (i.e. the same set with multiplication reversed), and identify a length $n$ Hochschild cochain $\tau$ on $A$ with $\tau^{\rm op}$ on $A^{\rm op}$ defined by $\tau^{\rm op}(x_1, \ldots, x_n) =\tau(x_n, \ldots, x_1)$. Then the above formula differ from the standard Hochschild differential on $\tau^{\rm op}$ up to a sign $(-1)^{n+1}$.
\end{rem}

\subsubsection{Hochschild cohomology for $A_\infty$ algebras}

Now let ${\mc A}^\flat$ be a flat ${\mc A}_\infty$ algebra. Define the length $n$-part of Hochschild cochain complex of $A^\flat$ to be 
\beqn
\CC^{\bullet, n} ({\mc A}^\flat) = \CC^{\bullet, n} ({\mc A}^\flat, {\mc A}^\flat) = {\rm Hom}_{\Lambda_{\mb K}}^\bullet ( ({\mc A}^\flat)^{\otimes n}, {\mc A}^\flat[n]).
\eeqn
Here $\bullet$ denotes the ${\mb Z}_2$-grading and ${\mc A}^\flat[n]$ denote the super vector space ${\mc A}^\flat$ with grading shifted by $n$ (mod $2$).

On the Hochschild cochain complex there is the {\bf Gerstenhaber product} (which is graded with respect to the reduced grading $\| \cdot \|$) defined by 
\beqn
(\phi \circ \psi)(x_s, \ldots, x_1) = \sum_{i+j+k = s} (-1)^{\| \psi \| \cdot \maltese_1^i} \phi(x_s, \ldots, \psi(x_{i+j}, \cdots, ), x_i, \ldots, x_1)
\eeqn
as well as the {\bf Gerstenhaber superbracket}
\beqn
[\phi, \psi]:= \phi \circ \psi - (-1)^{\| \phi \| \cdot \| \psi \| } \psi \circ \phi.
\eeqn
Then the $A_\infty$-structure on ${\mc A}^\flat$ is equivalent to an even Hochschild cochain $m^\flat$ with $m^\flat_0 = 0$ with the $A_\infty$ relation being equivalent to
\beqn
[m^\flat, m^\flat] = 2 m^\flat \circ m^\flat = 0.
\eeqn
We define the Hochschild differential $\delta_\CC$ on 
\beqn
\CC^\bullet({\mc A}^\flat) = \prod_{n \geq 0} \CC^{\bullet, n} ({\mc A}^\flat)
\eeqn
by the formula
\beqn
\delta_\CC(\phi):= [ m^\flat, \phi].
\eeqn
Notice that if $m^\flat_k \neq 0$ only when $k = 2$, ${\mc A}^\flat$ is a ${\mb Z}_2$-graded associative algebra with the Hochschild differential reduces to the differential \eqref{dcc}. The Hochschild cohomology of ${\mc A}^\flat$ is defined by 
\beqn
\hh^\bullet({\mc A}^\flat ):= {\rm ker} \delta_\CC/ {\rm im} \delta_\CC.
\eeqn

On $\CC^\bullet({\mc A}^\flat )$ there is also an $A_\infty$ structure whose composition maps start with $\delta_\CC$. We only need the 2-fold composition map, i.e., the Yoneda product.

\begin{defn}
The {\bf Yoneda product} on $\CC^\bullet({\mc A}^\flat)$, denoted by $\star$, is defined by 
\begin{multline}\label{Yoneda}
(\phi \star \psi) (a_k, \ldots, a_1) \\
= \sum (-1)^\clubsuit m_k^\flat \left( a_k, \ldots, \phi_r (a_{i+r}, \ldots, a_{i+1}), \cdots, \psi_s (a_{j+s}, \ldots, a_{j+1}), \ldots, a_1 \right)
\end{multline}
where the sum is taken over all $i, j, r, l$ such that each summand makes sense. The sign is defined by (see \eqref{maltese} for the definition of $\maltese$)
\beq\label{clubsuit}
\clubsuit:= \|\phi\| \cdot \left( \maltese_1^i + |\psi| \right) + \| \psi\| \cdot  \maltese_1^j.
\eeq
\end{defn}


The proof of the following can be found in \cite{Ganatra_thesis} and \cite{Mescher_ainfinity}.

\begin{prop}
The map $\star: \CC^\bullet({\mc A}^\flat) \otimes \CC^\bullet({\mc A}^\flat) \to \CC^\bullet({\mc A}^\flat)$ is a cochain map of even degree.
\end{prop}

Therefore the Yoneda product descends to Hochschild cohomology. We still call the induced one by Yoneda product and denote it by the same symbol $\star$. As the Yoneda product can be extended to an $A_\infty$ structure on $CC^\bullet({\mc A}^\flat)$, the Yoneda product on the Hochschild cohomology is associative. 

The Yoneda product has a chain-level unit in the strictly unital case. 

\begin{prop}
Suppose ${\mc A}^\flat$ has a strict unit $e$. Then the Hochschild cochain ${\bm 1}_{{\mc A}^\flat}$ defined by 
\beqn
{\bm 1}_{{\mc A}^\flat}(x_k, \ldots, x_1) = \left\{ \begin{array}{cc} 0,\ &\ k \geq 1,\\
                                   e,\ &\ k = 0. \end{array}\right.
\eeqn
is a unit with respect to the Yoneda product.
\end{prop}

\begin{proof}
Straightforward.
\end{proof}

Finally, we remark that the Yoneda product on $\hh^\bullet({\mc A}^\flat )$ is graded commutative. It is compatible with the Gerstenhaber bracket, which makes $\hh^\bullet({\mc A}^\flat )$ into a Gerstenhaber algebra.

\subsubsection{Clifford algebras}

The Lagrangian Floer cohomology ring of a torus is often isomorphic to a Clifford algebra. Hence the Hochschild cohomology of Clifford algebras are one of the most important cases related to symplectic geometry and mirror symmetry. Recall that given a finite-dimensional ${\mb K}$-vector space $W$ equipped with a quadratic form $q$, the Clifford algebra $Cl(W, q)$ is the tensor algebra of $W$ modulo the relation 
\beqn
w \otimes w' + w' \otimes w + 2 q(w, w') {\rm Id} = 0.
\eeqn
We only care about the case when $q$ is nondegenerate and when ${\mb K}$ has characteristic zero and is algebraically closed. In this case, all nondegenerate quadratic forms are equivalent to the standard one. When $W$ has dimension $n$, we abbreviate $Cl(W, q)$ by $Cl_n$. 

\begin{prop}\label{HH_computation_1}
For all $n \geq 0$, the Hochschild cohomology of $Cl_n$ is 
\beqn
\hh^k( Cl_n, Cl_n) = \left\{ \begin{array}{cc} {\mb K},\ &\ k = 0,\\
                                               0,\ &\ k \geq 1. \end{array} \right.
\eeqn
In particular, $\hh^0(Cl_n, Cl_n)$ is generated by the identity.
\end{prop}

\begin{proof}
The calculation was provided by Sheridan \cite{Sheridan_2016} and we recall it here. First, Hochschild cohomology is Morita invariant (see \cite[1.5.6]{Loday_cyclic}). Second, there are only two Morita equivalence classes among Clifford algebras, the even ones and the odd ones (Bott periodicity). Hence we only need to calculate for $n = 0$ and $n=1$. When $n= 0$,  $Cl_0 \cong {\mb K}$, giving $\hh^\bullet({\mb K}, {\mb K}) = {\mb K}$. When $n = 1$ \footnote{The original calculation of \cite{Sheridan_2016} for the $n=1$ case has a mistake resulting in a false result. We thank Nick Sheridan for clarifying it.}, the calculation can be deduced from the more general case of J. Smith \cite[Section 5]{Smith_2019} using reduced Hochschild cohomology.
\end{proof}

When the Floer cohomology algebra of a Lagrangian brane is isomorphic to a Clifford algebra, the argument via {\it formality} shows that the Hochschild cohomology of the corresponding $A_\infty$ algebra is the same as the Hochschild cohomology of the cohomology algebra. Recall that an $A_\infty$ algebra is called {\it formal} if it is $A_\infty$ quasi-isomorphic to its cohomology algebra. An associated algebra $A$ is called {\it intrinsically formal} if any ${\mb Z}_2$-graded $A_\infty$ algebra whose cohomology algebra is isomorphic to $A$ is formal. It was shown in \cite[Corollary 6.4]{Sheridan_2016} that $Cl_n$ is intrinsically formal. Due to the Morita invariance of Hochschild cohomology, the following statement is immediate.

\begin{cor}\label{HH_computation_2}
If ${\mc A}^\flat$ is a flat $A_\infty$ algebra over ${\mb K}$ whose cohomology algebra is isomorphic to $Cl_n$, then 
\beqn
\hh^\bullet ({\mc A}^\flat) = {\mb K}.
\eeqn 
\end{cor}


Notice that if in addition ${\mc A}^\flat$ is strictly unital, $1_{{\mc A}^\flat} \neq 0$ and it generates the Hochschild cohomology.

\section{Vortex Hamiltonian Floer theory}\label{sec:ham-package}\label{section4}

We review the construction of vortex Hamiltonian Floer theory developed by the second author \cite{Xu_VHF} following the proposal of Cieliebak--Gaio--Salamon \cite{Cieliebak_Gaio_Salamon_2000}.

\subsection{Floer chain complexes}

\subsubsection{Equivariant action functional}

Our convention for Hamiltonian vector field is fixed as follows. Let $(M, \omega)$ be a symplectic manifold and $H: M \to {\mb R}$ be a smooth function. The associated Hamiltonian vector field $X_H$ is specified by 
\beqn
dH = \omega( X_H, \cdot).
\eeqn

We would like to consider the Hamiltonian dynamics upstairs in the gauged linear sigma model. Let $X$ be the toric manifold we are considering. Let $H: S^1 \times X \to {\mb R}$ be a smooth Hamiltonian function. Let ${\rm Per}(H)$ be the set of $1$-periodic orbits \footnote{They must be contractible as $X$ is simply connected.} of $ H$ whose elements are maps $x: S^1 \to X$. The Hamiltonian $ H$ lifts to a $K$-invariant function on $S^1\times \mu^{-1}(0)$. Choose an arbitrary $K$-invariant extension $\wh H: S^1 \times V \to {\mb R}$ whose support is compact and disjoint from the unstable locus $V^{\rm us}$ under the $K^{\mb C}$-action. Consider the set of {\bf equivariant loops}
\beq
L^K (V):= \Big\{ {\mf x} = (\wh x, \zeta): S^1 \to V \times {\mf k}.
\Big\}
\eeq
Here function $\zeta: S^1 \to {\mf k}$ can be viewed as a gauge field on $S^1$. Notice that as $V$ is contractible, the loop $x$ is contractible and there is only one homotopy class of cappings. The loop group $L K$ acts on the set of capped equivariant orbits by 
\beqn
g\cdot {\mf x} = ( g\cdot \wh x, g\cdot \xi)\ {\rm where}\ (g \cdot \wh x)(t) = g(t) \wh x (t),\ (g \cdot \xi)(t) = \zeta (t) - \frac{d}{dt} \log g(t).
\eeqn
Define the action functional
\beq
\wh {\mc A}_H: L^K(V) \to {\mb R},\ {\mf x}\mapsto - \int_{\mb D} u^* \omega_V + \int_{S^1} \left( \langle \mu( \wh x (t)), \zeta (t) \rangle - \wh H_t( \wh x(t))  \right) dt
\eeq
where $u: {\mb D}^2 \to V$ is any capping. Critical points are solutions 
\begin{align}
&\ \mu(\wh x (t)) \equiv 0,\ &\ \wh x'(t) = X_{\wh H_t}(\wh x(t)) - {\mc X}_{\zeta (t)}(\wh x(t)).
\end{align}
Here ${\mc X}_{\zeta (t)}$ is the Hamiltonian vector field of the function $\langle \mu, \zeta (t) \rangle$. 

The action functional satisfies the following transformational law with respect to the loop group action. Indeed, for $g \in LK$ and ${\mf x} \in L^K(V)$, one has 
\beq\label{deck}
\wh {\mc A}_H ( g {\mf x}) = - \omega^K (g) + {\mc A}_H ({\mf x}).
\eeq
Denote 
\beqn
{\mf L}^K(V):= L^K(V)/ {\rm ker} \omega^K.
\eeqn
Its elements are denoted by $[{\mf x}]$. Then $\Gamma \cong LK/ {\rm ker} \omega^K$ acts on ${\mf L}^K(V)$. We denote the action by $g \cdot [{\mf x}]$. Then $\wh {\mc A}_H$ induces a functional on ${\mf L}^K(V)$, denoted by
\beqn
{\mc A}_H: {\mf L}^K(V) \to {\mb R}.
\eeqn
Each critical point of ${\mc A}_H$ is called an {\bf equivariant 1-periodic Hamiltonian orbit}.

There is a correspondence between ordinary Hamiltonian orbits downstairs and equivariant Hamiltonian orbits upstairs. More precisely, let $\widetilde {\rm Per}(H)$ be the covering of ${\rm Per}(H)$ consisting of equivalence classes $[u, x]$ of capped 1-periodic orbits of $H$: the equivalence relation $(u, x) \sim (u', x')$ is defined by the equality of action values:
\beqn
(u, x) \sim (u', x') \Longleftrightarrow x = x' \in {\rm Per}(H)\ {\rm and}\ \int_{{\mb D}^2} u^* \omega_X = \int_{{\mb D}^2} (u')^* \omega_X.
\eeqn
Then there is a map
\beq\label{downtoup}
\iota: \widetilde{\rm Per}(H) \to {\rm crit} {\mc A}_H \subset {\mf L}^K(V).
\eeq
Indeed, suppose $ x: S^1 \to X$ is a contractible 1-periodic orbit of the Hamiltonian flow of $  H$ and $ u: {\mb D}^2 \to X$ is a capping of $ x$. View $\mu^{-1}(0) \to X$ as a principal $K$-bundle $P$. The Euclidean metric on $V$ induces a connection on $P$ whose horizontal distribution is the orthogonal complement of tangent planes of $K$-orbits; equivalently, this gives a connection 1-form $\theta \in \Omega^1(\mu^{-1}(0)) \otimes {\mf k}$. The pullback $u^* P \to {\mb D}^2$ is trivial and different trivializations differ by a smooth map $g: {\mb D}^2 \to K$. Any trivialization of this pullback bundle induces a connection matrix $u^* \theta$ whose boundary restriction is $\zeta (t) dt$. A trivialization also induces a map $\wh u: {\mb D}^2 \to \mu^{-1}(0)$ lifting $u$. Let the boundary restriction of $\wh u$ be $\wh x$. Then ${\mf x} = (\wh x, \zeta)$ is an equivariant 1-periodic orbit, well-defined up to $L_0 K$-actions. Furthermore, if $u'$ is a different capping with the same resp. different action value, then the correspondence we just described gives the same resp. different element in ${\mf L}^K(V)$. 

\begin{lemma}
In the toric case the map \eqref{downtoup} is bijective.
\end{lemma}

\begin{proof}
Given any equivariant Hamiltonian orbit ${\mf x}$ upstairs, the map $\wh x: S^1 \to \mu^{-1}(0)$ projects down to a 1-periodic orbit $x: S^1 \to X$. As $X$ is simply connected, $x$ is contractible. Choose a capping $u: {\mb D}^2 \to X$ and let ${\mf x}' = (\wh x', \zeta')$ be equivariant Hamiltonian orbit lifting $[u, x]$. As the $K$-action on $\mu^{-1}(0)$ is free, there is a gauge transformation on the circle making $\wh x' = \wh x$. The condition 
\beqn
\wh x'(t) = X_{H_t}(\wh x (t)) - {\mc X}_{\zeta(t)} (\wh x(t))
\eeqn
implies that $\zeta = \zeta'$. 
\end{proof}

\begin{defn}\label{Conley-Zehnder}
The {\bf Conley--Zehnder index} of an equivariant 1-periodic orbit ${\mf x} \in {\rm crit} {\mc A}_H$ is the usual Conley--Zehnder index of the capped 1-periodic orbit $\iota^{-1}({\mf x}) \in \wt {\rm Per}(H)$, denoted by ${\rm CZ}({\mf x}) \in {\mb Z}$.
\end{defn}

\subsubsection{Floer trajectories}

Similar to the standard Hamiltonian Floer theory, one considers the equation for the gradient flow of the equivariant action functional. Choose a 1-periodic $K$-invariant $\omega_V$-compatible almost complex structure $\wh J_t$ on $V$. Formally the negative gradient flow equation of $\wh {\mc A}_H$ is the following equation for pairs $(u, \eta): {\mb R}\times S^1 \to V \times {\mf k}$
\begin{align*}
&\ \partial_s u + \wh J_t \left( \partial_t u + {\mc X}_\eta (u) - X_{\wh H_t}(u) \right) = 0,\ &\ \partial_s \eta + \mu(u) = 0.
\end{align*}
This is in fact the symplectic vortex equation on the cylinder ${\mb R}\times S^1$ for the trivial $K$-bundle and the standard cylindrical volume form, written in temporal gauge $A = d + \eta dt$. In general, for $A = d + \xi ds + \eta dt$, the vortex equation \eqref{vortex_equation} reads
\begin{align}\label{vortex_equation_2}
&\ \partial_s u + {\mc X}_{\xi}(u) + \wh J_t \left( \partial_t u + {\mc X}_{\eta} (u) - X_{\wh H_t}(u) \right)  = 0,\ &\ \partial_s \eta - \partial_t \xi + \mu(u) = 0.
\end{align}
It was shown in \cite{Xu_VHF} that any finite energy solution converges up to gauge transformation to critical points of ${\mc A}_H$. 

\begin{thm}\cite[Theorem 3.1, Corollary 4.3]{Xu_VHF} 
\begin{enumerate}
    \item Given a bounded solution ${\mf u} = (u, \xi, \eta)$ (i.e. finite energy solution with $u({\mb R}\times S^1)$ bounded) to \eqref{vortex_equation_2}, there exist a gauge equivalent solution, still denoted by $(u, \xi, \eta)$, as well as equivariant 1-periodic orbits ${\mf x}_\pm = (\wh x_\pm, \zeta_\pm)$ such that uniformly for $t \in S^1$
\beq\label{asymptotic}
\lim_{s \to \pm \infty} (u(s, \cdot), \xi(s, \cdot), \eta(s, \cdot)) = (\wh x_\pm, 0, \zeta_\pm).
\eeq

\item If ${\mf x}_\pm'$ are another pair of equivariant 1-periodic orbits satisfying \eqref{asymptotic} with ${\mf u}$ replaced by any gauge equivalent solution, then there exists $g_\pm \in LK$ with $g_- g_+^{-1} \in L_0 K$ such that ${\mf x}_\pm' = g_\pm {\mf x}_\pm$. 

\item If $H$ is a nondegenerate Hamiltonian downstairs, then one can make the convergence \eqref{asymptotic} exponentially fast by choosing suitable gauge equivalent solutions. More precisely, there exist $C>0$ and $\delta>0$ such that
\beqn
d_V(u(s, t), \wh x_\pm(t)) + | \xi(s, t)| + |\eta(s, t) - \zeta(t)| \leq C e^{-\delta|s|}.
\eeqn
Here $d_V$ is the Euclidean distance on $V$. Similar exponential decay estimates hold for covariant derivatives of arbitrary higher orders.\footnote{Exponential decay type estimates for vortices can also be found in \cite{Ziltener_Decay}\cite{Venugopalan_quasi}\cite{Chen_Wang_Wang}.}
\end{enumerate}
\end{thm}

Therefore, one can use a pair of elements ${\mf x}_\pm\in {\rm crit}{\mc A}_H \subset {\mf L}^K(V)$ to label solutions. Let 
\beqn
{\mc M}({\mf x}_-, {\mf x}_+)
\eeqn
be the set of gauge equivalence classes of bounded solutions ${\mf u}$ to \eqref{vortex_equation_2} modulo the ${\mb R}$-translation. One has the energy identity (\cite[Proposition 3.8]{Xu_VHF})
\beqn
E({\mf u}) = {\mc A}_H ({\mf x}_-) - {\mc A}_H ({\mf x}_+).
\eeqn

\begin{rem}\label{rem_admissible}
To achieve transversality, one has to avoid certain ``bad'' $K$-equivariant lifts of a given Hamiltonian $H$ downstairs and choose almost complex structures appropriately. In \cite[Section 6]{Xu_VHF} the second author used the notion of {\it admissible} almost complex structures and {\it admissible} $K$-invariant lifts of a Hamiltonian downstairs. We briefly recall the precise meanings of them adapted to the toric case. First, in the stable locus $V^{\rm st}$ there is the projection $\pi: V^{\rm st} \to X$ which is invariant under the complex torus $G$. Hence there is a splitting 
\beqn
TV|_{V^{\rm st}} \cong \pi^* TX \oplus ({\mf k}\otimes {\mb C}).
\eeqn
Throughout this paper we fix a $K$-invariant (small) open neighborhood $U$ of $\mu^{-1}(0)$ and consider only $K$-invariant, $\omega_V$-compatible almost complex structures $\wh J$ on $V$ which agrees with $J_V$ outside $U$ (this is necessary to guarantee the $C^0$-compactness in \cite{Xu_VHF}). Moreover, given a {\it nondegenerate} Hamiltonian $H: S^1 \times X \to {\mb R}$, an $S^1$-family of almost complex structures $\wh J_t$ is said to be {\it admissible} with respect to the $H$ downstairs if for any loop $\wh x: S^1 \to \mu^{-1}(0)$ that projects to a 1-periodic orbit downstairs, one imposes some conditions on the $1$-jet of $\wh J_t$ along $\wh x$ (see \cite[Definition 6.2]{Xu_VHF}. Then the notion of admissibility of $K$-invariant lifts of $H$ was defined (see \cite[Definition 6.5]{Xu_VHF}), which is a condition on the infinitesimal behavior of the lifts $\wh H_t$ along 1-periodic orbits given in terms of the Hessian of the equivariant action functional. 
\end{rem}


\begin{thm}
Given a nondegenerate Hamiltonian $H_t$ downstairs, for a generic admissible pair $(\wh H_t, \wh J_t)$, the following is true. 
\begin{enumerate}
\item Each moduli space ${\mc M}({\mf x}, {\mf y})$ is regular and has dimension ${\rm CZ}({\mf x}) - {\rm CZ}({\mf y}) - 1$. 

\item Moduli spaces with bounded energy are compact up to breaking as in the usual setting for the Uhlenbeck--Gromov--Floer compactification.

\item If ${\rm CZ}({\mf x}) - {\rm CZ}({\mf y}) = 1$, then the moduli space consists of finitely many points.

\item When ${\rm CZ}({\mf x}) - {\rm CZ}({\mf y}) = 2$, the compactified moduli space is a compact 1-dimensional manifold with boundary.
\end{enumerate}
\end{thm}

We briefly explain the reason for transversality. Indeed, as the total energy is finite and the volume on the cylinder is infinite, near infinity any solution is contained in the neighborhood $U$ of $\mu^{-1}(0)$ fixed in Remark \ref{rem_admissible}. Therefore, there is a nonempty open subset of the cylinder whose image is contained in the free locus of the $K$-action. Then using an equivariant version of the argument of Floer--Hofer--Salamon \cite{Floer_Hofer_Salamon} one can achieve transversality by perturbing $\wh J_t$ in a neighborhood of $\mu^{-1}(0)$. 

It is a standard procedure to construct a coherent system of orientations on the moduli spaces (see \cite{Floer_Hofer_Orientation}). Then for $R = {\mb Z}$, there is a well-defined count
\beqn
n({\mf x}, {\mf y}) \in {\mb Z}
\eeqn
which is the signed count of the number of Floer trajectories in $0$-dimensional components of ${\mc M}({\mf x}, {\mf y})$. When $R$ is any commutative ring with a unit, $n({\mf x}, {\mf y})$ induces an element
\beqn
n_R({\mf x}, {\mf y}) \in R
\eeqn
by the induced count via the map ${\mb Z} \to R$. 

\subsubsection{Floer homology} 


We first define the Floer chain group for a smaller Novikov ring. Recall that one has the finitely generated abelian group 
\beqn
\Gamma:= LK/ {\rm ker} \omega^K
\eeqn
which naturally embeds into ${\mb R}$. For any commutative ring $R$, introduce
\beqn
\Lambda_R^\Gamma:= \Big\{ \sum_{i=1}^\infty a_i T^{g_i}\in \Lambda_R\ |\ g_i \in \Gamma \Big\}.
\eeqn
We define the Floer chain group $\vcf_\bullet(\wh H)$ to be the ``downward'' completion: 
\beqn
\vcf_\bullet(\wh H; \Lambda_R^\Gamma ) = \Big\{ \sum_{i=1}^\infty b_i {\mf x}_i\ |\ b_i \in R,\ {\mf x}_i \in {\rm crit} {\mc A}_H,\ \lim_{i \to \infty} {\mc A}_H ({\mf x}_i) = -\infty \Big\}.
\eeqn
It is graded by the Conley--Zehnder index (modulo 2). The $\Lambda_R^\Gamma$-module structure is defined by 
\beqn
\Big( \sum_{i=1}^\infty a_i T^{g_i} \Big) \Big( \sum_{j=1}^\infty b_j {\mf x}_j \Big) = \sum_{i, j = 1}^\infty a_i b_j (g_i \cdot {\mf x}_j).
\eeqn
By \eqref{deck}, the right hand side is in $\vcf_\bullet(\wh H; \Lambda_R^\Gamma)$ and this is a well-defined action. Define 
\beqn
\vcf_\bullet(\wh H; \Lambda_R):= \vcf_\bullet( \wh H; \Lambda_R^\Gamma) \otimes_{\Lambda_R^\Gamma} \Lambda_R.
\eeqn

The Floer differential $\partial_{\wh J}: \vcf_\bullet (\wh H; \Lambda_R^\Gamma) \to \vcf_{\bullet -1}(\wh H; \Lambda_R^\Gamma)$ is defined by the counts $n_R ({\mf x}, {\mf y})$. More precisely, on generators,
\beqn
\partial_{\wh J} {\mf x} = \sum_{{\mf y}} n_R ({\mf x}, {\mf y}) {\mf y}.
\eeqn
One has $\partial_{\wh J}^2 = 0$, resulting in the {\bf vortex Floer homology}
\beqn
\vhf_\bullet( \wh H, \wh J; \Lambda_R^\Gamma).
\eeqn
Notice that the differential $\partial_{\wh J}$ {\it decreases} the action.

\subsubsection{Adiabatic limit}

The adiabatic limit argument allows us to relate the gauged linear sigma model with holomorphic curves in the symplectic quotient. While we do not need a complete analysis of such a correspondence, we do need to consider the family of vortex equations related to the adiabatic limit argument. Indeed, if on the infinite cylinder we choose, instead of the standard area form $ds dt$, a rescaled one $\lambda^2 ds dt$, then the corresponding vortex Floer equation reads 
\begin{align}\label{vortex_lambda}
&\ \partial_s u + {\mc X}_\xi (u) + \wh J_t \left( \partial_t u + {\mc X}_\eta(u) - X_{\wh H_t}(u) \right) = 0, &\ \partial_s \eta - \partial_t \xi + \lambda^2 \mu(u) = 0.
\end{align}
One can define a vortex Floer chain complex for the triple $(\lambda, \wh H, \wh J)$ in completely the same way as the $\lambda = 1$ case, once transversality holds, which can be achieved via perturbation. We denote the vortex Floer chain complex by $\vcf_\bullet^\lambda( \wh H, \wh J; \Lambda_R^\Gamma)$. The corresponding homology is denoted by 
\beqn
\vhf^\lambda_\bullet(\wh H, \wh J; \Lambda_R^\Gamma).
\eeqn
There are a few subtleties. First, given a nondegenerate Hamiltonian $H$ downstairs on $X$, the notion of admissible almost complex structures (\cite[Definition 6.2]{Xu_VHF}) is independent of $\lambda$; the notion of admissible lifts, however, depends on $\lambda$. 

\begin{defn}
A triple $(\lambda, \wh H, \wh J)$ is called a {\bf regular triple} if 
\begin{enumerate}
    \item The descent Hamiltonian $H$ on $X$ is nondegenerate.

    \item $(\wh H, \wh J)$ is admissible with respect to $H$.

    \item Moduli spaces of gauge equivalence classes of finite energy solutions to \eqref{vortex_lambda} are all regular.
\end{enumerate}
\end{defn}

\subsubsection{Continuation map}

Given two regular triples $(\lambda_\pm, \wh H_\pm, \wh J_\pm)$, one compares the two associated vortex Floer complexes via continuation maps. By an {\bf interpolation} between these two triples, we mean a triple $(\lambda_s, \wh H_s, \wh J_s)$ where $\lambda_s \in {\mb R}_+$ is a smooth function in $s\in {\mb R}$ which agrees with $\lambda_\pm$ near $\pm \infty$, $\wh H_s$ is a smooth family of $K$-invariant compactly supported functions parametrized by $(s, t) \in {\mb R} \times S^1$ which agrees with $\wh H_\pm$ near $\pm \infty$, and $\wh J_s$ is a smooth family of $K$-invariant $\omega_V$-compatible almost complex structures parametrized by $(s, t)\in {\mb R}\times S^1$ which agrees with $\wh J_\pm$ near $\pm \infty$. 

Choosing a generic interpolation $(\lambda_s, \wh H_s, \wh J_s)$, by considering moduli spaces of gauge equivalence classes of solutions to the equation 
\begin{align}\label{vortex_continuation}
&\ \partial_s u + {\mc X}_\xi + \wh J_{s, t} \left( \partial_t u + {\mc X}_\eta - X_{\wh H_{s, t}}(u) \right) = 0,\ &\ \partial_s \eta - \partial_t \xi + \lambda_s^2 \mu(u) = 0,
\end{align}
one can define a continuation map 
\beqn
\mf{cont}: \vcf_\bullet^{\lambda_-} (\wh H_-, \wh J_-; \Lambda_R^\Gamma) \to \vcf_\bullet^{\lambda_+} ( \wh H_+, \wh J_+; \Lambda_R^\Gamma)
\eeqn
completely analogous to the case of classical Hamiltonian Floer theory. The map $\mf{cont}$ is a chain homotopy equivalence, inducing an isomorphism on Floer homology
\beqn
\vhf_\bullet^{\lambda_-} ( \wh H_-, \wh J_-; \Lambda_R^\Gamma) \cong \vhf_\bullet^{\lambda_+} ( \wh H_+, \wh J_+; \Lambda_R^\Gamma).
\eeqn
Completely analogous to the classical situation, these isomorphisms are natural, hence the resulting homology groups define a common object called the {\bf vortex Hamiltonian Floer homology} of $V$, denoted by 
\beq
\vhf_\bullet(V; \Lambda_R^\Gamma).
\eeq
Define 
\beqn
\vhf_\bullet(V; \Lambda_R):= \vhf_\bullet( V; \Lambda_R^\Gamma) \otimes_{\Lambda_R^\Gamma} \Lambda_R.
\eeqn

In order to consider effects on the filtered theories, one needs to estimate the energy of solutions contributing to the continuation maps. 

\begin{prop}\label{prop_continuation_energy}
Given any solution ${\mf u} = (u, \xi, \eta)$ to \eqref{vortex_continuation} which converges to ${\mf x}_\pm\in {\rm crit} {\mc A}_{H_\pm}$ at $\pm \infty$, one has 
\beq\label{continuation_energy_identity}
\int_{{\mb R}\times S^1} \Big( |\partial_s u + {\mc X}_\xi(u)|^2 + \lambda_s^2 |\mu(u)|^2 \Big) ds dt
= {\mc A}_{H_-}({\mf x}_-) - {\mc A}_{H_+}({\mf x}_+) - \int_{{\mb R}\times S^1} \frac{\partial \wh H_{s, t}}{\partial s}(u) ds dt.
\eeq
In particular, if $\wh H_{s, t} = (1-\chi(s)) \wh H_- + \chi(s) \wh H_+$ for some non-decreasing function $\chi: {\mb R} \to [0, 1]$, then one has 
\beq\label{continuation_energy_inequality}
{\mc A}_{H_+}({\mf x}_+) \leq {\mc A}_{H_-}({\mf x}_-) + \int_0^1 \max_V \left( \wh H_- - \wh H_+ \right) dt.
\eeq
\end{prop}

\begin{proof}
When $\lambda_s$ is a constant, \eqref{continuation_energy_identity} is \cite[Proposition 7.5]{Xu_VHF}. The general case is the same as the area form on the domain does not affect the topological nature of the energy. 
As the left-hand-side of \eqref{continuation_energy_identity} is nonnegative, \eqref{continuation_energy_inequality} follows.
\end{proof}

\subsubsection{Computation of $\vhf$}

It is expected that the vortex Floer homology is isomorphic to the Hamiltonian Floer homology of the symplectic quotient, and hence its singular homology (in appropriate coefficients). However, such a calculation relies involved technical constructions. The Piunikhin--Salamon--Schwarz (PSS) approach forces one to deal with multiple covers of equivariant Floer cylinders with $H \equiv 0$ which may have negative equivariant Chern number. The adiabatic limit approach (similar to \cite{Gaio_Salamon_2005}) requires the study of affine vortices for a general toric manifold. In particular, for general symplectic quotients both approaches require the use of the virtual technique. 

However, in the toric case, even without having the PSS map, it is rather easy to compute the rank of $\vhf (V)$ as one can find a perfect Morse function. 

\begin{prop}\label{computation}
For any commutative ring $R$, as $\Lambda_R^\Gamma$-modules, $\vhf_\bullet (V; \Lambda_R^\Gamma)$ is isomorphic to $H_\bullet (X; \Lambda_R^\Gamma)$ (with the reduced ${\mb Z}_2$-grading) up to a degree shifting.
\end{prop}

\begin{proof}
Recall that the $2n$-dimensional toric manifold $X$ carries a Hamiltonian $T^n$-action. For a generic circle $S^1 \subset T^n$, the induced moment map $f: X \to {\mb R}$ is a perfect Morse function whose critical points are the toric fixed points. In particular, the Morse indices are all even. Then for $\epsilon$ small, $\epsilon f$ is a nondegenerate time-independent Hamiltonian. After a small perturbation and $K$-invariant lift to $V$, the corresponding vortex Floer chain complex has no two generators with adjacent degrees. Hence the $\vhf_\bullet (V; \Lambda_R^\Gamma)$ has the same rank as $H_\bullet(X; \Lambda_R^\Gamma)$. Lastly, the usual normalization of the Conley--Zehnder index is taken in such a way that if $x$ is a critical point of $\epsilon f$ viewed as a 1-periodic orbit with a constant capping, then 
\beqn
{\rm CZ}(x) = n - {\rm index}_f(x)
\eeqn
where $2n = {\rm dim} X$ and ${\rm index}_f(x)$ is the Morse index of $x$ (see \cite[(12.1.7)]{McDuff_Salamon_2004}). 
\end{proof}




\subsection{Small bulk deformations}

Here we define a family of deformations of the vortex Floer homology parametrized by ``small'' bulk deformations. Recall that the toric manifold $X$ has $N$ toric divisors $D_j$ corresponding to the $N$ faces of the moment polytope. These divisors are GIT quotients of the coordinate hyperplanes 
\beqn
V_j= \{ (x_1, \ldots, x_N) \in V\ |\ x_j = 0 \}.
\eeqn
Introduce a small bulk deformation of the form 
\beq\label{small_bulk_deformation_general}
{\mf b} = \sum_{j=1}^N c_j V_j\ {\rm where}\ c_j \in \Lambda_{0, R}.
\eeq
The ${\mf b}$-deformed vortex Floer complex is the complex generated by equivariant 1-periodic orbits upstairs whose differential counts gauge equivalence classes of solutions to the vortex equation in a different way: for each rigid (modulo gauge transformation) solution ${\mf u} = (u, \xi, \eta)$, we weight the count by the factor 
\beqn
\exp \left( \sum_{j=1}^N c_j (u \cap V_j) \right) \in \Lambda_R
\eeqn
where $u \cap V_j$ is the intersection number between the cylinder $u$ and the divisor $V_j$. Formally, this count coincides with the count of solutions on the cylinder with markings mapped to $V_j$.

\begin{rem}
The use of bulk deformations in Lagrangian Floer theory was invented by Fukaya--Oh--Ohta--Ono \cite{FOOO_toric_2, FOOO_mirror} which resembles the notion of {\it big quantum cohomology} in Gromov--Witten theory. Bulk deformations are adapted to Hamiltonian Floer theory in \cite{Usher_2011} and \cite{FOOO_spectral}. In gauged linear sigma model it was discussed in \cite{Woodward_toric}. The term ``small'' used here comes from the terminology in Gromov--Witten theory where small means deforming Gromov--Witten invariants by divisor classes and ``big'' means deforming by classes with arbitrary degrees.
\end{rem}

\subsubsection{Bulk-avoiding Hamiltonians}

One can only have a well-defined topological intersection number between Floer cylinders and the divisors if periodic orbits do not intersect these toric divisors. We introduce the following type of Hamiltonians on the toric manifold. 

\begin{defn}[Bulk-avoiding Hamiltonians] \label{bulk_avoiding}\hfill
\begin{enumerate}
\item A Hamiltonian $H$ on the toric manifold $X$ is called {\bf bulk-avoiding} if all 1-periodic orbits of $nH$ for all $n\geq 1$ do not intersect the divisor $D_1 \cup \cdots \cup D_N$. 


\item A bulk-avoiding admissible pair is an admissible pair $(\wh H, \wh J)$ such that $\wh H$ descends to a bulk-avoiding Hamiltonian downstairs.
\end{enumerate}
\end{defn}

It is easy to see that one can perturb any Hamiltonian to a bulk-avoiding one by arbitrarily $C^2$-small perturbation. Now we can define the topological intersection numbers. Let ${\mf u} = (u, \xi, \eta)$ be a solution to \eqref{vortex_equation_2} which converges to equivariant 1-periodic orbits ${\mf x}$ resp. ${\mf y}$ at $-\infty$ resp. $+\infty$. Then a generic compactly supported perturbation $\tilde u$ intersects transversely with $V_j$. Define 
\beqn
[{\mf u}] \cap V_j = \tilde u \cap V_j \in {\mb Z}
\eeqn
which counts transverse intersection points with signs. Notice that this number is well-defined: first, if $\tilde u'$ is another perturbation, then $\tilde u \cap V_j = \tilde u' \cap V_j$; second, if ${\mf u}' = (u', \xi', \eta')$ is gauge equivalent to ${\mf u}$ via a gauge transformation $g$, then $\tilde u':= g \tilde u$ is a perturbation of $u'$. As $V_j$ is $K$-invariant, $\tilde u'$ still intersect transversely with $V_j$ and the intersection number is the same. 

\subsubsection{Bulk-deformed vortex Floer complex}

For our application, we only consider small bulk deformations of the form 
\beqn
{\mf b} = \sum_{j=1}^N \log c_j\  V_j\ {\rm where}\ c_j \in {\mb Z}[{\bf i}] = {\mb Z}\oplus {\bf i}{\mb Z}.
\eeqn
Here ${\bf i} = \sqrt{-1}$ and one can regard ${\mb Z}[{\bf i}]\subset {\mb C}$. The weighted counts eventually only depend on $c_j$ so we allow $c_j$ to be zero and the ambiguity of taking logarithm does not affect further discussions. Consider the vortex Floer chain complex
\beqn
\vcf_\bullet( \wh H, \wh J; \Lambda_{{\mb Z}[{\bf i}]}).
\eeqn
Due to the special behavior of the bulk ${\mf b}$, the weighted counts of cylinders are still integral. Define the bulk-deformed vortex differential 
\beqn
\partial^{\mf b}: \vcf_\bullet (\wh H, \wh J; \Lambda_{{\mb Z}[{\bf i}]})  \to \vcf_\bullet ( \wh H, \wh J; \Lambda_{{\mb Z}[{\bf i}]})
\eeqn
by 
\beq\label{bulk_differential}
\partial^{\mf b}({\mf x}) = \sum_{{\mf y}\atop {\rm CZ}({\mf x})- {\rm CZ}({\mf y}) = 1} \left( \sum_{[{\mf u}]\in {\mc M}^{\rm cyl}({\mf x}, {\mf y})} \epsilon([{\mf u}]) \exp \left( \sum_{j=1}^N \log c_j\ [{\mf u}] \cap V_j \right) \right) {\mf y}.
\eeq
Here $\epsilon([{\mf u}]) \in \{\pm 1\}$ is the sign of the rigid solution $[{\mf u}]$. In particular, when ${\mf b} = 0$, the above coincides with the original differential map $\partial$. 

\begin{lemma} 
$\partial_{\mf b}$ is a legitimate linear map and $(\partial^{\mf b})^2 = 0$.
\end{lemma}

\begin{proof}
First, as $c_j \in {\mb Z}[{\bf i}]$, the weights
\beqn
\exp \left( \sum_{j=1}^N \log c_j \ [{\mf u}] \cap V_j \right) = \prod_{j=1}^N c_j^{[{\mf u}] \cap V_j}\in {\mb Z}[{\bf i}].
\eeqn
Hence the coefficients on the right hand side of \eqref{bulk_differential} are still in ${\mb Z}[{\bf i}]$. Second, by Gromov compactness, the sum \eqref{bulk_differential} is still in the module $\vcf_\bullet ( \wh H, \wh J; \Lambda_{{\mb Z}[{\bf i}]})$. Hence $\partial^{\mf b}$ is a well-defined linear map. To prove that its square is zero, consider for each ${\mf x}$ and ${\mf z}$ with Conley--Zehnder indices differing by $2$ and consider the $1$-dimensional components of the moduli space $\ov{\mc M}{}^{\rm cyl}({\mf x}, {\mf z})$. It can be further decomposed into connected components. Within each connected components, the topological intersection number for each cylinder with each $V_j$ is a constant. Moreover, for the concatenation of two cylinders $[{\mf u}_1]$ and $[{\mf u}_2]$ which is in the boundary of such a component, this intersection number with $V_j$ is equal to the sum $[{\mf u}_1]\cap V_j  + [{\mf u}_2] \cap V_j$. It follows that $(\partial^{\mf b})^2 = 0$.
\end{proof}

Hence for each regular admissible bulk-avoiding pair $(\wh H, \wh J)$, one can define the ${\mf b}$-deformed vortex Floer homology by
\beqn
\vhf_\bullet^{\mf b}(\wh H, \wh J; \Lambda_{{\mb Z}[{\bf i}]}):= {\rm ker} \partial^{\mf b}/ {\rm im} \partial^{\mf b}.
\eeqn
Below we summarize its properties.

\begin{thm}[Properties of bulk-deformed vortex Floer complex]\hfill \label{thm_vhf_bulk}
\begin{enumerate}
    \item For each regular bulk-avoiding admissible pair $(\wh H, \wh J)$, the complex $\vcf_\bullet (\wh H, \wh J; \Lambda_{{\mb Z}[{\bf i}]})$ with differential $\partial^{\mf b}$ is a ${\mb Z}_2$-graded filtered Floer--Novikov complex (see Definition \eqref{defn_fn}).

    \item For each two regular admissible bulk-avoiding pairs $(\wh H_1, \wh J_1)$ and $(\wh H_2, \wh J_2)$, there is a continuation map 
    \beqn
    \mf{cont}: \vcf_\bullet^{\mf b} (\wh H_1, \wh J_1; \Lambda_{{\mb Z}[{\bf i}]} ) \to \vcf_\bullet^{\mf b} (\wh H_2, \wh J_2; \Lambda_{{\mb Z}[{\bf i}]} )
    \eeqn
    which is canonical up to chain homotopy. Hence there is a ${\mb Z}_2$-graded $\Lambda_{{\mb Z}[{\bf i}]}$-module $\vhf_\bullet^{\mf b} (V; \Lambda_{{\mb Z}[{\bf i}]})$, called the ${\mf b}$-deformed vortex Floer homology, with canonical isomorphisms
    \beqn
    \vhf_\bullet^{\mf b} (\wh H, \wh J; \Lambda_{{\mb Z}[{\bf i}]}) \cong \vhf_\bullet^{\mf b} (V; \Lambda_{{\mb Z}[{\bf i}]})
    \eeqn
    for all regular admissible bulk-avoiding pairs $(\wh H, \wh J)$.

    \item There is a linear isomorphism 
    \beqn
    \vhf_\bullet^{\mf b} (V; \Lambda_{{\mb Z}[{\bf i}]}) \cong H_\bullet(X; \Lambda_{{\mb Z}[{\bf i}]}).
    \eeqn
\end{enumerate}
\end{thm}

\subsubsection{Poincar\'e duality}


In Morse--Floer theory one can define the Poincar\'e duality on the chain-level by ``reversing'' the Morse function or the symplectic action functional. We recall this construction in the setting of vortex Floer theory. If $\wh H: S^1\times V \to {\mb R}$ is a $K$-invariant Hamiltonian, define $\wh H^{\rm op}: S^1 \times V \to {\mb R}$ by 
\beqn
\wh H^{\rm op}(t, v) = - \wh H(-t, v).
\eeqn
Then similar to the case of the ordinary Floer homology (see \cite[Section 12.3]{McDuff_Salamon_2004}), there is a one-to-one correspondence between ${\rm crit} {\mc A}_H$ and ${\rm crit} {\mc A}_{H^{\rm op}}$. More precisely, if ${\mf x} = (\wh x, \eta) \in L^K(V)$ is an equivariant 1-periodic orbit, then 
\beqn
{\mf x}^{\rm op}:= (\wh x^{\rm op}, \eta^{\rm op})\ {\rm where}\ \wh x^{\rm op}(t) = \wh x(-t),\ \eta^{\rm op}(t) = - \eta(-t)
\eeqn
solves
\beqn
\frac{d}{dt} \wh x^{\rm op} (t)  + {\mc X}_{\eta^{\rm op}(t)}(\wh x^{\rm op}(t)) - X_{\wh H^{\rm op}}(\wh x^{\rm op}(t)) = 0
\eeqn
and hence is an equivariant 1-periodic orbits for $H^{\rm op}$. The map ${\mf x} \mapsto {\mf x}^{\rm op}$ induces a one-to-one correspondence
\beqn
{\rm crit} {\mc A}_H \cong {\rm crit} {\mc A}_{H^{\rm op}}
\eeqn
with critical values and Conley--Zehnder indices reversed.

Similarly, if $\wh J_t$ is an $S^1$-family of $K$-invariant almost complex structures on $V$, then define 
\beqn
(\wh J^{\rm op})_t = \wh J_{-t}.
\eeqn
One can verify easily that if $(\wh H, \wh J)$ is admissible, so is $(\wh H^{\rm op}, \wh J^{\rm op})$. 

Now we define a Poincar\'e pairing on the vortex Floer homology. Let $(\wh H_1, \wh J_1)$ and $(\wh H_2, \wh J_2)$ be two regular bulk-avoiding admissible pairs on $V$. Consider the genus zero curve with two incoming cylindrical ends, denoted by $\Sigma_{\supset}$. Choose an area form with cylindrical ends on $\Sigma_{\supset}$. Define a $K$-invariant Hamiltonian perturbation $\wh H_{\supset}$ on $\Sigma_{\supset}$ which is equal to $\wh H_1 dt$ on the first cylindrical end and which is equal to $\wh H_2^{\rm op} dt$ on the second cylindrical end. Choose a domain-dependent $K$-invariant almost complex structure $\wh J_{\supset}$ which agrees on $\wh J_1$ on the first cylindrical end and which is equal to $\wh J_2^{\rm op}$ on the second cylindrical end. Consider the $\wh H_{\supset}$-perturbed symplectic vortex equation on $\Sigma_{\supset}$ with respect to the family of almost complex structures $\wh J_{\supset}$. Finite energy solutions converge to critical points of ${\mc A}_{H_1}$ resp. ${\mc A}_{H_2^{\rm op}}$ at the two cylindrical ends. Then given ${\mf x} \in {\rm crit}{\mc A}_{H_1}$ and ${\mf y}^{\rm op} \in {\rm crit} {\mc A}_{H_2^{\rm op}} \cong {\rm crit} {\mc A}_{H_2}$, one can obtain a well-defined count 
\beqn
{\mf n}_{\supset}^{\mf b} ({\mf x}, {\mf y}) \in {\mb Z}
\eeqn
by looking at rigid solutions.
Define a bilinear pairing 
\beqn
\langle \cdot, \cdot \rangle^{\mf b}: \vcf_\bullet^{\mf b} ( \wh H_1, \wh J_1; \Lambda_R^\Gamma) \otimes \vcf_\bullet^{\mf b} ( \wh H_2^{\rm op}, \wh J_2^{\rm op}; \Lambda_R^\Gamma) \to R
\eeqn
by 
\beqn
\langle \sum_{i=1}^\infty a_i {\mf x}_i, \sum_{j=1}^\infty b_j {\mf y}_j^{\rm op} \rangle^{\mf b}:= \sum_{i, j} a_i b_j {\mf n}_{\supset}({\mf x}_i, {\mf y}_j^{\rm op}).
\eeqn
An argument via energy inequality shows that the above form is finite and well-defined; by considering 1-dimensional moduli spaces one can show that the above pairing descends to homology
\beqn
\langle \cdot, \cdot \rangle^{\mf b}: \vhf_\bullet^{\mf b} ( \wh H_1, \wh J_1; \Lambda_R^\Gamma) \otimes \vhf_\bullet^{\mf b} ( \wh H_2, \wh J_2; \Lambda_R^\Gamma) \to R.
\eeqn
One can also show that the pairing is compatible with respect to the continuation map. Hence it induces a pairing 
\beqn
\langle \cdot, \cdot \rangle^{\mf b}: \vhf_\bullet^{\mf b} (V; \Lambda_R^\Gamma) \otimes \vhf_\bullet^{\mf b} (V; \Lambda_R^\Gamma) \to R.
\eeqn 

Now we specialize to the case when $\wh H_1 = \wh H_2 = \wh H$ and $\wh J_1 = \wh J_2 = \wh J$. In this case the pairing takes a simple form on the chain level. Indeed, if we choose $\wh H_\supset$ and $\wh J_\supset$ to be the trivial ones, then the countings $n_\supset^{\mf b} ({\mf x}, {\mf y}^{\rm op})$ is 1 if ${\mf x} = {\mf y}$ and zero otherwise. Then if  
\begin{align*}
&\ \alpha = \sum a_i {\mf x}_i \in \vcf_\bullet^{\mf b} (\wh H, \wh J; \Lambda_R^\Gamma),\ &\ \beta = \sum b_j {\mf x}_j^{\rm op} \in \vcf_\bullet^{\mf b} (\wh H^{\rm op}, \wh J^{\rm op}; \Lambda_R^\Gamma)
\end{align*}
one has 
\beqn
\langle \alpha, \beta \rangle^{\mf b} = \sum_i a_i b_i \in R.
\eeqn
This sum is finite as ${\mc A}_H({\mf x}_i) \to -\infty$ and ${\mc A}_{H^{\rm op}}({\mf x}_j^{\rm op}) = - {\mc A}_H({\mf x}_j) \to -\infty$.

\subsubsection{Pair-of-pants products}

A TQFT type construction allows us to define a multiplicative structure on the vortex Floer homology. In particular, using any volume form on the pair-of-pants with cylindrical ends, one can define the pair-of-pants product
\beqn
\ast_{\mf b}: \vhf_\bullet^{\mf b} (V; \Lambda_R^\Gamma ) \otimes \vhf_\bullet^{\mf b} ( V; \Lambda_R^\Gamma ) \to \vhf_\bullet^{\mf b} (V; \Lambda_R^\Gamma )[n]
\eeqn
which is associative. Here $2n = {\rm dim} X$. The details were given in \cite{Wu_Xu}.

There is also an identity element in the vortex Floer homology. Fix a regular bulk-avoiding admissible pair $(\wh H, \wh J)$. Consider a once-punctured sphere $\Sigma_{\rm cigar}$ which is biholomorphic to the complex plane. View the puncture as an output. Equip $\Sigma_{\rm cigar}$ with a cylindrical volume form $\nu_{\rm cigar}$ so that one has has the isometric identification 
\beqn
{\mb C} \setminus B_1 \cong [0, +\infty) \times S^1.
\eeqn
Turn on the Hamiltonian perturbation on this cylindrical end, meaning that one has a Hamiltonian perturbation 
\beqn
{\mc H} \in \Omega^1(\Sigma_{\rm cigar}, C^\infty_c (V)^K)\ {\rm s.t.}\ {\mc H}|_{[S, +\infty) \times S^1} = H_t dt\ {\rm for}\  S \gg 0.
\eeqn
Choose a domain-dependent $K$-invariant $\omega_V$-compatible almost complex structure ${\mc J}$ parametrized by $z \in \Sigma_{\rm cigar}$ such that over the cylindrical end it agrees with $J_t$. Consider the Hamiltonian perturbed symplectic vortex equation 
\begin{align*}
&\ \ov\partial_{A, {\mc H} } u = 0,\ &\ F_A + \mu(u) \nu^{\rm cigar} = 0.
\end{align*}
Each finite energy solution ${\mf u} = (A, u)$ converges to an equivariant 1-periodic orbit and hence represents an element ${\mf x} \in {\rm crit} {\mc A}_H$. Hence for each ${\mf x}$ there is a moduli space 
\beqn
{\mc M}^{\rm cigar}({\mf x}).
\eeqn
Elements in this moduli space have a uniform energy bound by $- {\mc A}_H ({\mf x}) + C$ where $C$ depends on the perturbation data on the cigar which is uniformly bounded. The virtual dimension is $n - {\rm CZ}({\mf x})$. Counting elements (with signs) of index zero moduli spaces ${\mc M}^{\rm cigar}({\mf x})$ defines an element 
\beqn
{\bm 1}_{{\mf b}, \wh H}^{\glsm} = \sum_{\mf x} {\mf n}_{\rm cigar}^{\mf b} ({\mf x}) {\mf x} \in \vcf_n^{\mf b} (\wh H, \wh J; \Lambda_R^\Gamma ).
\eeqn
Standard TQFT argument shows that ${\bm 1}_{{\mf b}, \wh H}^{\glsm}$ is closed, induces a well-defined element in $\vhf_\bullet^{\mf b} (V; \Lambda_R^\Gamma)$, and is the multiplicative identity of $\vhf_\bullet^{\mf b} (V; \Lambda_R^\Gamma)$. Denote this element by 
\beq
{\bm 1}_{\mf b}^{\glsm} \in \vhf_n^{\mf b} (V; \Lambda_R^\Gamma).
\eeq

\begin{lemma}
The element ${\bm 1}_{\mf b}^{\rm GLSM}$ is nonzero.
\end{lemma}

\begin{proof}
In the undeformed case this was proved using the closed-open map in \cite{Wu_Xu} and the fact that some Lagrangian Floer theory is nontrivial. Here as we know that the algebra $\vhf_\bullet^{\mf b} (V; \Lambda_R^\Gamma)$ is nonzero (see Lemma \ref{computation}) for any ring $R$, one must have ${\bm 1}_{\mf b}^\glsm \neq 0$.
\end{proof}

\begin{lemma}\label{lem:pairing-product}
One has 
\beqn
\langle \alpha, \beta \rangle^{\mf b} \neq 0 \Longrightarrow \langle \alpha \ast_{\mf b} \beta, {\bm 1}_{\mf b}^\glsm \rangle^{\mf b}  \neq 0.
\eeqn
\end{lemma}

\begin{proof}
This theorem follows from the standard TQFT and cobordism argument. See Figure \ref{figure1}. The details are left to the reader. 
\begin{figure}[h]
    \centering
    \includegraphics[scale = 0.5]{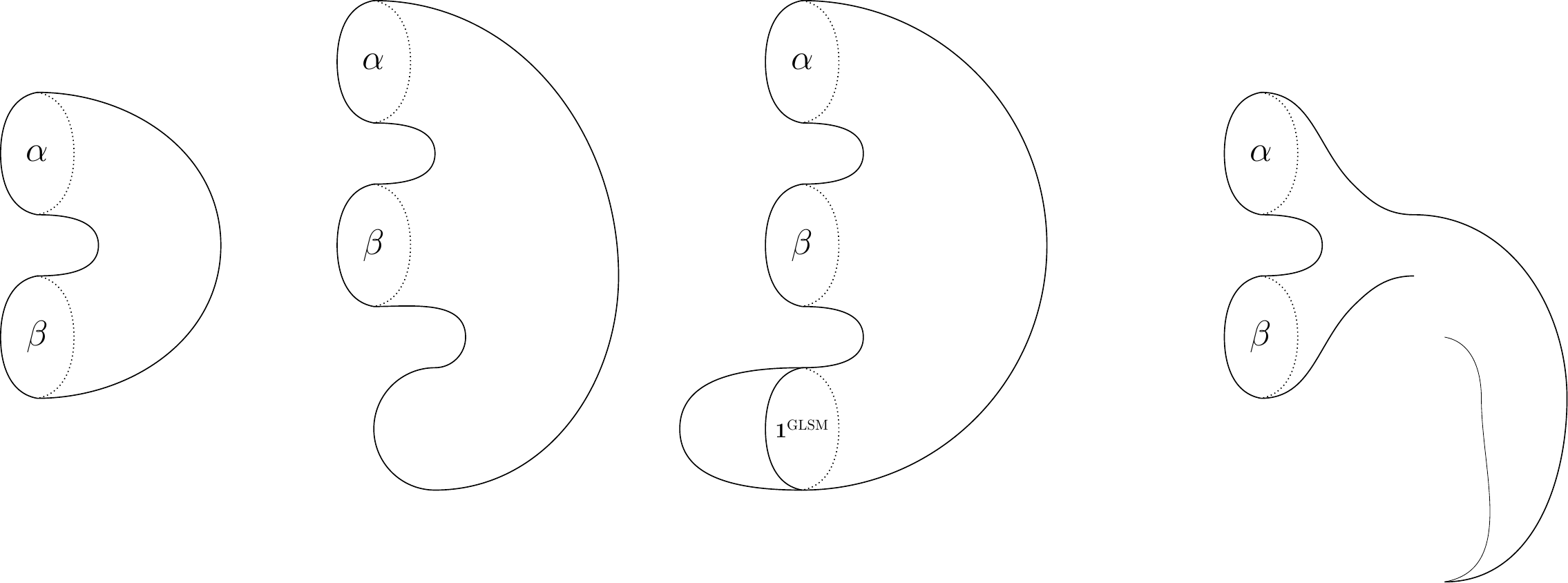}
    \caption{}
    \label{figure1}
\end{figure}
\end{proof}

Before we end this part, we state a major step towards our proof of the Hofer--Zehnder conjecture. 

\begin{thma}
There exists a bulk-deformation ${\mf b}$ of the form 
\beqn
{\mf b} = \sum_{j=1}^N \log c_j V_j
\eeqn
with $c_j \in {\mb Z}[{\bf i}]$ such that the algebra $\vhf_\bullet^{\mf b}(V; \Lambda_{\ov{\mb Q}})$ is semisimple in the sense of Definition \ref{defn_semisimple}.
\end{thma}

The proof occupies Section \ref{section9} and Section \ref{section10}, using the closed-open string map in the vortex setting.




\subsection{Bulk-deformed spectral invariants, persistence modules, and barcodes}

We fit the bulk-deformed vortex Floer theory into the abstract packages developed by Usher etc. Let ${\mf b}$ be a bulk-deformation of the form \eqref{small_bulk_deformation_general}.

\begin{prop}\label{prop_obvious}
Given a regular bulk-avoiding pair $(\wh H, \wh J)$, the quadruple 
\beqn
{\mf c}^{\mf b}(\wh H, \wh J):= ( P_H, {\mc A}_H, {\rm CZ}_{(2)}, n^{\mf b})
\eeqn
is a ${\mb Z}_2$-graded Floer--Novikov package over $R$ (see Definition \ref{defn_fn}).
\end{prop}

\begin{proof}
Straightforward.
\end{proof}

Next we consider the quantitative dependence of the vortex Floer chain complex on the Hamiltonian. We restrict to the case where $R = {\mb K}$ is a field. The vortex Floer chain complex $\vcf_\bullet^{\mf b}(\wh H, \wh J; \Lambda_{\mb K}^\Gamma)$ is the associated Floer--Novikov complex. 

\begin{prop}\label{prop_quasiequivalence_distance}
Given two regular bulk-avoiding pairs $(\wh H_1, \wh J_1)$ and $(\wh H_2, \wh J_2)$, the quasi-equivalence distance (see Definition \ref{defn_quasiequivalence}) between $\vcf_\bullet^{\mf b}(\wh H_1, \wh J_1; \Lambda_{\mb K}^\Gamma)$ and $\vcf_\bullet^{\mf b}( \wh H_2, \wh J_2; \Lambda_{\mb K}^\Gamma)$ is no greater than the Hofer distance between the induced Hamiltonians $H_1, H_2$ downstairs, i.e.
\beqn
d_Q \Big( \vcf_\bullet^{\mf b}(\wh H_1, \wh J_1; \Lambda_{\mb K}^\Gamma), \vcf_\bullet^{\mf b}( \wh H_2, \wh J_2; \Lambda_{\mb K}^\Gamma) \Big) \leq \max \Big\{ \int_0^1 \max_X (H_2 - H_1) dt,\ \int_0^1 \max_X (H_1 - H_2) dt \Big\}.
\eeqn
\end{prop}

\begin{proof}
This follows from the quantitative analysis of the continuation map. As the bulk ${\mf b}$ and the coefficient field are fixed, we drop it from notations. To show that the complex only depends on the induced Hamiltonian downstairs (measured by quasiequivalence distance), we need to introduce the parameter $\lambda$ (see \eqref{vortex_lambda}). For each regular bulk-avoiding triple $(\lambda, \wh H, \wh J)$, there is a Floer--Novikov package ${\mf c}^\lambda( \wh H, \wh J)$ defined from (${\mf b}$-deformed) counts of solutions to \eqref{vortex_lambda}. Denote the associated Floer--Novikov complex by $\vcf_\bullet^\lambda(\wh H, \wh J; \Lambda_{\mb K}^\Gamma)$ with valuation denoted by $\ell^\lambda$. 

\begin{lemma}\label{quasi-equivalence_weak}
The quasi-equivalence distance between $\vcf_\bullet^{\lambda_1}(\wh H_1, \wh J_1)$ and $\vcf_\bullet^{\lambda_2}( \wh H_2, \wh J_2)$ is bounded by 
\beqn
\wh d_{\rm Hofer}( \wh H_1, \wh H_2):= \max \left\{ \int_0^1 \max_V ( \wh H_2 - \wh H_1) dt,\ \int_0^1 \max_V (\wh H_1 - \wh H_2) \right\}.
\eeqn
\end{lemma}

\begin{proof}
Indeed, this follows from the energy calculation for the continuation maps (see Proposition \ref{prop_continuation_energy}). One can construct chain homotopy equivalences $\Phi$, $\Psi$ between these two complexes and maps $K_1$, $K_2$ as in the diagram
\beqn
\xymatrix{ \vcf_\bullet^{\lambda_1} (\wh H_1, \wh J_1)  \ar@(ul,ur)^{K_1} \ar@/^2.0pc/[r]^{\Phi} & \vcf_\bullet^{\lambda_2} ( \wh H_2, \wh J_2) \ar@/^2.0pc/[l]_{\Psi}  \ar@(ul,ur)^{K_2}}.
\eeqn
The first item of Definition \ref{defn_quasiequivalence} follows directly from \eqref{continuation_energy_inequality}. Using the same method, the second item of Definition \ref{defn_quasiequivalence} can be verified for the maps $K_1$, $K_2$.
\end{proof}

We fix the two regular bulk-avoiding pairs $(\wh H_\pm, \wh J_\pm)$. For each $\epsilon>0$, one can find a $K$-invariant cut-off function $\rho_\epsilon: V \to [0, 1]$ supported near $\mu^{-1}(0)$ such that if we define $\wh H_\pm^\epsilon:= \rho_\epsilon \wh H_\pm$, then 
\beqn
\wh d_{\rm Hofer}( \wh H_-^\epsilon, \wh H_+^\epsilon) \leq d_{\rm Hofer}(H_-, H_+) + \epsilon.
\eeqn
Hence in view of Lemma \ref{quasi-equivalence_weak} above, we only need to prove the following.

\begin{lemma}
Suppose $(\wh H_\pm, \wh J_\pm)$ are two regular bulk-avoiding pairs such that $\wh H_+$ and $\wh H_-$ descend to the same Hamiltonian $H$ downstairs. Then 
\beqn
d_Q ( \vcf_\bullet( \wh H_-, \wh J_-), \vcf_\bullet( \wh H_+, \wh J_+)) = 0.
\eeqn
\end{lemma}

\begin{proof}
We prove that the quasi-equivalence distance is less than $\epsilon$ for all $\epsilon>0$. Notice that the potential failure of this assertion comes from the difference between $\wh H_-$ and $\wh H_+$ which is {\it a priori} large outside $\mu^{-1}(0)$. We use the adiabatic limit argument to push solutions contributing to the continuation maps near the level set $\mu^{-1}(0)$. 

Choose a sequence $\lambda_i \to \infty$. For each $\lambda_i$, one can choose a $\lambda_i$-admissible lift $\wh H_\pm^{\lambda_i}$ of $H$. As the admissible condition is only about the infinitesimal behaviors of the lifts $\wh H_\pm^{\lambda_i}$ near lifts of 1-periodic orbits of $H$, we may require that 
\beqn
\| \wh H_\pm^{\lambda_i} - \wh H_\pm \|_{C^0} \leq \epsilon.
\eeqn
Hence by Lemma \ref{quasi-equivalence_weak}, one only needs to consider the quasi-equivalence 
\beqn
d_Q \left( \vcf_\bullet^{\lambda_i} ( \wh H_-^{\lambda_i}, \wh J_-^{\lambda_i}), \vcf_\bullet^{\lambda_i}( \wh H_+^{\lambda_i}, \wh J_+^{\lambda_i}) \right).
\eeqn
We claim that the above sequence (in $i$) converges to zero. 

We set up the moduli spaces for the continuation maps. Choose a non-decreasing function $\chi: {\mb R} \to [0, 1]$ which is zero on $(-\infty, 0]$ and $1$ on $[1, +\infty]$. Consider the equation with
\beqn
\wh H_{s, t}^{\lambda_i} = ( 1- \chi(s)) \wh H_-^{\lambda_i} + \chi(s) \wh H_+^{\lambda_i}.
\eeqn
We claim that, for all $\epsilon>0$, there exists $i_\epsilon>0$ such that when $i \geq i_\epsilon$, for all finite energy solutions to \eqref{vortex_continuation}, if the limit at $\pm \infty$ is ${\mf x}_\pm$, then one has 
\beqn
{\mc A}_H({\mf x}_{+, i}) - {\mc A}_H({\mf x}_{-, i}) \leq \epsilon.
\eeqn
This would establish item (1) of Definition \ref{defn_quasiequivalence}. 

Suppose on the contrary that this is not true. Then there exist $\delta>0$, a subsequence (still indexed by $i$), a sequence of solutions ${\mf u}_i = (u_i, \xi_i, \eta_i)$ to the equation connecting ${\mf x}_{-, i}$ and ${\mf x}_{+, i}$ such that 
\beqn
{\mc A}_H({\mf x}_{+, i}) - {\mc A}_H({\mf x}_{-, i}) \geq \delta > 0.
\eeqn
By the energy identity \eqref{continuation_energy_identity}, one has a uniform bound which is independent of $\lambda_i$:
\beqn
E_{\lambda_i} ({\mf u}_i) = {\mc A}_H({\mf x}_{-, i}) - {\mc A}_H({\mf x}_{+, i}) - \int_{[0, 1]\times S^1} \partial_s \wh H_{s, t}^{\lambda_i} (u) ds dt \leq C.
\eeqn

Now one can apply the adiabatic limit argument. Notice that although we cannot guarantee the convergence of $\wh H_{s, t}^{\lambda_i}$, but we may require that $\wh J_{s, t}^{\lambda_i}$ converges in sufficiently high order to a fixed almost complex structure $\wh J$ outside a compact subset of $V$. In the $\lambda_i \to \infty$ limit, {\it a priori} there are three types of bubbles (see \cite[Section 11]{Gaio_Salamon_2005}): holomorphic spheres in $V$, holomorphic spheres in $X$, and {\it affine vortices}, which are solutions to the vortex equation over ${\mb C}$ (without Hamiltonian term). The three kind of bubbles can be classified by the rate of energy concentration compared to the rate of the divergence $\lambda_i \to \infty$. As there is a lower bound on the energy of these bubbles, the uniform bound on energy implies that, after passing to a subsequence (still indexed by $i$), except near a finite subset $Z \subset [0, 1]\times S^1=:Q$ at which bubbling could occur, the energy density 
\beqn
|\partial_s u_i + {\mc X}_{\xi_i}(u_i)|^2 + \lambda_i^2 |\mu(u_i)|^2
\eeqn
stays bounded. In particular, the map $u_i|_Q$ stays arbitrarily close to $\mu^{-1}(0)$ except near $Z$ as $i \to \infty$. More precisely, for any $r>0$, there exists $i_r>0$ such that for all $i \geq i_r$, 
\beq
\sup_{z\in [0, 1]\times S^1 \setminus B_r(Z)} |\mu(u_i(z))| \leq r.
\eeq
Then one has 
\begin{multline*}
{\mc A}_H({\mf x}_{+, i}) -{\mc A}_H({\mf x}_{-, i}) \leq  \int_{Q} |\partial_s \wh H_{s, t}^{\lambda_i} (u_i)| ds dt\\
\leq \int_{Q \setminus B_r(Z)} |\partial_s \wh H_{s, t}^{\lambda_i}(u_i)| ds dt + \int_{B_r(Z)} |\partial_s \wh H_{s, t}^{\lambda_i}(u_i)| ds dt.
\end{multline*}
As $\wh H_-^{\lambda_i} = \wh H_+^{\lambda_i}$ on $\mu^{-1}(0)$, the first item is bounded by $Cr$; the second term is bounded by $C{\rm Area}(B_r(Z))$ which can be arbitrarily small. This contradicts the assumption that ${\mc A}_H({\mf x}_{+, i}) - {\mc A}_H({\mf x}_{-, i}) \geq \delta > 0$. 

Therefore, we established item (1) of Definition \ref{defn_quasiequivalence}. The case of item (2) is similar and hence omitted. 
\end{proof}
Now the proof of Proposition \ref{prop_quasiequivalence_distance} is complete.
\end{proof}

\subsubsection{Spectral invariants}

Spectral numbers of Hamiltonian diffeomorphisms were introduced by Oh \cite{Oh_spectral_1}, Schwarz \cite{Schwarz_spectral} and enhanced by Entov--Polterovich \cite{Entov_Polterovich_1, Entov_Polterovich_2, Entov_Polterovich_3}. In \cite{Wu_Xu} Wu and the second author constructed the analogue in the vortex Floer theory. 

By Theorem \ref{Usher_thm} and Proposition \ref{prop_obvious}, one can define the spectral numbers
\beqn
\rho^{\mf b} (\alpha; \wh H, \wh J):= \rho_{{\mf c}^{\mf b}(\wh H, \wh J)}(\alpha)\in {\mb R} \cup \{-\infty\},\ \forall \alpha \in \vhf_\bullet^{\mf b}(V; \Lambda_{{\mb Z}[{\bf i}]}^\Gamma).
\eeqn
One can establish the following properties of these spectral numbers, which were proved in \cite{Wu_Xu} in the undeformed (${\mf b} = 0$) case.

\begin{thm}(cf. \cite[Proposition 3.6]{Wu_Xu})
The spectral numbers $\rho^{\mf b} (\alpha; \wh H, \wh J)$ have the following properties.
\begin{enumerate}

\item {\bf (Independence of lifting and almost complex structure)}
The number $\rho^{\mf b} (\alpha; \wh H, \wh J)$ only depends on the induced Hamiltonian $H$ downstairs. Denote this number by 
\beqn
c^{\mf b} (\alpha, H) \in {\mb R}.
\eeqn

\item {\bf (Homogeneity)} Given $\alpha \in \vhf(V; \Lambda_{{\mb Z}[{\bf i}]}^\Gamma)$ and $\lambda \in \Lambda_{{\mb Z}[{\bf i}]}^\Gamma$, for any $H$, one has 
\beqn
c^{\mf b} (\lambda \alpha, H) = c^{\mf b} (\alpha, H) - {\mf v}(\lambda).
\eeqn
One uses this formula to extend the spectral numbers to classes in 
\beqn
\vhf_\bullet^{\mf b} (V; \Lambda_{{\mb Z}[{\bf i}]} ) = \vhf_\bullet^{\mf b} ( V; \Lambda_{{\mb Z}[{\bf i}]}^\Gamma) \otimes_{\Lambda_{{\mb Z}[{\bf i}]}^\Gamma} \Lambda_{{\mb Z}[{\bf i}]}.
\eeqn

\item {\bf (Lipschitz continuity)} Given any two nondegenerate Hamiltonians $H_1, H_2$ downstairs, one has 
\beqn
\int_{S^1} \min_{X} ( H_1 - H_2) dt \leq  c^{\mf b} (\alpha, H_1) - c^{\mf b} (\alpha, H_2) \leq \int_{S^1} \max_X (H_1 - H_2) dt.
\eeqn
This implies that $c^{\mf b} (\alpha, H)$ is defined for all Hamiltonians.

\item {\bf (Invariance)} $c^{\mf b} (\alpha, H)$ only depends on the homotopy class of the Hamiltonian path $\tilde \phi_H$ on $X$. Let ${\rm Ham}(X)$ be the group of Hamiltonian diffeomorphisms on $X$ and let $\tham(X) \to {\rm Ham}(X)$ be the covering of homotopy classes of Hamiltonian isotopies on $X$. Then we can define 
\beqn
c^{\mf b} (\alpha, \tilde \phi ) \in {\mb R} \cup \{-\infty\}\ \forall \alpha \in \vhf_\bullet^{\mf b}(V; \Lambda_R),\ \tilde\phi \in \tham(X).
\eeqn

\item {\bf (Triangle inequality)} For any $\alpha_1, \alpha_2 \in \vhf (V; \Lambda_R)$ and $\tilde \phi_1, \tilde \phi_2\in \tham(X)$ one has 
\beqn
c^{\mf b} ( \alpha_1 \ast \alpha_2, \tilde \phi_1 \tilde \phi_2) \leq c^{\mf b} ( \alpha_1, \tilde \phi_1) + c^{\mf b} ( \alpha_2, \tilde \phi_2).
\eeqn

\end{enumerate}
\end{thm}

\begin{defn}
The {\bf valuation} of a class $\alpha \in \vhf_\bullet^{\mf b} (V; \Lambda_R)$ is defined to be 
\beqn
{\mc A}^{\mf b} (\alpha):= c^{\mf b} (\alpha, \tilde {\rm Id})\in {\mb R} \cup \{+\infty\}.
\eeqn
\end{defn}

\subsubsection{Poincar\'e duality}

One useful property of the spectral numbers is related to the Poincar\'e duality map. 

\begin{prop}\label{prop413}
Let ${\mb K}$ be a field.
\begin{enumerate}
\item For any $\alpha \in \vhf_\bullet^{\mf b}(V; \Lambda_{\mb K})$ and $\tilde \phi \in \tham(X)$, there holds
\beqn
c^{\mf b} (\alpha, \tilde \phi) = - \inf \Big\{ c^{\mf b} ( \beta, \tilde \phi^{-1} )\ |\ \langle \alpha, \beta \rangle^{\mf b} \neq 0 \Big\}.
\eeqn

\item If $\langle \alpha, \beta \rangle^{\mf b} \neq 0$, then 
\beqn
{\mc A}^{\mf b} (\alpha) + {\mc A}^{\mf b} (\beta)\geq 0.
\eeqn
\end{enumerate}
\end{prop}

\begin{proof}
Notice that one only needs to prove this proposition for coefficient field being $\Lambda_{\mb K}^\Gamma$. In the case of ordinary Hamiltonian Floer theory, the proof of (1) uses the PSS map and the correspondence between the pairing $\langle \cdot, \cdot \rangle$ and the intersection pairing on the singular homology of the manifold (see \cite{Entov_Polterovich_1}\cite{Ostrover_2006}\cite{FOOO_spectral}). It was pointed in \cite{Usher_2010} that (1) holds for abstract filtered Floer--Novikov complexes. As the complex $\vhf_\bullet^{\mf b} ( \wh H, \wh J; \Lambda_{\mb K}^\Gamma)$ is an abstract filtered Floer--Novikov complex over $\Lambda_{\mb K}^\Gamma$ (see Proposition \ref{prop_obvious}), (1) follows. For (2), take $\tilde \phi = {\rm Id}$. Then
\beqn
{\mc A}^{\mf b}(\alpha) = c^{\mf b} (\alpha, {\rm Id}) = - \inf \Big\{ {\mc A}^{\mf b}(\beta) \ |\ \langle \alpha, \beta \rangle^{\mf b} \neq 0 \Big\}.
\eeqn
Hence if $\langle \alpha, \beta \rangle^{\mf b} \neq 0$, ${\mc A}^{\mf b}(\beta) \geq - {\mc A}^{\mf b}(\alpha)$.
\end{proof}

\subsubsection{Persistence modules and barcodes}

Recall that (see Subsection \ref{subsection_persistence}) to any filtered Floer--Novikov complex $CF_\bullet({\mf c})$ over the Novikov field $\Lambda_{\mb K}^\Gamma$ one can associate a persistence module ${\bm V}({\mf c})$. In particular, for each regular bulk-avoiding admissible pair $(\wh H, \wh J)$, the bulk-deformed vortex Floer complex $\vcf_\bullet^{\mf b}(\wh H, \wh J; \Lambda_{\mb K}^\Gamma)$ gives a persistence module, denoted by 
\beqn
{\bm V}^{\mf b} (\wh H, \wh J; \Lambda_{\mb K}^\Gamma).
\eeqn
We omit the dependence on the bulk deformation ${\mf b}$ most of the time. One can check easily that we can extend the coefficient field to the universal Novikov field $\Lambda_{\mb K}$, obtaining a persistence module ${\bm V} (\wh H, \wh J; \Lambda_{\mb K})$ with 
\beqn
V^s(\wh H, \wh J; \Lambda_{\mb K}):= HF_\bullet^{\leq s} (\wh H, \wh J; \Lambda_{0, {\mb K}}^\Gamma) \otimes_{\Lambda_{0, {\mb K}}^\Gamma} \Lambda_{0, {\mb K}}.
\eeqn
When the ground field ${\mb K}$ is clear from the context, we often abbreviate this persistence module by ${\bm V}(\wh H, \wh J)$. One can prove, using the continuation map, that up to isomorphism, this persistence module is independent of the choice of the almost complex structure $\wh J$. Hence denote the persistence module by ${\bm V}(\wh H)$. One can also use the same idea of proving Proposition \ref{prop_quasiequivalence_distance} that, for different lifts $\wh H_1, \wh H_2$ of the same Hamiltonian $H$ downstairs, the interleaving distance between ${\bm V}(\wh H_1)$ and ${\bm V}(\wh H_2)$ is zero. By identifying persistence modules with zero interleaving distance, the persistence module only depends on the Hamiltonian path $\tilde \phi \in \tham(X)$ generated by $H$. Hence we loosely denote the object by ${\bm V}(\tilde\phi)$.

Recall also that to any Floer--Novikov complex one can associate a barcode (and hence a reduced barcode). The reduced barcode corresponding to a regular bulk-avoiding admissible pair $(\wh H, \wh J)$ is denoted by ${\mc B}(\wh H, \wh J)$. One can prove that (similar to the case of ordinary Floer barcodes, see \cite[Proposition 5.3]{Usher_2013}) the reduced barcode only depends on the time-1 map $\phi = \phi_H$ on the toric manifold $X$. Hence we also denote it by ${\mc B}(\phi)$.

\section{Local Floer theory}\label{sec:local-Floer}

To extend the Hofer--Zehnder conjecture to degenerate Hamiltonian diffeomorphisms, one needs to have a good notion of counts of fixed points. Following \cite{Shelukhin_2022}, we will use the rank of a local version of the vortex Floer homology (with bulk deformation), which is ultimately isomorphic to the local Floer homology in the classical sense, to define a homological count of fixed points. This section can be skipped at first reading, especially if the reader is mainly interested in the nondegenerate case. 

\subsection{Local vortex Floer homology with bulk}

We adapt the definition of local Floer homology in the vortex setting, possibly with bulk deformations. 
Let ${\mf b}$ be a small bulk deformation. Let $\phi: X \to X$ be a Hamiltonian diffeomorphism and $x \in X$ be an isolated fixed point. 
Let $H_t$ be a 1-periodic family of Hamiltonian on $X$ generating the Hamiltonian isotopy $\phi_t$ with $\phi_1 = \phi$. Let $x(t) = \phi_t(p)$ be the corresponding 1-periodic orbit of $H$, which may intersect the bulk divisor $D\subset X$. Choose a small perturbation $H_1$ of $H$ supported near $x(t)$ such that all nearby 1-periodic orbits are nondegenerate and are disjoint from the bulk divisor $D$. 

Let $\wh H$ be a $K$-invariant lift of $H$ and $\wh H_1$ be a $K$-invariant admissible lift of $H_1$. Then the 1-periodic orbit $x(t)$ lifts to a gauge equivalence class of equivariant 1-periodic orbits. Let ${\mf x}(t) = (x(t), \eta(t))$ be a representative. There are also gauge equivalence classes of equivariant 1-periodic orbits of $\wh H_1$ which are near ${\mf x}$. Indeed, fixing the $L_0 K$-orbit of ${\mf x}(t)$, there are well-defined $L_0 K$-orbits of equivariant 1-periodic orbits which are nearby. Then for each pair of nearby equivariant 1-periodic orbits ${\mf x}_1$, ${\mf y}_1$ of $\wh H_1$, there is a canonical homotopy class of (small) cylinders connecting ${\mf x}_1$ and ${\mf y}_1$. Consider the moduli space of solutions to the vortex equation over the cylinder connecting ${\mf x}_1$ and ${\mf y}_1$. The energy of these solutions is 
\beqn
{\mc A}_{H_1}({\mf x}_1) - {\mc A}_{H_2}({\mf y}_1)
\eeqn
which can be arbitrarily small. Then similar to the case of ordinary local Floer homology, these moduli spaces can be used to define a chain complex over any coefficient filed ${\mb K}$. As the orbits are disjoint from $D$, one can also use topological intersection numbers with the bulk divisor and associated weighted counts to define the bulk-deformed version. Denote the resulting homology by 
\beqn
\vhf^{{\mf b}, {\rm loc}} (H, x; {\mb K}).
\eeqn
The continuation argument shows that the homology is independent of the data $(\wh H_1, \wh J_1)$. On the other hand, {\it a priori} the homology depends on the bulk ${\mf b}$. When ${\mf b} = 0$, denote this homology by $\vhf_{\rm loc}(H, x; {\mb K})$.

\begin{prop}\label{prop:indep-bulk}
One has 
\beq\label{local_independent}
\vhf^{{\mf b},{\rm loc}} ( H, x; {\mb K}) \cong \vhf^{\rm loc}( H, x; {\mb K}).
\eeq
\end{prop}

\begin{proof}
First, suppose $x$ does not intersect the bulk divisor $D \subset X$. Then for any small perturbation $H_1$ of $H$, all cylinders contributing to the definition of the local Floer homology have zero topological intersection number with the divisor upstairs. Hence \eqref{local_independent} holds in this case. 

Now suppose $x$ intersects the bulk divisor $D$. One can find a loop of Hamiltonian diffeomorphisms $\psi(t)\in \ham(X)$ supported near $x(t)$ such that $y(t):= \psi(t) (x(t))$ is disjoint from $D$. Moreover, define
\beqn
y(t) = (\psi(t) \phi(t) \psi(0)^{-1})(\psi(0) (q)) = (\psi(t) \phi(t) \psi(0)^{-1})(y(0)),
\eeqn
then $y(t)$ is a 1-periodic orbit of the Hamiltonian isotopy $\psi(t) \phi(t) \psi(0)^{-1}$. Let the generating Hamiltonian function of this new family be $G$, which can be made sufficiently close to $H$. Then $y(t)$ is also an isolated 1-periodic orbit of $G$. Then a generic perturbation of $G$ also serves as a perturbation of $H$. Hence for all bulk ${\mf b}$,
\beqn
\vhf^{{\mf b}, {\rm loc}} (H, x; {\mb K}) \cong \vhf^{{\mf b}, {\rm loc}} (G, y; {\mb K}).
\eeqn
However, as $y$ is disjoint from $D$, the right hand side is isomorphic to $\vhf^{\rm loc}(G, y; {\mb K})$, which is also isomorphic to $\vhf^{\rm loc}(H, x; {\mb K})$. 
\end{proof}

Now we prove that the local vortex Floer homology is isomorphic to the local Floer homology inside the symplectic quotient.

\subsection{Isomorphism with ordinary local Floer homlogy}

We prove the following result. 

\begin{prop}\label{prop:iso-local}
There is an isomorphism
\beqn
\vhf^{\rm loc}(H, x; {\mb K}) \cong HF^{\rm loc}(H, x; {\mb K}).
\eeqn
\end{prop}

\begin{proof}
It follows from the adiabatic limit argument in the same spirit as in \cite{Dostoglou_Salamon, Gaio_Salamon_2005} and \cite{Lagrange_multiplier}. Let $H_1$ be a nondegenerate Hamiltonian on $X$ which is arbitrarily close to $H$. Let $J_1$ be an $S^1$-family of almost complex structures on $X$ such that the local Floer homology at the fixed point $x$ can be defined via the Floer equation for $(H_1, J_1)$. Let $\wh H_1$ be an admissible lift and $\wh J_1$ be a lift of $J_1$. Consider the local vortex Floer homology defined by critical points of ${\mc A}_{H_1}$ which are close to the fixed point $x \in {\rm Fix}(\phi)$ whose differential counts rigid solutions to the equation \eqref{vortex_lambda} (with $(\wh H, \wh J)$ replaced by $(\wh H_1, \wh J_1)$). For all $\lambda$, we need to perturb $\wh J_1$ to $\wh J_\lambda$ to achieve transversality. We assume that as $\lambda \to +\infty$, $\wh J_\lambda$ converges to $\wh J_1$ in $C^\infty$. Using continuation maps we can show that the resulting homology is independent of $\lambda$. Moreover, the energy of relevant solutions can be arbitrarily small. Then consider the $\lambda \to \infty$ limit. For any sequence $\lambda_i \to \infty$ and any sequence of solutions to \eqref{vortex_lambda} for $\lambda = \lambda_i$ which contributes to the local vortex Floer differential, there is an upper bound of the energy of these solutions. Then by the adiabatic limit compactness theorem (see \cite{Gaio_Salamon_2005}\cite{Wang_Xu} in similar settings) a subsequence converges to a possibly broken ordinary Floer trajectory inside $X$ modulo bubbling. As there is a lower bound for the energy of bubbles (the minimal energy of $J_1$-holomorphic spheres in $X$ and affine affine vortices in $V$), we can choose the perturbation $H_1$ sufficiently close to $H$ so that bubbles can be ruled out. Then if we are considering the zero-dimensional moduli spaces, then the possible limits must be unbroken trajectories in $X$.

Now we claim that for $\lambda$ sufficiently large, there is an orientation-preserving bijection between index zero solutions to \eqref{vortex_lambda} (modulo gauge transformation) and index zero solutions to the ordinary Floer equation in $X$. Indeed, using the same kind of estimates as in \cite{Dostoglou_Salamon}\cite{Gaio_Salamon_2005} (and the much simpler case in \cite{Lagrange_multiplier}) one can construct a gluing map from the limiting moduli space to the vortex moduli space with sufficiently large parameter $\lambda$. The compactness result explained above shows that the gluing map is surjective, while via the implicit function theorem one can show that the gluing map is injective. The fact that the gluing map preserves orientation follows from the explicit comparison of the linearized Fredholm operators (they differ by, roughly speaking, an invertible operator). 
\end{proof}

In view of Proposition \ref{prop:iso-local} and the properties of local Floer homology as proved in, e.g., \cite{Ginzburg_2010}, one can derive the following property of the local vortex Floer homology.

\begin{cor}\label{cor:property-vortex-local}
The local vortex Floer homology has the following properties.
\begin{enumerate}
\item (Up to isomorphism) $\vhf^{\rm loc}(H, x; {\mb K})$ only depends on the fixed point $x$ and the time-1 map $\phi\in \ham(X)$. Hence we denote the local vortex Floer homology by 
\beqn
\vhf^{\rm loc} (\phi, x; {\mb K}).
\eeqn

\item If $\phi^k$ is an admissible iteration of $\phi$ at $x$, then 
\beqn
\vhf^{\rm loc} ( \phi, x; {\mb K}) \cong \vhf^{\rm loc} (\phi^k, x; {\mb K}).
\eeqn
\end{enumerate}
\end{cor}

\subsection{Reduced barcodes of degenerate Hamiltonians}

Recall that one can associate to each nondegenerate Hamiltonian on a closed symplectic manifold a (finite) barcode. As this association is Lipschitz continuous with respect to the bottleneck distance for barcodes and Hofer metric for Hamiltonians, we hope one can define barcodes for all Hamiltonians using this Lipschitz continuity. However, the bottleneck distance is not complete. Therefore, {\it a priori}, the barcode for a general Hamiltonian only exists in the completion. 

\begin{thm}\label{thm_degenerate_barcode} Let ${\mb K}$ be a field. Let $\phi \in \ham(X)$ be a Hamiltonian diffeomorphism with isolated fixed points. Let ${\mc B}_{\mb K}^{\mf b} (\phi)$ be the ({\it a priori} infinite) reduced barcode of $\phi$ (in coefficient field $\Lambda_{\mb K}^\Gamma$). Then ${\mc B}_{\mb K}^{\mf b} (\phi)$ has finitely many bars whose number of end points is equal to $N(\phi; {\mb K})$. 
\end{thm}

\begin{cor}
The total bar length is defined for all $\phi\in \ham(X)$ with isolated fixed points.
\end{cor}

Now we prove Theorem \ref{thm_degenerate_barcode}. Suppose $\phi \in \ham(X)$ has only isolated fixed points. Let $H$ be a Hamiltonian whose time one map is $\phi$. Let $\wh H$ be any $K$-invariant lift of $H$ and let $\wh J$ be a $K$-invariant $\omega_V$-compatible almost complex structure. Notice that in general $(\wh H, \wh J)$ is not an admissible pair so does not have a vortex Floer complex. However, one can still consider the vortex equation with the data $(\wh H, \wh J)$. 

\begin{lemma}\label{no_short_trajectory}
There exists $\delta>0$ which only depends on $(\wh H, \wh J)$ satisfying the following condition. Let $x(t) \neq y(t)$ be two different 1-periodic orbits of $H$ downstairs. Let ${\mf u}$ be a possibly broken solution to \eqref{vortex_equation_2} which connects $x(t)$ and $y(t)$ (without conditions on capping). Then the energy of ${\mf u}$ is at least $\delta$. 
\end{lemma}

\begin{proof}
For admissible $(\wh H, \wh J)$ this statement is proved as \cite[Proposition 5.5]{Xu_VHF} using a compactness argument. Notice that to run the compactness argument and to have the notion of converging to a 1-periodic orbit, one does not really need to require that the Hamiltonian is nondegenerate or the pair $(\wh H, \wh J)$ is admissible.
\end{proof}

\begin{cor}
The lengths of all bars in ${\mc B}_{\mb K}^{\mf b} (\phi)$ are no less than $\delta$.
\end{cor}

\begin{proof}
Suppose on the contrary ${\mc B}_{\mb K}^{\mf b} (\phi)$ has a finite bar whose length $\delta'$ is positive and smaller than $\delta$. Let $(\wh H_k, \wh J_k)$ be a sequence of regular bulk-avoiding pairs such that $(\wh H_k, \wh J_k)$ converges to $(\wh H, \wh J)$. Consider the reduced barcode associated to $\phi_{H_k}$. By the continuous dependence of barcodes on the Hamiltonian, for $k$ sufficiently large, there exists a finite bar in ${\mc B}(\phi_{H_k})$ whose length is between $\frac{\delta'}{2}$ and $\delta - \epsilon$ for some small $\epsilon$. By the definition of barcodes by Usher--Zhang, there exists a rigid solution ${\mf u}_k$ to \eqref{vortex_equation_2} with data $(\wh H_k, \wh J_k)$ whose energy is between $\frac{\delta'}{2}$ and $\delta-\epsilon$. Via the compactness argument, there is a subsequence, still indexed by $k$, such that ${\mf u}_k$ converges to a possibly broken trajectory with data $(\wh H, \wh J)$ whose total energy is between $\frac{\delta'}{2}$ and $\delta -\epsilon$. This contradicts Lemma \ref{no_short_trajectory}.
\end{proof}

\begin{proof}[Proof of Theorem \ref{thm_degenerate_barcode}]
Choose a sequence of regular bulk-avoiding admissible pair $(\wh H_k, \wh J_k)$ converging to $(\wh H, \wh J)$. Notice that one can choose the sequence such that the number of fixed points of $\phi_{H_k}$ is bounded. Consider the complex $\vcf_\bullet^{\mf b}(\wh H_k, \wh J_k; \Lambda_{\mb K}^\Gamma)$. One can write 
\beqn
\partial = \partial_{\rm short} + \partial_{\rm long}
\eeqn
where $\partial_{\rm short}$ counts rigid trajectories whose energy is smaller than $\delta$ and $\partial_{\rm long}$ counts rigid trajectories whose energy is bigger than $\delta$. Then $\partial_{\rm short}^2 = 0$ and its homology coincides with the direct sum of all local vortex Floer homology of $\phi$. Moreover, one can decompose the reduced barcode of $\phi_{H_k}$ as 
\beqn
{\mc B}_{\mb K}^{\mf b} (\phi_{H_k}) = {\mc B}_{\rm short}( \phi_{H_k}) \sqcup {\mc B}_{\rm long}( \phi_{H_k}) \sqcup {\mc B}_\infty(\phi_{H_k})
\eeqn
where the first component consists of finite bars of lengths at most $\delta$ and the second component consists of other finite bars. As $\partial_{\rm short}^2 = 0$, one can also define a barcode ${\mc B}_{\rm local}(\phi_{H_k})$ by modifying the definition of Usher--Zhang, whose finite part coincides with ${\mc B}_{\rm short}(\phi_{H_k})$. Then by the definition, 
\beqn
\begin{split}
 N(\phi_{H_k}; {\mb K}) = &\ \# {\rm End}({\mc B}_{\rm short}(\phi_{H_k})) + \sum_{x \in {\rm Fix}(\phi)} {\rm dim} \vhf^{\rm loc}(\phi, x)\\
= &\ \# {\rm End}({\mc B}_{\rm short}(\phi_{H_k})) + \# {\rm End}({\mc B}_{\rm long}(\phi_{H_k})) + {\rm dim} \vhf_\bullet^{\mf b} (V).
\end{split}
\eeqn
As in the limit, all short bars disappear and long bars survive with respect to the bottleneck distance, the theorem follows. 
\end{proof}

\section{Boundary depth}\label{sec:beta}

In this section we prove Theorem \ref{thmc}, namely, under the semisimple condition, the boundary depth of the vortex Floer complex of any Hamiltonian diffeomorphism is uniformly bounded from above. 

\subsection{Vortex Floer persistence modules}

Recall that from Section \ref{subsubsec:floer-persistent} we know that any Floer--Novikov complex over a Novikov field $\Lambda_{\mb K}^\Gamma$ induces a persistence module over ${\mb K}$. Given a regular bulk-avoiding admissible pair $(\wh H, \wh J)$ and a bulk deformation
\beqn
{\mf b} = \sum_{j=1}^N \log c_j V_j\ {\rm where}\ c_j \in {\mb Z}[{\bf i}],
\eeqn
the persistence module induced from the complex $\vcf_\bullet^{\mf b}(\wh H, \wh J; \Lambda_{\ov{\mb F}_p}^\Gamma)$ is denoted by 
\beqn
{\bm V}_{(p)} (\wh H, \wh J).
\eeqn
Recall that each filtered Floer--Novikov complex has a finite boundary depth which coincides with the boundary depth of the associated persistence module. We denote the boundary depth of ${\bm V}_{(p)}(\wh H, \wh J)$ by 
\beqn
\beta_{(p)}( \wh H, \wh J)\in [0, +\infty),
\eeqn
which is equal to the length of the longest finite bar in the associated barcode (Proposition \ref{prop:beta-equal}).

\begin{prop}\label{prop_boundary_depth_continuity}
Given any two regular bulk-avoiding admissible pairs $(\wh H_1, \wh J_1)$ and $(\wh H_2, \wh J_2)$, for any prime $p$, one has 
\beq\label{boundary_depth_continuity}
| \beta_{(p)}( \wh H_1, \wh J_1) - \beta_{(p)} (\wh H_2, \wh J_2)| \leq 2 d_{\rm Hofer} (H_1, H_2).
\eeq
In particular, the boundary depth only depends on the descent Hamiltonian downstairs.
\end{prop}

\begin{proof}
This is a consequence of the stability of the persistence module and the boundary depth. Indeed, Proposition \ref{prop_quasiequivalence_distance} implies that the quasi-equivalence distance between $\vcf_\bullet^{\mf b}(\wh H_1, \wh J_1; \Lambda_{\ov{\mb F}_p})$ and $\vcf_\bullet^{\mf b}(\wh H_2, \wh J_2; \Lambda_{\ov{\mb F}_p})$ is at most equal to the Hofer distance $d_{\rm Hofer}(H_1, H_2)$. Using Theorem \ref{thm:bottle-neck-Hofer}, it implies that the interleaving distance between the two associated persistence modules is no greater than the same bound. By Proposition \ref{prop_boundary_depth_stability}, one can conclude \eqref{boundary_depth_continuity}.
\end{proof}

Using typical arguments, one can also show that the boundary depth only depends on the induced (nondegenerate) Hamiltonian isotopy $\tilde \phi_H$ on the toric manifold $X$. Then Proposition \ref{prop_boundary_depth_stability} implies that $\beta_{(p)}$ descends to a Hofer continuous function 
\beqn
\beta_{(p)}: \tham(X) \to [0, +\infty).
\eeqn

Below is the main theorem of this section.

\begin{thm}\label{thm_boundary_depth}
Suppose there exist $p_0>0$ and $C_0>0$ such that for all prime $p \geq p_0$, the algebra $\vhf_\bullet^{\mf b}(V; \Lambda_{\ov{\mb F}_p})$ is a semisimple $\Lambda_{\ov{\mb F}_p}$-algebra with idempotent generators $e_{l, (p)}, \ldots, e_{m, (p)}$ satisfying 
\beqn
\ell_p( e_{l, (p)}) \leq C_0.
\eeqn
Then there exists $C>0$ such that for all prime $p \geq p_0$ and all ${\tilde\phi} \in \tham(X)$, one has 
\beqn
\beta_{(p)} (\tilde \phi) \leq C.
\eeqn
\end{thm}

\subsection{Action by quantum multiplication}

Recall how we define pair-of-pants product on the vortex Floer homology (see \cite{Wu_Xu}). On the pair-of-pants $\Sigma^{\rm pop}$, equip the two inputs bulk-avoiding admissible pairs $(\wh H_1, \wh J_1)$ and $(\wh H_2, \wh J_2)$ and equip the output another bulk-avoiding admissible pair $(\wh H_3, \wh J_3)$. Extend these data to a domain-dependent Hamiltonian perturbation and a domain-dependent almost complex structure on $\Sigma^{\rm pop}$. By counting solutions to the Hamiltonian perturbed vortex equation on $\Sigma^{\rm pop}$ (with appropriate weights coming from the bulk deformation ${\mf b}$), one can define a chain map
\beqn
\vcf_\bullet^{\mf b}(\wh H_1, \wh J_1; \Lambda_{\mb K}) \otimes \vcf_\bullet^{\mf b} (\wh H_2, \wh J_2; \Lambda_{\mb K}) \to \vcf_\bullet^{\mf b}(\wh H_3, \wh J_3; \Lambda_{\mb K}).
\eeqn

We fix a class $\alpha \in \vhf_\bullet^{\mf b}( V; \Lambda_{\mb K})$. For each $\delta>0$, let $\wh H_\delta$ be a bulk-avoiding admissible Hamiltonian on $V$ with $\| \wh H_\delta \|_{C^2} \leq \delta$. We temporarily omit the dependence on the almost complex structure and the coefficient field from the notations. For notational simplicity, we also omit the bulk ${\mf b}$ in the formulas at the moment. Consider the chain-level map
\beqn
\vcf_\bullet (\wh H_\delta) \otimes \vcf_\bullet (\wh H) \to \vcf_\bullet (\wh H).
\eeqn
By using the energy inequality, one can show that there exists a constant $C > 0$, such that for all $s, \tau \in {\mb R}$, the above multiplication induces a bilinear map
\beq\label{chain_multiplication}
\vhf_\bullet^{\leq \tau} (\wh H_\delta) \otimes \vhf_\bullet^{\leq s}(\wh H) \to \vhf_\bullet^{\leq s + \tau + C \delta} (\wh H).
\eeq
Denote
\beqn
{\mc A}^{\mf b} (\alpha): = c^{\mf b} (\alpha, 0) = \lim_{\delta \to 0} c^{\mf b} (\alpha, H_\delta).
\eeqn
Then one has the linear map for all $\epsilon>0$, one can choose $\delta$ sufficiently small so that by setting $\tau = {\mc A}^{\mf b}(\alpha) + \delta$ and inserting a representative of $\alpha$ in $\vhf_\bullet^{\leq \tau}(\wh H_\delta)$ in \eqref{chain_multiplication}, one obtains a well-defined map
\beqn
m_\epsilon (\alpha): \vhf^{\leq s}_\bullet ( \wh H) \to \vhf_\bullet^{\leq s + {\mc A}^{\mf b} (\alpha) + \epsilon} ( \wh H).
\eeqn
Using the standard argument one can show that this map only depends on the class $\alpha$. By applying any positive shift, the above operation defines a family of operations which are recorded in the following statement.

\begin{prop}
For all $\epsilon>0$, the maps $m_{\epsilon}(\alpha)$ define a morphism of persistence modules
\beqn
m_\epsilon (\alpha): {\bm V}({\tilde\phi}) \to {\bm V}({\tilde \phi})[{\mc A}(\alpha) + \epsilon]\ \forall \epsilon>0
\eeqn
satisfying for all $\epsilon < \epsilon'$, one has 
\beqn
m_{\epsilon'}(\alpha) = {\rm shift}_{\epsilon'-\epsilon} \circ m_{\epsilon}(\alpha).
\eeqn
\end{prop}

\begin{defn}
Given $\alpha \in \vhf_\bullet^{\mf b}(V; \Lambda_{\mb K}) \setminus \{0\}$ and $\epsilon>0$, the persistence module ${\bm W}_\alpha(\tilde \phi)_{\epsilon}$ is defined by 
\beqn
W_\alpha(\tilde \phi)_{\epsilon}^s = {\rm Im} \left( m_{\epsilon}(\alpha): \vhf_\bullet^{\leq s - {\mc A}^{\mf b} (\alpha)}(\tilde H) \to \vhf_\bullet^{\leq s +\epsilon} (\tilde H) \right) \subset V(\tilde \phi)^{s+ \epsilon}.
\eeqn
\end{defn}

\begin{rem}
Our notion of persistence modules (Definition \ref{defn_persistence}) is very different from the traditionally used ones (see for example \cite{Polterovich_Shelukhin_Stojisavljevic} where similar operators were firstly defined for Floer persistence modules in the monotone case); notably we allow each piece $V^s$ of a persistence module ${\bm V}$ to be infinite-dimensional. Hence it is not straightforward, though not necessarily difficult, to prove that when $\epsilon \to 0$, the above persistence modules ``converges,'' giving a limiting object similar to the one used in \cite{Shelukhin_2022}. However, we could also carry the $\epsilon$ everywhere as we are doing here.
\end{rem}

\subsection{Proof of Theorem \ref{thm_boundary_depth}}

We prove Theorem \ref{thm_boundary_depth} following the strategy of \cite{Shelukhin_2022}. This theorem is the consequence of Lemma \ref{lemma_interleaving_bound}, Lemma \ref{lemma_split_boundary_depth}, and Lemma \ref{spectral_norm_estimate} below.

We first introduce and simplify the notations. As we work with an individual prime, we drop the dependence on the prime $p$ in most notations. Let $e_1, \ldots, e_m$ be the idempotent generators of $\vhf_\bullet^{\mf b}(V; \Lambda_{\ov{\mb F}_p})$. For each nondegenerate $\tilde \phi \in \tham(X)$, consider the direct sum persistence module
\beqn
{\bm W} (\tilde \phi)_{\epsilon} = \bigoplus_{l=1}^m {\bm W}_{e_l} (\tilde \phi)_{\epsilon}.
\eeqn

\begin{lemma}\label{lemma_interleaving_bound}
The interleaving distance between ${\bm V} (\tilde \phi)$ and ${\bm W} (\tilde \phi)_{\epsilon}$ is at most $C_0 + \epsilon$.
\end{lemma}

For all ${\tilde\phi}\in \tham(X)$, define
\beqn
\gamma ({\tilde\phi}):= \max_{1\leq l \leq m} \gamma_{e_l} ({\tilde\phi}) := \max_{1 \leq l \leq m} \left( c^{\mf b} ( e_l, {\tilde\phi}) + c^{\mf b}  ( e_l, {\tilde \phi}^{-1}) \right).
\eeqn
Temporarily let ${\rm pr}: \tham(X) \to \ham(X)$ be the canonical projection. Define for $\phi \in \ham(X)$
\beqn
\gamma (\phi):= \inf_{{\rm pr}({\tilde \phi}) = \phi} \gamma ({\tilde\phi}).
\eeqn

The following is an analogue of \cite[Proposition 12]{Shelukhin_2022}.

\begin{lemma}\label{lemma_split_boundary_depth}
The boundary depth of the persistence module ${\bm W}_{e_l} ({\tilde\phi})_{\epsilon}$ is finite. 
Moreover, given nondegenerate ${\tilde \phi}, {\tilde \psi} \in \tham(X)$, for each $l = 1, \ldots, m$, one has 
\beq\label{boundary_depth_variation}
\left| \beta({\bm W}_{e_l} ({\tilde\phi})_\epsilon ) - \beta({\bm W}_{e_l} ({\tilde\psi})_\epsilon ) \right| \leq \gamma_{e_l} ( \tilde\phi \tilde\psi^{-1}).
\eeq
\end{lemma}

The following is analogous to \cite[Proposition 5.4]{Usher_2011} and \cite[Proposition 13]{Shelukhin_2022}.

\begin{lemma}\label{spectral_norm_estimate}
For all ${\tilde \phi}\in \tham(X)$, one has
\beqn
c^{\mf b} (e_l, {\tilde\phi}) + c^{\mf b} (e_l, {\tilde\phi}^{-1}) \leq 4C_0.
\eeqn
\end{lemma}

\begin{proof}[Proof of Theorem \ref{thm_boundary_depth}]
As the boundary depth depends continuously on the Hamiltonian isotopy $\tilde \phi$, one only needs to prove the theorem for nondegenerate ones. First, by Lemma \ref{lemma_interleaving_bound}, the interleaving distance between ${\bm V} (\tilde \phi)$ and ${\bm W} (\tilde \phi)_{\epsilon}$ is bounded by $C_0$. Hence by  Proposition \ref{prop_boundary_depth_stability}, it suffices to bound the boundary depth of ${\bm W}(\tilde\phi)_{\epsilon}$. As ${\bm W} (\tilde \phi)_{\epsilon}$ is the direct sum of ${\bm W}_{e_l}(\tilde \phi)_{\epsilon}$, it suffices to bound the boundary depth of ${\bm W}_{e_l}(\tilde \phi)_{\epsilon}$ for all idempotent generators $e_l$. Then applying Lemma \ref{lemma_split_boundary_depth}, one obtains
\beqn
\beta({\bm W}_{e_l}(\tilde \phi)_{\epsilon} ) \leq \gamma_{e_l} (\tilde\phi \tilde\psi^{-1}) + \beta( {\bm W}_{e_l}(\tilde \psi)_{\epsilon}) \leq 4 C_0 + \beta( {\bm W}_{e_l}(\tilde \psi)_{\epsilon})
\eeqn
where $\tilde\psi \in \tham(X)$ is an arbitrary fixed nondegenerate Hamiltonian isotopy on $X$. Then the right hand side is finite and independent of $\tilde\phi$.
\end{proof}

\subsection{Proofs of the technical lemmas}

In this subsection we drop all dependence on the bulk deformation from notations. 

\begin{proof}[Proof of Lemma \ref{lemma_interleaving_bound}]
We construct maps between persistence modules ${\bm f}_{\epsilon}: {\bm V}(\tilde \phi) \to {\bm W}(\tilde \phi)_{\epsilon}[C_0]$ and ${\bm g}_{\epsilon}: {\bm W}(\tilde \phi)_{\epsilon} \to {\bm V}(\tilde \phi)[C_0]$ as follows. For $s \in {\mb R}$, define 
\beqn
f_{\epsilon}^s: V(\tilde\phi)^s \to \bigoplus_{l=1}^m W_{e_l}(\tilde\phi)^{s+C_0}_{\epsilon}
\eeqn
to be the composition of 
\begin{align*}
&\ V(\tilde\phi)^s \to \bigoplus_{l=1}^m W_{e_l}(\tilde \phi)_{\epsilon}^{s + {\mc A}(e_l)},\ &\ \alpha \mapsto (e_1 \ast \alpha, \ldots, e_m \ast \alpha)
\end{align*}
and the natural map
\beqn
\bigoplus_{l=1}^m W_{e_l}(\tilde\phi)_{\epsilon}^{s + {\mc A} (e_l)} \to \bigoplus_{l=1}^m W_{e_l}(\tilde\phi)_{\epsilon}^{s + C_0}.
\eeqn
Define 
\begin{align*}
&\ g_{+\epsilon}^s: \bigoplus_{l=1}^m W_{e_l}(\tilde\phi)_{\epsilon}^s \to V(\tilde \phi)^{s + C_0},\ &\ (\alpha_1, \ldots, \alpha_m) \mapsto \iota^{s + \epsilon, s+C_0} (\alpha_1 + \cdots + \alpha_m).
\end{align*}
It is straightforward to check, using the fact that $e_1 + \cdots  + e_m = {\bm 1}^{\glsm}_{\mf b}$ and that $e_l$ are idempotent generators, that ${\bm f}_{\epsilon}$, ${\bm g}_{\epsilon}$ provide $C_0$-interleaving between ${\bm V}(\tilde \phi)$ and ${\bm W}(\tilde\phi)_{\epsilon}$. 
\end{proof}

\begin{proof}[Proof of Lemma \ref{lemma_split_boundary_depth}]
The detailed proof would be almost identical to the part of the proof of \cite[Proposition 12]{Shelukhin_2022} corresponding to this lemma. Hence we only briefly sketch it. First we show the finiteness of the boundary depth. The boundary depth of ${\bm V}(\tilde\phi)$ is finite because it coincides of the boundary depth of the associated Floer--Novikov complex (see Proposition \ref{prop_boundary_depth_equivalence}). Hence by Lemma \ref{lemma_interleaving_bound} and Proposition \ref{prop_boundary_depth_stability}, ${\bm W}(\tilde\phi)_{\epsilon}$ has finite boundary depth. Therefore each summand ${\bm W}_{e_l}(\tilde\phi)_{\epsilon}$ has finite boundary depth.

Now we prove the inequality \eqref{boundary_depth_variation}. Let $F, G$ be Hamiltonians downstairs with time-1 maps $\tilde \phi$ and $\tilde \psi$ respectively. Choose bulk-avoiding admissible lifts $\wh F$, $\wh G$ upstairs and let $(\wh F, \wh J_F)$, $(\wh G, \wh J_G)$ be regular pairs. Let $\ell_F$ resp. $\ell_G$ be the non-Archimedean valuation on the complex $\vcf_\bullet ( \wh F, \wh J_F)$ resp. $\vcf_\bullet ( \wh G, \wh J_G)$. Let $\Delta_{\wh F, \wh G} = G \# \ov{F}$ be the difference Hamiltonian upstairs with descent difference Hamiltonian $\Delta_{F, G}$ downstairs. Let $\wh J_{F, G}$ be an admissible almost complex structure so that the pair $(\Delta_{\wh F, \wh G}, \wh J_{F, G})$ is regular. One can obtain a pair $(\Delta_{\wh G, \wh F}, \wh J_{G, F})$ with the roles of $\wh F$ and $\wh G$ reversed.

Now fix $\epsilon>0$. Choose a cycle $c_{\wh F, \wh G, \epsilon}\in \vcf_\bullet (\Delta_{\wh F, \wh G}, \wh J_{F, G})$ representing $e_l$ such that 
\beqn
\ell ( c_{\wh F, \wh G, \epsilon}) \leq c (e_l, \Delta_{\wh F, \wh G}) + \epsilon.
\eeqn
We also choose a cycle $c_{\wh G, \wh F, \epsilon}\in \vcf_\bullet ( \Delta_{\wh G, \wh F}, \wh J_{G, F})$ representing $e_l$ with 
\beqn
\ell( c_{\wh G, \wh F, \epsilon}) \leq c  ( e_l, \Delta_{\wh G, \wh F}) + \epsilon.
\eeqn

Now  after choosing perturbation data on the pair-of-pants, one can define a chain map 
\beqn
C_{\wh F, \wh G, \epsilon}: \vcf_\bullet ( \wh F, \wh J_F) \to \vcf_\bullet ( \wh G, \wh J_G),\ x \mapsto c_{\wh F, \wh G, \epsilon} \ast x.
\eeqn
satisfying
\beqn
\ell_G ( C_{\wh F, \wh G, \epsilon}(x)) \leq c  (e_l, \Delta_{\wh F, \wh G}) + \ell_F(x) + 2 \epsilon.
\eeqn
Similarly, by using the cycle $c_{\wh G, \wh F, \epsilon}$ one can also define a chain map 
\beqn
C_{\wh G, \wh F, \epsilon}: \vcf_\bullet ( \wh G, \wh J_G) \to \vcf_\bullet  (\wh F, \wh J_F)
\eeqn
satisfying 
\beqn
\ell_F( C_{\wh G, \wh F, \epsilon}(y)) \leq c (e_l, \Delta_{G, F}) + \ell_G(y) + 2\epsilon.
\eeqn
The lemma will follow from the following claim.

{\bf Claim.} $C_{\wh F, \wh G, \epsilon}$ and $C_{\wh G, \wh F, \epsilon}$ induce a $\frac{1}{2} \gamma_{e_l}(\tilde \phi \tilde \psi^{-1}) + 4\epsilon$-interleaving between ${\bm W}_{e_l}(\wh F)_{\epsilon}$ and ${\bm W}_{e_l}(\wh G)_{\epsilon}$.

The detailed proof would also be almost identical to that of \cite{Shelukhin_2022} except for notations. We omit the details.
\end{proof}

\begin{rem}
	As one can infer from the above proof, the inequality ``$\beta \leq \gamma$" is a consequence of studying filtered continuations maps in terms of taking the pair-of-pants product with the filtered continuation elements, which in particular does not depend on the semi-simplicity assumption.
\end{rem}

\begin{proof}[Proof of Lemma \ref{spectral_norm_estimate}]
Using Proposition \ref{prop413}, Lemma \ref{lem:pairing-product} and the triangle inequality for spectral invariants, one has
\beqn
\begin{split}
- c (e_l, {\tilde\phi}^{-1}) = &\ \inf \Big\{ c (\alpha, {\tilde \phi})\ |\ \langle e_l, \alpha \rangle \neq 0 \Big\}\\ 
\geq &\ - {\mc A} (e_l) + \inf \Big\{ c  ( e_l \ast \alpha, {\tilde \phi})\ |\ \langle e_l \ast \alpha, {\bm 1}^\glsm \rangle \neq 0 \Big\}\\
\geq &\ - {\mc A} (e_l) + \inf \Big\{ c (e_l, {\tilde\phi}) - {\mc A} ( (e_l \ast \alpha)^{-1}) \ |\ \langle e_l \ast \alpha, {\bm 1}^\glsm \rangle \neq 0 \Big\}\\
\geq &\ - {\mc A} (e_l) + c(e_l, {\tilde \phi}) + \inf \Big\{ - {\mc A} ( e_l \ast \alpha) - {\mc A} ( (e_l \ast \alpha)^{-1}) \ |\ e_l \ast \alpha \neq 0 \Big\} \\
&\ + \inf \Big\{ {\mc A} ( e_l \ast \alpha) \ |\ \langle e_l \ast \alpha, {\bm 1}^\glsm \rangle \neq 0 \Big\}
\end{split}
\eeqn
Here the quantum product and the Poincar\'e pairing are both the bulk-deformed versions. Notice that as $e_l$ is an idempotent generator, $e_l \ast \alpha = \lambda(\alpha) e_l$ and $(e_l \ast \alpha)^{-1} = \lambda(\alpha)^{-1} e_l$. Hence
\beqn
{\mc A} ( e_l \ast \alpha) + {\mc A} ( (e_l \ast \alpha)^{-1}) = 2 {\mc A} (e_l) - {\mf v}(\lambda(\alpha)) - {\mf v}(\lambda(\alpha)^{-1}) = 2 {\mc A} (e_l)
\eeqn
which is uniformly bounded. Moreover, by Proposition \ref{prop413}
\beqn
\inf \Big\{ {\mc A} ( e_l \ast \alpha) \ |\ \langle e_l \ast \alpha, {\bm 1}^\glsm \rangle \neq 0 \Big\} \geq - {\mc A} ( {\bm 1}^\glsm).
\eeqn
Therefore
\beqn
c (e_l, {\tilde\phi}) + c (e_l, {\tilde\phi}^{-1}) \leq 3 {\mc A} (e_l ) +  {\mc A} ( {\bm 1}^\glsm).
\eeqn
Lemma \ref{spectral_norm_estimate} follows by using the assumption ${\mc A} (e_l) \leq C_0$ and noticing  
\beqn
{\mc A}({\bm 1}^\glsm) = {\mc A} ( e_1 + \cdots + e_m) \leq \max_{1 \leq l \leq m } {\mc A} (e_l) \leq C_0.\qedhere
\eeqn
\end{proof}

\begin{rem}
	The above argument crucially relies on the semi-simplicity assumption, which allows us to take advantage of the feature that any nonzero element in a field summand of the quantum homology is invertible. Note that such a phenomenon is ultimately due to the abundance of rational curves in toric manifolds.
\end{rem}

\section{${\mb Z}/p$-equivariant vortex Floer theory}\label{sec:equiv}

Following \cite{Seidel_pants, Shelukhin-Zhao}, we develop ${\mb Z}/p$-equivariant Hamiltonian Floer theory in the vortex setting. Using equivariant pair of pants operations, we show that the following analogue of \cite[Theorem D]{Shelukhin_2022} about the total bar length holds in our setting.

\begin{thm}\label{thm:smith}
Let $\phi$ be a Hamiltonian diffeomorphism on the toric symplectic manifold $(X, \omega)$ with lift $\tilde\phi \in \tham (X)$. Then for any odd prime $p$, if ${\rm Fix}(\phi)$ and ${\rm Fix}(\phi^p)$ are finite, then 
    \beq
    \tau_{(p)}^{\mf b}( \tilde\phi^p) \geq p \cdot \tau_{(p)}^{\mf b}( \tilde\phi)
    \eeq
\end{thm}

Here we work over $\Lambda_{\ov{\mb F}_p}$, which is omitted from the notations above. Given the arguments from \cite[Section 6]{Shelukhin_2022}, the only missing ingredient for establishing Theorem \ref{thm:smith} is the package of ${\mb Z}/p$ Borel equivariant vortex Floer theory with bulk deformation. As demonstrated in other parts of the paper, one salient feature of vortex Floer theory is the absence of sphere bubbles due to the contractibility of symplectic vector space, which allows us to achieve transversality in many settings by only perturbing the almost complex structure. Specializing to the Borel equivariant theory, except for the necessity to deal with the symplectic vortex equations and the appearance of Novikov coefficients, our theory is quite similar to the exact setting as from the original reference \cite{Seidel_pants, Shelukhin-Zhao}, at least for bulk-avoiding Hamiltonians, which suffice for our purpose via a limiting argument. Therefore, unless there is anything special in our situation, we will be brief and refer the reader to the original references for full proofs.


In this section the bulk deformation ${\mf b}$ is fixed. All curve counts are weighted by the bulk term. We often drop it in order to shorten the notations.

\subsection{The Borel construction}

We take the following model of $E {\mb Z}/p$: the ambient space is 
\beqn
S^{\infty} := \{ (z_0, z_1, \dots) \ | \ z_k \in {\mb C} \text{ for } k \in {\mb Z}_{\geq 0},  \sum |z_k|^2 = 1, \text{ only finitely many } z_k \text{'s are nonzero}\},
\eeqn
and the group ${\mb Z} / p$ freely acts on $S^{\infty}$ by multiplying each coordinate by $p$-th roots of unity. The quotient space of $S^\infty$ under this ${\mb Z}/p$-action is a model for the classifying space $B {\mb Z}/p$. The group cohomology of ${\mb Z}/p$ over ${\mb F}_p$ is recovered as the (graded-commutative) cohomology ring
\beqn
{\it H}^*(B{\mb Z}/p; {\mb F}_p) =  {\mb F}_p \llbracket u\rrbracket \langle \theta\rangle, \mathrm{deg}(u)=2 \text{ and } \mathrm{deg}(\theta)=1.
\eeqn
For $\epsilon > 0$ sufficiently small, $E {\mb Z}/p$ admits a ${\mb Z}/p$-invariant Morse function
\beqn
\tilde{F}(z) = \sum k|z_k|^2 + \epsilon \sum \mathrm{re}(z_k^p)
\eeqn
obtained by perturbing the standard Morse--Bott function $\sum k|z_k|^2$ on $S^{\infty}$ along the critical submanifolds. The function $\tilde{F}(z)$ has the following properties:
\begin{enumerate}
\item defining the map
\beq\label{borel_translation}
\begin{aligned}
\tilde{\tau}: S^{\infty} &\to S^{\infty} \\
(z_0, z_1, \dots) &\mapsto (0, z_0, z_1, \dots),
\end{aligned}
\eeq
then we have $\tilde{F} \circ \tilde{\tau} = \tilde{F} + 1$;
\item for any $l \in {\mb Z}_{\geq 0}$, the critical points of $\tilde{F}$ obtained from perturbing the critical submanifold $\{ |z_l| = 1 \}$ of $\sum k|z_k|^2$ can be indexed by
\beqn
Z_{2l}^{0}, \dots, Z_{2l}^{p-1}, \text{ and }, Z_{2l+1}^{0}, \dots, Z_{2l+1}^{p-1},
\eeqn
where each $Z_{2l}^{i}$ has Morse index $2l$ and each $Z_{2l+1}^{i}$ has Morse index $2l+1$;

\item the sets $\{Z_{2l}^{0}, \dots, Z_{2l}^{p-1}\}$ and $\{Z_{2l+1}^{0}, \dots, Z_{2l+1}^{p-1} \}$ respectively form an ${\mb Z}/p$-orbit of the ${\mb Z}/p$-action on $S^\infty$;

\item there exists a ${\mb Z}/p$-equivariant Riemannian metric $\tilde{g}$ on $S^\infty$ such that $(\tilde{f}, \tilde{g})$ is Morse--Smale, and the differential on the corresponding Morse cochain complex is
\beqn
\begin{aligned}
Z_{2l}^{m} &\mapsto Z_{2l+1}^{m} - Z_{2l+1}^{m+1}, \\
Z_{2l+1}^{m} &\mapsto Z_{2l+2}^0 + \cdots + Z_{2l+2}^{p-1},
\end{aligned}
\eeqn
where the index $m \in {\mb Z}/p$ is read cyclically.
\end{enumerate}

\subsection{The Tate construction}

Next, we review the Tate construction for cyclic groups of prime order. Let $R$ be a unital commutative ring which is an ${\mb F}_p$-algebra (later $R$ will become $\ov{\mb F}_p$). Suppose $(\hat{C}_\bullet, d_{\hat{C}})$ is a ${\mb Z}_2$-graded chain complex defined over the Novikov ring $\Lambda_{0, R}$. Note that $\Lambda_{0, R}$ is a module over $\Lambda_{0, {\mb F}_p}$. Introduce the graded field
\beqn
{\mc K} = {\mb F}_p [u^{-1}, u\rrbracket, \ \mathrm{deg}(u)=2.
\eeqn
Then the ${\mb Z} / p$-equivariant Tate complex 
\beqn
C_{\mathrm{Tate}}({\mb Z} / p, \hat{C}_\bullet^{\otimes p})
\eeqn
is a module over $\Lambda_{0, \mc K}\langle \theta \rangle$ where $\mathrm{deg}(\theta)=1, \theta^2=0$, explicitly given by
\beqn
\hat{C}_\bullet^{\otimes p} \otimes_{\Lambda_{0, {\mb F}_p}} \Lambda_{0, \mc K} \langle \theta \rangle.
\eeqn
The differential $d_{\mathrm{Tate}}$ is $\Lambda_{0,R} \otimes_{\Lambda_{0, {\mb F}_p}} \Lambda_{0, \mc K}$-linear, such that for $x_0 \otimes \cdots \otimes x_{p-1} \in \hat{C}_\bullet^{\otimes p}$, we have
\beqn
\begin{aligned}
d_{\mathrm{Tate}}(x_0 \otimes \cdots \otimes x_{p-1}) &= d_{\hat{C}}^{\otimes p}(x_0 \otimes \cdots \otimes x_{p-1}) + \theta (id - \zeta) (x_0 \otimes \cdots \otimes x_{p-1}), \\
d_{\mathrm{Tate}}(\theta (x_0 \otimes \cdots \otimes x_{p-1})) &= - \theta d_{\hat{C}}^{\otimes p}(x_0 \otimes \cdots \otimes x_{p-1}) + u (id + \zeta + \cdots + \zeta^{p-1})(x_0 \otimes \cdots \otimes x_{p-1}),
\end{aligned}
\eeqn
in which the $\zeta$ is the automorphism on $\hat{C}_\bullet^{\otimes p}$ defined by
\beqn
x_0 \otimes \cdots \otimes x_{p-1} \mapsto (-1)^{|x_{p-1}|(|x_0| + \cdots + |x_{p-2}|)} x_{p-1} \otimes x_0 \otimes \cdots \otimes x_{p-2}.
\eeqn
In other words, the Tate complex $(C_{\mathrm{Tate}}({\mb Z} / p, \hat{C}_\bullet^{\otimes p}), d_{\mathrm{Tate}})$ is obtained from the ${\mb Z}/p$ group cohomology of the chain complex $(\hat{C}_\bullet^{\otimes p}, d_{\hat{C}}^{\otimes p})$ by inverting the equivariant parameter $u$. Denote the homology of the Tate complex by
\beqn
H_{\rm Tate}({\mb Z}/p, \hat C_\bullet^{\otimes p}).
\eeqn
The following algebraic statement will be used in establishing the localization result proved later.

\begin{lemma}\label{lemma:quasi-frob}\cite[Lemma 21]{Shelukhin_2022}
Denote the homology of $(\hat{C}_\bullet, d_{\hat{C}})$ by $\hat{H}_\bullet$. The $p$-th power map
\beq\label{eqn:p-power}
\begin{aligned}
\hat{C}_\bullet &\to \hat{C}_\bullet^{\otimes p} \\
x &\mapsto x \otimes \cdots \otimes x
\end{aligned}
\eeq
induces an isomorphism of $\Lambda_{0,R} \otimes_{\Lambda_{0, {\mb F}_p}} \Lambda_{0,\mc K}$-modules
\beqn
r_p^* (\hat{H}_\bullet \otimes_{\Lambda_{0,{\mb F}_p}} \Lambda_{0,\mc K}) \to H_{\mathrm{Tate}}({\mb Z} / p, \hat{C}_\bullet^{\otimes p})
\eeqn
where $r_p$ is the operator on $\Lambda_{0,R} \otimes_{\Lambda_{0, {\mb F}_p}} \Lambda_{0, \mc K}$ defined by mapping the Novikov variable $T$ to $T^{1/p}$.
\end{lemma}
This is referred to as the quasi-Frobenius isomorphism in \cite[Section 7]{Shelukhin_2022}.

\subsection{${\mb Z}/p$-equivariant vortex Floer theory}

Given a 1-periodic Hamiltonian $H_t$ on $X$, its $p$-th iteration is the family $H^{(p)}_t:= H_{pt}$
If $\phi: X \to X$ is the time-1 map of $H$, then the time-1 map of $H^{(p)}$ is the iteration $\phi^p$. Following \cite{Seidel_pants, Shelukhin-Zhao}, we define the ${\mb Z}/p$-equivariant vortex Hamiltonian Floer homology for $H^{(p)}$ by using the family Floer homology coming from the Borel construction. For all the Floer-theoretic constructions involving moduli spaces, we always assume that the Hamiltonians involved in the discussion are nondegenerate.

Recall that the toric divisors of $X$ are given by $D_1, \cdots, D_N$, which are obtained as the symplectic quotient of the coordinate hyperplanes $V_1, \cdots, V_N$ in the symplectic vector space $V$. As in the definition of bulk-deformed Floer homology, we assume that the Hamiltonian $H$ is bulk-avoiding; in particular, for any odd prime $p$, $1$-periodic orbits of $H$ and $H^{(p)}$ are disjoint from $V_1\cup \cdots \cup V_N$. We also assume that both $H$ and $H^{(p)}$ are nondegenerate. Let $\wh H$ be an admissible lift of $H$ and $\wh H^{(p)}$ an admissible lift of $H^{(p)}$ (see Remark \ref{rem_admissible}). 
Let $\wh J^{(p)} = \{ \wh J_t^{(p)}\}_{t\in S^1}$ be a $1$-periodic family of compatible almost complex structures on $V$ such that the pair $(\wh H^{(p)}, \wh J^{(p)})$ is admissible and the Floer chain complex $\vcf_\bullet^{\mf b}( \wh H^{(p)}, \wh J^{(p)}; \Lambda_{0, R})$ has a well-defined bulk-deformed differential $\partial^{(p)}_{\mf b}$, where ${\mf b} = \sum_{i=1}^N \log c_i\  V_i$ is a chosen bulk in which $c_i \in {\mb Z}[{\bf i}]$. Note that we work over $\Lambda_{0,R}$ instead of $\Lambda_R$, which does not introduce any further difficulty due to the fact that $\partial^{(p)}_{\mf b}$ preserves the energy filtration on $\vcf_\bullet^{\mf b} ( \wh H^{(p)}, \wh J^{(p)}; \Lambda_{0, R})$. 

To define equivariant differentials, we include more parameters from the Borel construction. We choose an $S^{\infty} = E {\mb Z}/p$ family of time-dependent compatible \emph{admissible} almost complex structures 
\beqn
\wh J_\infty^{(p)} = \{ \wh J^{(p)}_{t,z} \}_{t \in S^1, z \in S^{\infty}} 
\eeqn
satisfying the following requirements:
\begin{enumerate}
    \item Near each critical point $Z^0_{i}, i \in {\mb Z}_{\geq 0}$ of the Morse function $\tilde{F}(z)$ on $S^{\infty}$, we have $\wh J^{(p)}_{t,z} = \wh J^{(p)}_{t}$;
    
    \item Regard ${\mb Z}/p \subset S^1$. For any $m \in {\mb Z}/p$ and $z \in S^{\infty}$, there holds the equivariance relation
    \beqn
    \wh J^{(p)}_{t - m,z} = \wh J^{(p)}_{t,m \cdot z};
    \eeqn
    
    \item $\wh J^{(p)}_{t,z}$ is invariant under the translation \eqref{borel_translation}. Namely
    \beqn
    \wh J^{(p)}_{t,\tilde{\tau}(z)} = \wh J^{(p)}_{t,z}.
    \eeqn
\end{enumerate}
After making such a choice, we can write down the following version of parametrized vortex Floer equation. Let ${\mf x}_\pm = (x_\pm, \eta_\pm) \in {\rm crit}{\mc A}_{H^{(p)}}$ be a pair of equivariant $1$-periodic orbits of $H^{(p)}$ (which do not depend on the lift $\wh H^{(p)}$). Given $i \in {\mb Z}_{\geq 0}, m \in {\mb Z}/p$ and $\alpha \in \{0,1\}$, the moduli space 
\beqn
{\mc M}^{i.m}_{\alpha}({\mf x}_-, {\mf x}_+)
\eeqn
consists of gauge equivalence classes of pairs of smooth maps (the gauge transformations act on the $(u, \phi, \psi)$-component)
\beqn
(u, \phi, \psi): {\mb R}_s\times S^1_t \to V \times {\mf k} \times {\mf k}, \quad \quad \quad w: {\mb R}_s \to S^{\infty}
\eeqn
which satisfy the equations and asymptotic conditions
\beq\label{eqn:equiv-Floer}
\left\{ \begin{array}{llll} \displaystyle 
\partial_s u + {\mc X}_{\phi}(u) + \wh J_{w(s), t}^{(p)} (\partial_t u + {\mc X}_\psi (u) - X_{\wh H_t^{(p)}}(u))  = 0,\ &\ \partial_s \psi - \partial_t \phi + \mu(u) = 0, \\
\partial_s w(s) - \nabla \tilde{F}(w) = 0, \\
\displaystyle \lim_{s \to -\infty}(u(s,\cdot), \phi(s,\cdot), \psi(s,\cdot), w(s)) = (x_-, 0, \eta_-, Z_{\alpha}^0), \\
\displaystyle \lim_{s \to + \infty}(u(s,\cdot), \phi(s,\cdot), \psi(s,\cdot), w(s)) = (x_+, 0, \eta_+, Z_i^m),
\end{array} \right.
\eeq
modulo the ${\mb R}$-translation action given by
\beqn
(u(s,\cdot), \phi(s,\cdot), \psi(s,\cdot), w(s)) \mapsto (u(s+r,\cdot), \phi(s+r,\cdot), \psi(s+r,\cdot), w(s+r)), \quad \quad \quad r \in {\mb R}.
\eeqn
Because of the absence of sphere bubbles, as the capped orbits impose an upper bound on energy, the moduli space ${\mc M}^{i.m}_{\alpha}({\mf x}_-, {\mf x}_+)$ admits a Uhlenbeck--Gromov--Floer compactification $\ov{{\mc M}}{}^{i.m}_{\alpha}({\mf x}_-, {\mf x}_+)$ by adding equivalence classes of solutions the above coupled equations defined over broken configurations. On the other hand, for a generic choice of $\{ J^{(p)}_{t,z} \}_{t \in S^1, z \in S^{\infty}}$, the moduli space ${\mc M}^{i.m}_{\alpha}({\mf x}_-, {\mf x}_+)$ is transversely cut out, and the dimension of the moduli space satisfies
\beqn
\dim {\mc M}^{i.m}_{\alpha}({\mf x}_-, {\mf x}_+) = {\rm CZ}({\mf x}_-) - {\rm CZ}({\mf x}_+) + i -\alpha - 1.
\eeqn
For a more detailed discussion of these facts, the reader may consult \cite[Section 4]{Seidel_pants}, \cite[Section 6]{Shelukhin-Zhao}, whose arguments apply to our case after using the setup from \cite[Section 6]{Xu_VHF}.

After achieving transversality, for each triple $i \in {\mb Z}_{\geq 0}, m \in {\mb Z}/p$ and $\alpha \in \{0,1\}$, we can define a $\Lambda_{0,R}$-linear map $\partial_{\alpha,{\mf b}}^{i,m}$ on $\vcf_\bullet (\wh H^{(p)}, \wh J^{(p)}; \Lambda_{0, R})$ of the form
\beqn
\partial_{\alpha,{\mf b}}^{i,m}({\mf x}) = \sum_{{\mf y}\atop {\rm CZ}({\mf x}) - {\rm CZ}({\mf y}) + i -\alpha = 1} \left( \sum_{[({\mf u},w)]\in {{\mc M}^{i.m}_{\alpha}({\mf x}, {\mf y})}} \epsilon([({\mf u},w)]) \exp \left( \sum_{i=1}^N \log c_i\ [{\mf u}] \cap V_i \right) \right) {\mf y},
\eeqn
where $\epsilon([({\mf u},w)]) \in \{\pm 1\}$ is the sign of the rigid solution $[{\mf u}]$, which is well-defined due to the existence of coherent orientations, and $[{\mf u}] \cap V_i$ is defined as before, coming from the topological intersection number. We further introduce the notation
\beqn
\partial_{\alpha,{\mf b}}^{i} = \partial_{\alpha,{\mf b}}^{i,0} + \cdots + \partial_{\alpha,{\mf b}}^{i,p-1}.
\eeqn

\begin{defn}
The ${\mb Z}/p$-equivariant ${\mf b}$-deformed vortex Floer chain complex
\beqn
\vcf_\bullet^{{\mb Z}/p}( \wh H^{(p)}, \wh J^{(p)}_\infty; \Lambda_{0, R})
\eeqn
is the ${\mb Z}_2$-graded $\Lambda_{0, R}$-module given by
\beqn
\vcf_\bullet ( \wh H^{(p)}, \wh J^{(p)}; \Lambda_{0, R})\llbracket u \rrbracket \langle \theta \rangle, \mathrm{deg}(u)=2, \mathrm{deg}(\theta)=1
\eeqn
with $\Lambda_{0, R}\llbracket u \rrbracket$-linear differential
\beqn
\begin{aligned}
\partial^{(p)}_{eq, {\mf b}}({\mf x} \otimes 1) &= \sum_{i \geq 0} \partial_{0,{\mf b}}^{2i}({\mf x}) \otimes u^i + \sum_{i \geq 0} \partial_{0,{\mf b}}^{2i+1}({\mf x}) \otimes u^i \theta, \\
\partial^{(p)}_{eq, {\mf b}}({\mf x} \otimes \theta) &= \sum_{i \geq 0} \partial_{1,{\mf b}}^{2i+1}({\mf x}) \otimes u^i \theta + \sum_{i \geq 1} \partial_{1,{\mf b}}^{2i}({\mf x}) \otimes u^i.
\end{aligned}
\eeqn
\end{defn}
The statement that $(\partial^{(p)}_{eq, {\mf b}})^2 = 0$ follows from the signed count of boundaries of the compactified $1$-dimensional moduli spaces $\ov{{\mc M}}{}^{i.m}_{\alpha}({\mf x}_-, {\mf x}_+)$. The differential is well-defined over $\Lambda_R^0$ because we only perturb the almost complex structure to achieve transversality. By continuation map considerations, the resulting homology group
\beqn
\vhf_\bullet^{{\mb Z}/p}( \wh H^{(p)}, \wh J_\infty^{(p)}; \Lambda_{0, R})
\eeqn
is independent of the choice of $\wh J_\infty^{(p)}$, and it is a module over $\Lambda_{0, R} \llbracket u \rrbracket \langle \theta \rangle$. By inverting $u$, we can define
\beqn
\vcf_{\mathrm{Tate}}( \wh H^{(p)}, \wh J_\infty^{(p)}; \Lambda_{0, R}) = \vcf^{{\mb Z}/p} ( \wh H^{(p)}, \wh J^{(p)}_\infty; \Lambda_{0, R})[u^{-1}, u \rrbracket \langle \theta \rangle
\eeqn
for which the differential is the $\Lambda_{0,R}[u^{-1}, u \rrbracket$-linear extension of $\partial^{(p)}_{eq, {\mf b}}$. The homology group is written as
\beqn
\vhf_{\mathrm{Tate}}( \wh H^{(p)}; \Lambda_{0, R}),
\eeqn
which is a module over $\Lambda^{0, R}  \otimes_{\Lambda_{0, {\mb F}_p}} \Lambda^{0, {\mc K}}\langle \theta \rangle$.

Here is some explanation of the definition of the equivariant differential. By definition, the leading order term $\partial_{0, {\mf b}}^0$ 
agrees with the differential $\partial^{(p)}_{\mf b}$ on the complex $\vcf_\bullet (\wh H^{(p)}, \wh J^{(p)}; \Lambda_{0, R})$, so does $\partial^1_{1,{\mf b}}$. The space of equivariant loops $L^{K}(V)$ admits an $S^1$-action by shifting the domain parameter, and the natural inclusion ${\mb Z}/ p \subset S^1$ defines a ${\mb Z}/ p$-action on $L^{K}(V)$ such that the action functional ${\mc A}_{H^{(p)}}$ is invariant under such an action. More concretely, the reparametrization
\beqn
{\mf x}(t) = (x(t), \eta(t)) \mapsto (x(t+\frac1p), \eta(t+\frac1p))
\eeqn
generates a ${\mb Z}/p$-action on the Floer homology 
\beqn
R_{1/p}: \vhf_\bullet^{\mf b} (\wh H^{(p)}; \Lambda_{0, R}) \to \vhf_\bullet^{\mf b} ( \wh H^{(p)}; \Lambda_{0, R})
\eeqn
which is realized by the composition
\beqn
\vcf_\bullet ( \wh H^{(p)}, \wh J^{(p)}; \Lambda_{0, R}) \xrightarrow[\text{pullback}]{\sim} \vcf_\bullet ( \wh H^{(p)}, \wh J^{(p)}_{\cdot - 1/p}; \Lambda_{0, R}) \xrightarrow{\text{continuation}} \vcf_\bullet ( \wh H^{(p)}, \wh J^{(p)}; \Lambda_{0, R})
\eeqn
after passing to homology. Here $\wh J_{\cdot - \frac{1}{p}}^{(p)}$ is the $S^1$-family of almost complex structures whose value at moment $t$ is $\wh J_{t- \frac{1}{p}}^{(p)}$. The action $R_{1/p}$ generates a ${\mb Z}/p$-action on homology; we denote 
\beqn
R_{m/p}:= (R_{1/p})^m: \vhf_\bullet^{\mf b}( \wh H^{(p)}; \Lambda_{0, R}) \to \vhf_\bullet^{\mf b}( \wh H^{(p)}; \Lambda_{0, R}).
\eeqn
Then the map $\partial^1_{0,{\mf b}}$ descends to 
\beqn
id - R_{1/p}: \vhf_\bullet^{\mf b} ( \wh H^{(p)}; \Lambda_{0, R}) \to \vhf_\bullet^{\mf b}( \wh H^{(p)}; \Lambda_{0, R})
\eeqn
on homology, while the map $\partial^2_{1,{\mf b}}$ descends to $
id + R_{1/p} + \cdots + R_{(p-1)/p}$. 
The higher order terms encodes the chain homotopies realizing relations of the form $(R_{1/p})^p = id$ on homology, and higher homotopy relations.

Finally, we observe that the degree filtration on the chain complex $\vcf_\bullet^{{\mb Z}/p}(\wh  H^{(p)}, \wh J_\infty^{(p)}; \Lambda_R^0)$ induced from variables $u$ and $\theta$ is preserved by the equivariant differential $\partial^{(p)}_{eq, {\mf b}}$, and such a filtration is complete and exhaustive. Therefore, we have a spectral sequence converging to $\vhf^{{\mb Z}/p}( \wh H^{(p)}; \Lambda_{0, R})$, whose first page can be identified with $\vhf_\bullet^{\mf b} ( \wh H^{(p)}; \Lambda_{0, R} )\llbracket u \rrbracket \langle \theta \rangle$. The same holds for the Tate version, which inverts the variable $u$.

\subsection{Equivariant $p$-legged pants operations}

In this subsection, we define equivariant ``$p$-legged" pants operations on vortex Hamiltonian Floer theory, which generalizes the constructions from \cite{Seidel_pants, Shelukhin-Zhao} to our situation. We use the homological convention, so the roles of the positive and negative cylindrical ends are the opposite of those from \emph{loc. cit.}. We will continue the setup from the previous subsection, and keep using the notations $H$, $\wh H$, $H^{(p)}$, $\wh H^{(p)}$, and $\wh J_{t, z}^{(p)}$. Furthermore, we choose a $1$-parameter family of compatible almost complex structures $\wh J$ on $V$ such that $(\wh H, \wh J)$ is regular and the Floer chain complex $\vcf_\bullet^{\mf b}(\wh H, \wh J; \Lambda_{0, R})$ is well-defined.

The equivariant pants operation  is defined over a particularly designed domain. Let $\pi: S_{\mc P} \to {\mb R} \times S^1$ be the $p$-fold branched cover with unique branch point $(0,0) \in {\mb R} \times S^1$ whose ramification point has maximal ramification order. Then $S_{\mc P}$ has $p+1$ punctures, regarded as $p$ negative ends and one positive ends. Suppose $S_{\mc P}$ are equipped with cylindrical ends
\beqn
\epsilon_i^{-}: (-\infty, -1] \times S^1 \to S_{\mc P}, \quad \epsilon_i^+: [1,\infty) \times S^1_p \to S_{\mc P}, \quad i \in {\mb Z} / p, 
\eeqn
subject to the conditions
\beqn
\begin{aligned}
\pi(\epsilon_i^{-}(s,t)) &= (s,t), \quad m \cdot (\epsilon_{i}^{-}(s,t)) = \epsilon_{i+m}^{-}(s,t) \\
\pi(\epsilon_i^{+}(s,t)) &= (s,t), \quad m \cdot (\epsilon_i^{+}(s,t)) = \epsilon_{i+m}^{+}(s,t) = \epsilon_i^{+}(s,t+m), \quad \text{ for } m \in {\mb Z}/p,
\end{aligned}
\eeqn
where $S^1_p := {\mb R} / p {\mb Z}$ is the $p$-fold cover of $S^1 = {\mb R} / {\mb Z}$. Note that all $\epsilon_i^+$ are obtained from shifting from each other by certain $m \in {\mb Z}/p$.

The domain-dependent almost complex structure needs to have particular symmetry. We consider 
almost complex structures $\wh J^+_\infty$ on $V$ parametrized by $z \in S^\infty$, $t \in S^1$, and $s \geq -1$, such that:

\begin{enumerate}

\item for $s \geq 2$ and $z \in S^{\infty}$, we have $\wh J^+_{s,t,z} = \wh J^{(p)}_{t,z}$;

\item for any $m \in {\mb Z}/p$ and $z \in S^{\infty}$, there holds the equivariance relation
\beqn
\wh J^+_{s,t - \frac{m}{p},z} = \wh J^+_{s,t,m \cdot z};
\eeqn

\item $\wh J^+_{s,t,z}$ is invariant under the translation:
    \beqn
    \wh J^+_{s,t,\tilde{\tau}(z)} = \wh J^+_{s,t,z}.
    \eeqn
\end{enumerate}
Given such a choice, we further look at almost complex structures $\wh J^{-,i}_\infty$ parametrized by with $s \leq 1$, $t\in S^1$, $z \in S^\infty$, and indexed by $i \in {\mb Z}/p$ (the label of negative ends) satisfying:
\begin{enumerate}

\item for $s \leq -2$ and any $w \in S^{\infty}$, we have $\wh J^{-,i}_{s,t,z} = \wh J_t$ for any $i \in {\mb Z}/p$;

\item for any $i \in {\mb Z}/p$ and $w \in S^{\infty}$, we have the equality $\wh J^{-,i}_{s,t,z} = \wh J^{+}_{s,t,z}$ hold for $-1 \leq s \leq 1$;

\item for any $m,i \in {\mb Z}/p$ and $z \in S^{\infty}$, there holds the equivariance relation
\beqn
\wh J^{-,i}_{s,t - \frac{m}{p},z} = \wh J^{-,i+m}_{s,t,z};
\eeqn

\item $\wh J^{-,i}_{s,t,z}$ is invariant under the translation:
    \beqn
    \wh J^{-,i}_{s,t,\tilde{\tau}(z)} = \wh J^{-,i}_{s,t,z}.
    \eeqn
\end{enumerate}
If $w: {\mb R} \to S^{\infty}$ is a parametrized negative gradient flow line of $\tilde{F}$, the above data specify a family of almost complex structures $\{ \wh J^{\mc P}_{v,w} \}_{v \in S_{\mc P}}$ given by:
\begin{enumerate}
    \item $\wh J^{\mc P}_{v,w} = \pi^* \wh J^{-,i}_{s,t,w(s)} = \pi^*J^{+}_{s,t,w(s)}$ for $v \in \pi^{-1}([-1,1] \times S^1)$ and $\pi(v) = (s,t)$;
    
    \item over the negative ends, $\wh J^{\mc P}_{v,w} = \pi^* \wh J^{-,i}_{s,t,w(s)}$ if $v = \epsilon_i^-(s,t)$ for all $i = 0,1, \dots, p-1$;
    
    \item over the positive end, $\wh J^{\mc P}_{v,w} = \wh J^+_{s,t, m \cdot w(s)}$ for all $m \in {\mb Z}/p$ and $z = \epsilon_m^+(s,t)$.
\end{enumerate}
We need to further introduce a Hamiltonian perturbation term 
\beqn
\wh {\mc H}^{\mc P} \in \Omega^1(S_{\mc P}, C^{\infty}(V)^K)
\eeqn
satisfying the following conditions.
\begin{enumerate}

\item For any $i \in {\mb Z}/p$, we have $\wh {\mc H}^{\mc P} (\epsilon_i^-(s,t)) = \wh H_t \otimes dt$;

\item On the positive end, for any $i \in {\mb Z}/p$, there holds $\wh{\mc H}^{\mc P} (\epsilon_i^+(s,t)) = \wh H_{t+i}^{(p)} \otimes dt$;

\item The ${\mb Z}/p$-equivariance condition $\wh {\mc H}^{\mc P}(m \cdot v) = \wh {\mc H}^{\mc P}(v)$ holds;

\item Let ${\mc H}^{\mc P}\in \Omega^1( S_{\mc P}, C^\infty(X))$ be the induced Hamiltonian perturbation term on $X$. Then the curvature of the Hamiltonian connection ${\mc H}^{\mc P}$ on $S_{\mc P}$ is $0$.
\end{enumerate}

Consider moduli spaces of perturbed vortex equation over the surface $S_{\mc P}$. Let $P \to S_{\mc P}$ be the trivial $K$-bundle. Given ${\mf x}_+ = (x_+, \eta_+) \in {\rm crit}{\mc A}_{H^{(p)}}$ and ${\mf x}_0 = (x_0, \eta_0), \dots, {\mf x}_{p-1} = (x_{p-1}, \eta_{p-1}) \in {\rm crit}{\mc A}_H$, for any $i \in {\mb Z}_{\geq 0}, m \in {\mb Z}/p$ and $\alpha \in \{0,1\}$ we can introduce the moduli space
\beqn
{\mc M}^{i.m}_{{\mc P}, \alpha}({\mf x}_0, \dots, {\mf x}_{p-1}; {\mf x}_+)
\eeqn
which parametrizes gauge equivalence classes of pairs
\beqn
(u, A) \in C^{\infty}(S_{\mc P}, V) \times {\mc A}(P), \quad \quad \quad w: {\mb R}_s \to S^{\infty}
\eeqn
which satisfy the equations and asymptotic conditions
\beqn
\left\{ \begin{array}{llll} \displaystyle 
\ov{\partial}_{A, \wh {\mc H}^{\mc P}, \wh J^{\mc P}_{v,w}} u= 0, \quad \quad \quad *F_A + \mu(u) = 0, \\
w' (s) - \nabla \tilde{F}(w) = 0, \\
\displaystyle \lim_{s \to -\infty}(u(\epsilon_j^-(s,\cdot)), A(\epsilon_j^-(s,\cdot)), w(\epsilon_j^-(s,\cdot))) = (x_j, 0, \eta_j, Z_{\alpha}^0), \quad \forall j \in {\mb Z}/p, \\
\displaystyle \lim_{s \to \infty}(u(\epsilon_0^+(s,\cdot)), A(\epsilon_0^+(s,\cdot)), w(\epsilon_0^+(s,\cdot))) = (x_+, 0, \eta_+, Z_i^m).\end{array} \right.
\eeqn
As expected, the moduli space ${\mc M}^{i.m}_{{\mc P}, \alpha}({\mf x}_0, \dots, {\mf x}_{p-1}; {\mf x}_+)$ admits an Uhlenbeck--Gromov--Floer compactification $\ov{{\mc M}}{}^{i.m}_{{\mc P}, \alpha}({\mf x}_0, \dots, {\mf x}_{p-1}; {\mf x}_+)$, whose detailed description can be found in \cite[Section (4c)]{Seidel_pants}. For a generic choice of almost complex structures and Hamiltonian connections, the moduli space ${\mc M}^{i.m}_{{\mc P}, \alpha}({\mf x}_0, \dots, {\mf x}_{p-1}; {\mf x}_+)$ is cut out transversely, whose dimension is given by
\beqn
{\rm CZ}({\mf x}_0) + \cdots + {\rm CZ}({\mf x}_{p-1}) - {\rm CZ}({\mf x}_+) + i -\alpha.
\eeqn

We define the pants operations using the above moduli spaces. For each $i \in {\mb Z}_{\geq 0}$, $m \in {\mb Z}/p$, and $\alpha \in \{0,1\}$, define
\beqn
\begin{aligned}
{\mc P}^{i,m}_{\alpha, {\mf b}}: \vcf_\bullet ( \wh H, \wh J; \Lambda_{0, R})^{\otimes p} &\to \vcf_\bullet ( \wh H^{(p)}, \wh J^{(p)}; \Lambda_{0, R}) \\
{\mf x}_0 \otimes \cdots \otimes {\mf x}_{p-1} \mapsto \sum_{{\mf x}\atop {\rm CZ}({\mf x}_0) + \cdots + {\rm CZ}({\mf x}_{p-1}) - {\rm CZ}({\mf x}) + i -\alpha = 0} &\left( \sum_{[({\mf u},w)]\in {{\mc M}^{i.m}_{\alpha}({\mf x}_0, \dots, {\mf x}_{p-1}; {\mf x})}} \epsilon([({\mf u},w)]) \exp \left( \sum_{i=1}^N \log c_i\ [{\mf u}] \cap V_i \right) \right) {\mf x},
\end{aligned}
\eeqn
where a discussion on the sign $\epsilon([({\mf u},w)])$ can be found in \cite[Appendix A]{Shelukhin-Zhao}.
\begin{defn}
Let ${\mc P}^i_{\alpha, {\mf b}} = {\mc P}^{i,0}_{\alpha, {\mf b}} + \cdots + {\mc P}^{i,p-1}_{\alpha, {\mf b}}$. The ${\mb Z}/p$-equivariant product is defined to be
\beqn
\begin{aligned}
{\mc P}: \vcf_\bullet (\wh H,\wh J; \Lambda_{0, R})^{\otimes p}\llbracket u \rrbracket \langle \theta \rangle &\to \vcf^{{\mb Z}/p}( \wh H^{(p)}, \wh J^{(p)}_\infty; \Lambda_{0, R} ) \\
{\mc P}(- \otimes 1) &= \sum_{i \geq 0}  {\mc P}^{2i}_{0,{\mf b}} \otimes u^i + \sum_{i \geq 0} {\mc P}^{2i+1}_{0, {\mf b}} \otimes u^i \theta, \\
{\mc P}(- \otimes \theta) &= \sum_{i \geq 1} {\mc P}^{2i}_{1, {\mf b}} \otimes u^i + \sum_{i \geq 0} {\mc P}^{2i+1}_{1,{\mf b}} \otimes u^i \theta.
\end{aligned}
\eeqn
\end{defn}

We can apply the Tate construction to the complex $\vcf_\bullet( \wh H, \wh J; \Lambda_{0, R})$, obtaining the Tate complex
\beqn
(C_{\mathrm{Tate}}({\mb Z} / p, \vcf_\bullet ( \wh H, \wh J; \Lambda_{0, R})^{\otimes p}), \partial_{\mathrm{Tate}}).
\eeqn
By inverting $u$, the ${\mb Z}/p$-equivariant product ${\mc P}$ induces a map
\beqn
C_{\mathrm{Tate}}({\mb Z} / p, \vcf_\bullet ( \wh H, \wh J; \Lambda_{0, R})^{\otimes p}) \to \vcf_{\mathrm{Tate}}( \wh H^{(p)},  \wh J^{(p)}_\infty; \Lambda_{0, R}),
\eeqn
which is also denoted by ${\mc P}$. By \cite[Section 8.3]{Shelukhin-Zhao}, ${\mc P}$ defines a chain map on the Tate chain complexes. In fact, the chain map property holds without inverting $u$, but we will not need such a statement.

We also need to define a ${\mb Z}/p$-equivariant coproduct operation 
\beqn
{\mc C}: \vcf^{{\mb Z}/p}( \wh H^{(p)}, \wh J^{(p)}_\infty; \Lambda_{0, R}) \to \vcf_\bullet (\wh H, \wh J; \Lambda_{0, R})^{\otimes p}\llbracket u \rrbracket \langle \theta \rangle
\eeqn
by counting solutions to Hamiltonian-perturbed vortex equations defined over the Riemann surface $S_{\mc C}$ which is a $p$-fold branched cover of ${\mb R} \times S^1$ with unique branch point $(0,0) \in {\mb R} \times S^1$ whose ramification point has maximal ramification order, coupled with negative gradient trajectory equations of $\tilde{F}: S^{\infty} \to {\mb R}$, but this time $S_{\mc C}$ has a single negative cylindrical end and $p$ positive cylindrical ends. By inverting the $u$-variable, ${\mc C}$ induces a map
\beqn
\vcf_{\mathrm{Tate}} (\wh H^{(p)}, \wh J^{(p)}_\infty; \Lambda_{0, R}) \to C_{\mathrm{Tate}}({\mb Z} / p, \vcf_\bullet (\wh H, \wh J; \Lambda_{0, R})^{\otimes p}).
\eeqn

\subsection{Local theories and equivariant localization}
In this subsection, we sketch the ingredients necessary for the proof of the following statement:

\begin{thm}
The equivariant ${\mb Z}/p$-equivariant product on homology
\beq\label{eqn:pants-iso}
{\mc P}: {\it H}_{\rm{Tate}}({\mb Z}/p, \vcf_\bullet^{\mf b}(\wh H, \wh J; \Lambda_{0, R})^{\otimes p}) \to \vhf_{\mathrm{Tate}} ( \wh H^{(p)}, \wh J_\infty^{(p)}; \Lambda_{0, R})
    \eeq
    is an isomorphism.
\end{thm}

First, one needs an equivariant version of the local vortex Floer theory as discussed in Section \ref{sec:local-Floer}. Namely, if $x$ is a $1$-periodic orbit of the Hamiltonian $H$, and if we denote by $x^{(p)}$ the $p$-th iteration of $x$, which is necessarily a $1$-periodic orbit of $H^{(p)}$, then there is a well-defined (bulk-deformed) ${\mb Z}/p$-equivariant local Floer homology group
\beqn
\vhf^{{\mb Z}/p}_{\rm loc}(H^{(p)}, x^{(p)}; \Lambda_{0, R}),
\eeqn
which is defined by looking at contributions to $\partial^{(p)}_{eq, {\mf b}}$ from solutions to \eqref{eqn:equiv-Floer} which are contained in a $C^2$-small neighborhood of the equivariant lift of $x^{(p)}$. By inverting $u$, the Tate version is denoted by $\vhf_{\rm Tate}^{\rm loc}( H^{(p)}, x^{(p)}; \Lambda_{0, R})$. One can similarly define the local version of the ${\mb Z}/p$-equivariant product and coproduct operation localized near $x$
\beqn
\begin{aligned}
{\mc P}_{x}^{\rm loc}: H_{\rm{Tate}}({\mb Z}/p, \vcf^{\rm loc}_\bullet ( H, x; \Lambda_{0,R})^{\otimes p}) \to \vhf_{\rm{Tate}}^{\rm loc}( H^{(p)}, x^{(p)}; \Lambda_{0,R}), \\
{\mc C}_{x}^{\rm loc}: \vhf_{\rm{Tate}}^{\rm loc}( H^{(p)}, x^{(p)}; \Lambda_R^0) \to H_{\rm{Tate}}({\mb Z}/p, \vcf^{\rm loc}(H,x; \Lambda_{0, R} )^{\otimes p}).
\end{aligned}
\eeqn
Note that just as in the non-equivariant setting, equivariant local Floer theories can be defined for isolated but not necessarily nondegenerate iterations.

Second, the main result of \cite[Section 10]{Shelukhin-Zhao} shows that if $x$ and $x^{(p)}$ are nondegenerate, the composition satisfies
\beqn
{\mc C}_{x}^{\rm loc} \circ {\mc P}_{x}^{\rm loc} = (-1)^n u^{n(p-1)} \cdot id,
\eeqn
which is an isomorphism as $u$ is invertible in the ground ring of the Tate version, thus ${\mc P}_{x}^{\rm loc}$ is an isomorphism by rank considerations. The proof in \emph{loc. cit.} goes through an auxiliary operation ${\mc Z}_{x}^{\rm loc}$ satisfying ${\mc C}_{x}^{\rm loc} \circ {\mc P}_{x}^{\rm loc} = {\mc Z}_{x}^{\rm loc}$, which can be defined in our setting following \cite[Definition 10.1]{Shelukhin-Zhao}. On the other hand, the calculation ${\mc Z}_{x}^{\rm loc} = (-1)^n u^{n(p-1)} $ is based on reducing to local Morse theory by a deformation argument, which is also legitimate in the vortex setting. Then by virtue of the proof of Proposition \ref{prop:iso-local}, when the Hamiltonian $H$ is a $C^2$-small Morse function, we can match the upstairs and downstairs moduli spaces, so that the calculation also works in our setting. Note that by Proposition \ref{prop:indep-bulk}, the bulk deformation does not affect the argument.

Finally, we can write
\beqn
{\mc P} = \sum_{x} {\mc P}_{x}^{\rm loc} + O(T^{\delta}), \quad \quad \quad \delta > 0,
\eeqn
in which $x$ ranges over all $1$-periodic orbits of $H$ and $O(T^{\delta})$ denotes an operation with positive valuation. Because the local operations ${\mc P}_{x}^{\rm loc}$ are isomorphisms, and the contributions of the simple (i.e., non-iterated) $1$-periodic orbits of $H^{(p)}$ to the Tate construction are trivial, we see that ${\mc P}$ is an isomorphism over $\Lambda_{0, R} [u^{-1}, u \rrbracket = \Lambda_{0, R} \otimes_{\Lambda_{0, {\mb F}_p}} \Lambda_{0, \mc K}$. This finishes the sketch of the proof of the equivariant localization isomorphism.

\subsection{Growth of total bar length}

After demonstrating the existence of the equivariant Hamiltonian Floer package in the vortex setting, we are in the right position to prove the inequality of total bar length.

\begin{proof}[Proof of Theorem \ref{thm:smith}]
With equivariant Hamiltonian Floer theory and local Floer homology in our hands, the arguments from \cite[Section 7]{Shelukhin_2022} can easily be adapted to the current situation without much modification. Consequently, we will only provide a sketch of the proof, and refer the reader to \emph{loc. cit.} for complete arguments.

Firstly, we recall the following alternative characterization of the total bar length. Given a field ${\mb K}$, if we define the vortex Hamiltonian Floer homology over the Novikov ring
\beqn
\Lambda_{0, \mb K} := \Big\{ \sum_{i=1}^\infty a_i T^{g_i}\ |\ g_i \in {\mb R}_{\geq 0},\ a_i \in {\mb K},\  \lim_{i \to \infty} g_i = +\infty \Big\}
\eeqn
instead of its field of fractions $\Lambda_{\mb K}$, then finite bars are reflected as nontrivial \emph{torsion components}, which is the language used in \cite{FOOO-torsion}. If we denote the direct sum of the torsion components of the Floer homology $\vhf_\bullet^{\mf b}(\tilde\phi; \Lambda_{0, \mb K})$ by
\beq\label{eqn:torsion}
\Lambda_{0, \mb K} / T^{g_1} \Lambda_{0,\mb K} \oplus \cdots \oplus \Lambda_{0, \mb K} / T^{g_s} \Lambda_{0, \mb K}, \quad \quad \quad \text{ with }\ g_1 \geq \cdots \geq g_s \geq 0,
\eeq
then we can write the total bar length of $\tilde\phi$ over $\Lambda_{\mb K}$ as
\beqn
\tau_{(p)}^{\mf b}( \tilde\phi, \Lambda_{\mb K}) = g_1 + \cdots + g_s,
\eeqn
and the boundary depth is given by $g_1$, c.f. \cite[Section 4.4.4]{Shelukhin_2022}. Note that these torsion exponents correspond to \emph{verbose bar-length spectrum} in the sense of \cite{Usher_Zhang_2016}, which means that $g_i$ can be $0$, due to that fact that the Floer differential in our discussion may not strictly decrease the energy.


The claim is more easily to be proved when $\tilde\phi^p$ is nondegenerate and bulk-avoiding. We choose a generating Hamiltonian $H$ for the Hamiltonian isotopy $\tilde\phi$. The comparison between $\tau_{(p)}^{\mf b}( \tilde\phi^p)$ and $\tau_{(p)}^{\mf b}( \tilde\phi)$ is established in the following three steps.
\begin{enumerate}
    \item Using the quasi-Frobenius isomorphism from Lemma \ref{lemma:quasi-frob}, one can show that the total bar length of
    \beqn
    (C_{\mathrm{Tate}}({\mb Z} / p, \vcf_\bullet (\wh H, \wh J; \Lambda_{0, R})^{\otimes p}), \partial_{\mathrm{Tate}}),
    \eeqn
    i.e., the sum of torsion exponents of the homology group ${\it H}_{\rm{Tate}}({\mb Z}/p, \vcf_\bullet (\wh H, \wh J; \Lambda_{0, R})^{\otimes p})$, is equal to $p$ times the quantity $\tau_{(p)}^{\mf b}( \tilde\phi)$.
    
    \item By appealing to the isomorphism in Equation \ref{eqn:pants-iso} and an application of the homological perturbation lemma, it is shown in \cite[Section 7.3.1]{Shelukhin_2022} the total boundary depth, i.e., the sum of torsion exponents of ${\it H}_{\rm{Tate}}({\mb Z}/p, \vcf_\bullet (\wh H, \wh J; \Lambda_{0, R})^{\otimes p})$ agrees with that of $\vhf_{\mathrm{Tate}}( \wh H^{(p)}; \Lambda_{0, R})$.
    
    \item Using \cite[Proposition 17, Lemma 18]{Shelukhin_2022}, one can prove that the total boundary depth of $\vhf_{\mathrm{Tate}}( \wh H^{(p)}; \Lambda_{0, R})$ is bounded from above by $\tau_{(p)}^{\mf b}( \tilde\phi^p)$, which is a reminiscent of the Borel spectral sequence in the context of filtered Floer theory.
\end{enumerate}

Finally, to establish Theorem \ref{thm:smith} for $\tilde\phi$ and $\tilde\phi^p$ which are not necessarily bulk-avoiding and may admit isolated degenerate fixed points, an approximation argument and multiple applications of the homological perturbation lemma as in \cite[Section 7.4]{Shelukhin_2022} suffice.
\end{proof}

\begin{rem}
To prove Theorem \ref{thm:smith} for degenerate Hamiltonian diffeomorphisms with isolated fixed points assuming the corresponding result for nondegenerate ones, one can alternatively use the following more elementary argument. Suppose $H$ is not necessarily bulk-avoiding and may have isolated but degenerate fixed points and periodic points of period $p$. Let $H_i$ be a sequence of nondegenerate and bulk-avoiding Hamiltonians on $X$ which converges to $H$ under $C^2$-norm. We can choose the perturbations $H_i$ to be supported in an arbitrarily small neighborhood of the one period orbits of $H$ and $H^{(p)}$, over which the perturbation is modeled on $\epsilon_i f$ where $f$ is a Morse function. Then the above implies that 
\beqn
\tau_{(p)}^{\mf b}(H_i^{(p)}) \geq p \tau_{(p)}^{\mf b}( H_i).
\eeqn
Notice that the reduced barcode of $H_i$ resp. $H_i^{(p)}$ is a Cauchy sequence with respect to the bottleneck distance with a uniformly bounded number of bars. Our choice of perturbation also guarantees a uniform upper bound for the short bars. Moreover, we know that the barcode of the limit $H$ resp. $H^{(p)}$ is finite. Hence the total bar length of $H_i$ resp. $H_i^{(p)}$ converges to that of $H$ resp. $H^{(p)}$, which implies the desired result.
\end{rem}

\section{Proof of the main theorem}\label{section_proof}\label{section8}

We now can prove Theorem \ref{thm:main}. Let $\phi: X \to X$ be a Hamiltonian diffeomorphism with isolated fixed points satisfying 
\beq\label{eqn:exceeds}
N(\phi; {\mb Q}):= \sum_{x \in {\rm Fix}(\phi)} HF^{\rm loc}(\phi, x; {\mb Q}) > \sum_{i=0}^{{\rm dim} X} {\rm dim} H_i(X).
\eeq
By Proposition \ref{prop:iso-local}, one can replace the ordinary local Floer homology by the (bulk-deformed) local vortex Floer homology. One also knows from Proposition \ref{computation} and Theorem \ref{thm_vhf_bulk} that the total rank of rational homology of $X$ agrees with the rank of the bulk-deformed vortex Floer homology of $V$. Hence \eqref{eqn:exceeds} can be rewritten as 
\beqn
N(\phi; {\mb Q}) = \sum_{x \in {\rm Fix} (\phi)} {\rm dim}_{\ov{\mb Q}} \vhf^{\rm loc}(\phi, x; \ov{\mb Q}) > {\rm dim}_{\Lambda_{\mb Q}} \vhf_\bullet^{\mf b}(V; \Lambda_{\ov{\mb Q}}).
\eeqn
Because the (local) vortex Floer homology are defined over the integers (see Section \ref{sec:local-Floer}), by the universal coefficient theorem, for $p$ sufficiently large, 
\beqn
\sum_{x \in {\rm Fix}(\phi)} \dim_{\ov{\mb F}_p} \vhf^{\rm loc}(\phi, x; \ov{\mb F}_p) >  {\rm dim}_{\Lambda_{\ov{\mb F}_p}} \vhf_\bullet^{\mf b}(V; \Lambda_{\ov{\mb F}_p}).
\eeqn
Suppose on the contrary that $\phi$ only has only finitely many periodic points. Then for any sufficiently large prime $p$, for all $k \geq 1$, ${\rm Fix}(\phi^{p^k}) = {\rm Fix}(\phi)$. 
Then by Corollary \ref{cor:property-vortex-local}, one has 
\beqn
\sum_{x \in {\rm Fix}(\phi^{p^k})} {\rm dim}_{\ov{\mb F}_p} \vhf^{\rm loc}(\phi^{p^k}, x; \ov{\mb F}_p) = \sum_{x\in {\rm Fix}(\phi)} {\rm dim}_{\ov{\mb F}_p} \vhf^{\rm loc}(\phi, x; \ov{\mb F}_p) > {\rm dim}_{\Lambda_{\ov{\mb F}_p}} \vhf_\bullet^{\mf b}(V; \Lambda_{\ov{\mb F}_p}).
\eeqn
Consider the barcode of $\phi^{p^k}$ coming from the bulk-deformed vortex Floer theory (over the Novikov field $\Lambda_{\ov{\mb F}_p}$). The above implies that the number of finite bars is positive and independent of the iteration $p^k$. The uniform bound on the boundary depth (length of the longest finite bar) given by Theorem \ref{thmc} implies that the total bar length $\tau^{\mf b}_{(p)}(\phi^{p^k})$ is uniformly bounded. 

On the other hand, by Theorem \ref{thm_total_length_growth}, for any $k \geq 1$, the total bar length grows as
\beqn
\tau^{\mf b}_{(p)} (\phi^{p^k}) \geq p^k \cdot \tau^{\mf b}_{(p)} (\phi) \geq C p^k > 0.
\eeqn
This is a contradiction. Hence $\phi$ must have infinitely many periodic points. 

\begin{rem}
Because the above argument works for any $p \geq p_0$, we know that the number of periodic points of $\phi$ grows like $\frac{k}{\log(k)}$ as $k \to \infty$, as a result of the prime number theorem.
\end{rem}

\begin{rem}
    Arguments of the above form first appeared in \cite[Section 8]{Shelukhin_2022}, which we reproduce in our context for completeness. As noted above, Shelukhin's result on the Hofer--Zehnder conjecture relies on the assumptions that the ambient symplectic manifold is monotone and that the quantum homology is semisimple, which respectively account for the inequalities \eqref{eqn:linear-grow} (the monotonicity condition allows one to define Floer theory integrally using classical methods) and \eqref{eqn:uniform-bound} (which will be discussed in more detail in the body part of this paper). For general toric symplectic manifolds, traditional Hamiltonian Floer homology is only defined over the rationals, which sets difficulties for establishing symplectic Smith-type inequalities. Moreover, the quantum homology of toric symplectic manifolds fails to be semisimple in general, which is already the case even for Fano/monotone toric manifolds \cite{Ostrover-Tyomkin}.
\end{rem}

\section{Open string theory I. Quasimap Floer theory}\label{section9}

In this section we recall the construction of quasimap Lagrangian Floer theory developed by Woodward \cite{Woodward_toric}. The basic idea agrees with the philosophy of gauged linear sigma model \cite{Witten_LGCY}: one replaces the count of holomorphic curves in the toric manifold $X$ by an equivariant count of holomorphic curves upstairs. There are two significant consequences: first, one can achieve transversality of moduli spaces at a very low cost; second, the counts of curves are all integer-valued.

We use the Morse--Bott model for Lagrangians Floer theory to construct open-string theories and closed-open maps. We extend the use of domain-dependent perturbations for bulk-deformed vortex Floer cohomology to the open-string situation. We first need to fix certain notions and notations to describe the combinatorial data of various moduli spaces.

\subsection{Trees and treed disks}\label{subsec:trees}

We first set up the convention of trees used in this paper.

\begin{convention}[Convention for trees]
A tree, usually denoted by $\Gamma$, consists of a nonempty set of vertices $V_\Gamma$ and a nonempty set of edges $E_\Gamma$. The set of vertices is decomposed into the set of finite vertices and the set of vertices at infinity, and the decomposition is denote by
\beqn
V_\Gamma = V_\Gamma^{\rm finite} \sqcup V_\Gamma^\infty.
\eeqn
We always assume 
\begin{enumerate}

\item $V_\Gamma^\infty$ contains a distinguished vertex $v_{\rm root}$ called the {\it root}. 

\item The valence (degree) of any $v\in V_\Gamma^\infty$ is either one or two.
\end{enumerate}
The set $V_\Gamma$ is partially ordered in the following way: we denote by $v_\alpha \succ v_\beta$ if $v_\alpha$ and $v_\beta$ are adjacent and $v_\beta$ is closer to the root. In this way vertices at infinities are either {\it incoming} (called {\it inputs}) or {\it outgoing} (called {\it outputs}); in particular the output $v_{\rm root}$ is outgoing. 

Edges are decomposed into four groups: the set of finite edges $E_\Gamma^{\rm finite}$ consisting of edges connecting two finite vertices, the set of incoming semi-infinite edges $E_\Gamma^{\rm in}$ consisting of edges connecting $v_\alpha \in V_\Gamma^\infty$ with $v_\beta \in V_\Gamma^{\rm finite}$ with $v_\alpha \succ v_\beta$, the set of outgoing semi-infinite edges $E_\Gamma^{\rm out}$ consisting edges connecting $v_\alpha \in V_\Gamma^{\rm finite}$ and $v_\beta \in V_\Gamma^\infty$ with $v_\alpha \succ v_\beta$, and the set of infinite edges $E_\Gamma^\infty$ connecting two vertices at infinity. We also call incoming resp. outgoing semi-infinite edges inputs resp. outputs. 

A tree $\Gamma$ is called {\it unbroken} if all vertices $v \in V_\Gamma^\infty$ has valence $1$. A vertex $v \in V_\Gamma^\infty$ of valence $2$ is called a {\it breaking} of the tree $\Gamma$. Breakings separate $\Gamma$ into {\it unbroken components}.

A {\it ribbon tree} is a tree $\Gamma$ together with an isotopy class of embeddings $\Gamma \hookrightarrow {\mb R}^2$. Equivalently, it means for each vertex $v \in V_\Gamma$ the adjacent edges are cyclically ordered. As $\Gamma$ is rooted, it follows that all incoming edges are strictly ordered.

A ribbon tree is {\it stable} if the valence of each finite vertex is at least three.

\end{convention}

\subsubsection{Metric ribbon trees}

A {\it metric} on a ribbon tree $\Gamma$ is a function 
\beqn
{\bm l}: E_\Gamma^{\rm finite} \to [0, +\infty).
\eeqn
The underlying decomposition 
\beqn
E_{\Gamma}^{{\rm finite}}= E_{\Gamma}^{{\rm finite}, 0} \sqcup E_{\Gamma}^{{\rm finite}, +} = {\bm l}^{-1}(\{0\})\sqcup {\bm l}^{-1}((0, +\infty))
\eeqn
is called a {\it metric type}, denoted by $[{\bm l}]$. We often call the pair $(\Gamma, [{\bm l}])$ a {\bf domain type}. A {\bf metric ribbon tree} of type $(\Gamma, [{\bm l}])$ is a pair $(\Gamma, {\bm l})$ such that ${\bm l}$ has the metric type $[{\bm l}]$.

As in \cite[Section 3.3]{Woodward_toric}, one needs to work with unstable trees. We hence replace the usual stability condition by another {\it minimality} condition. We say that a metric ribbon trees $(\Gamma, {\bm l})$ (resp. domain type $(\Gamma, [{\bm l}])$) is {\bf minimal} if it has no finite edges of length zero or infinite edges. Hence for each domain type $\Gamma$, there is a canonical minimal one $\Gamma^{\rm min}$ obtained from $\Gamma$ by shrinking edges violating the minimality condition.

We define perturbations over the universal trees. Consider a minimal domain type $\Gamma = (\Gamma, [{\bm l}])$ (which is not necessarily stable). Then there is a moduli space of metric trees of type $\Gamma$, denoted by $\mc{MT}_\Gamma$, which is homeomorphic to $(0, +\infty)^{\# E_{\uds\Gamma}^{{\rm finite}, +}}$, whose elements parametrize the lengths of finite edges with positive lengths.  There is also a {\it universal tree}
\beqn
\mc{UT}_\Gamma \to \mc{MT}_\Gamma
\eeqn
whose fiber over a point $p \in \mc{MT}_\Gamma$ is homeomorphic to a metric tree representing $p$ (the infinities of semi-infinite or infinite edges are regarded as points in the metric tree). 

The above moduli spaces have natural compactifications. In fact, we can define a partial order among all minimal domain types. We say that a minimal domain type $\Gamma$ {\bf degenerates} to another minimal domain type $\Pi$, denoted by $\Pi \preceq \Gamma$, if $\Pi$ is obtained from $\Gamma$ by composing the following types of operations 
\begin{enumerate}

\item Shrinking the length of a finite edge in $\Gamma$ to zero and collapse this edge. 

\item Breaking a finite edge of positive length to a pair of semi-infinite edges joined at a new 
vertex at infinity.
\end{enumerate}
Notice that if $\Pi\preceq \Gamma$, then there is a canonical surjective map $\rho: V_\Gamma^{\rm finite} \to V_\Pi^{\rm finite}$. Then $\mc{MT}_\Gamma$ has the natural compactification
\beqn
\ov{\mc{MT}}_\Gamma:= \bigsqcup_{\Pi \preceq\Gamma} \mc{MT}_\Pi.
\eeqn
The universal tree is also extended to the compactification, which is denoted by 
\beqn
\ov{\mc{UT}}_\Gamma \to \ov{\mc{MT}}_\Gamma.
\eeqn
There is a special closed subset $\ov{\mc{UT}}{}_\Gamma^{\rm node} \subset \ov{\mc{UT}}_\Gamma$ corresponding to infinities or vertices. Notice that the complement of $\ov{\mc{UT}}{}_\Gamma^{\rm node}$ inside the interior $\mc{UT}_\Gamma$ is a smooth manifold.

\subsubsection{Treed disks}

\begin{defn}
Given a domain type $\Gamma = (\Gamma, [{\bm l}])$. A {\it treed disk} of type $\Gamma$, denoted by $C = S \cup T$, is the configuration given by the union of disk components $S_\alpha \cong {\mb D}$ for all vertices $v_\alpha\in V_{\Gamma}^{\rm finite}$, a metric ${\bm l}$ on $\Gamma$ of type $[{\bm l}]$ and an interval $I_e$ of length ${\bm l}(e)$ for each finite edge $e \in E_\Gamma^{\rm finite}$. The notion of isomorphisms between treed disks is standard and omitted.
\end{defn}






\subsection{Quasimap Floer theory for Lagrangians}

We recall the quasimap Floer theory developed by Woodward \cite{Woodward_toric}. Let ${\bf u} \in {\rm Int} P \subset {\mb R}^n$ be an interior point of the moment polytope $P$ of the toric manifold $X$. Recall that the number of faces $N$ of $P$ coincides with the dimension of $V$. Let $L = L({\bf u}) \subset X$ be the torus fiber over ${\bf u}$. Let $\wh L  = \wh L ({\bf u}) \subset \mu^{-1}(0) \subset V$ be the lift of $L({\bf u})$, which is a $K$-invariant Lagrangian torus in $V$. Explicitly, we have 
\beqn
\wh L = \prod_{i=1}^N \Big\{ z_i \in {\mb C}\ |\ |z_i|^2 = \tau_i \Big\}
\eeqn
where $\tau_i$ are determined by ${\bf u}$ and the constant term in the moment map $\mu$.

A {\bf holomorphic quasidisk} is an ordinary holomorphic map $u: ({\mb D}, \partial {\mb D}) \to (V, \wh L)$ (with respect to the standard complex structure $\wh J_V$). Two holomorphic quasidisks $u$ and $u'$ are $K$-equivalent if there exists $g \in K$ such that $g u = u'$. Each $K$-equivalence class of holomorphic quasidisks represents a disk class 
\beqn
\beta \in H_2( V, \wh L)/K \cong H_2(V, \wh L).
\eeqn
Each such class has a well-defined energy 
\beqn
\omega(\beta) = \omega_V (\beta) \in {\mb R}
\eeqn
and a well-defined Maslov index 
\beqn
i(\beta) \in 2{\mb Z}.
\eeqn
Given $k$ and $\beta \in H_2(V,  \wh L)$, let ${\mc M}_{k+1}^{\rm disk} (\beta)$ be the moduli space of $K$-equivalence classes of holomorphic quasidisks of class $\beta$ with $k+1$ boundary marked points, and let $\ov{\mc M}{}_{k+1}^{\rm disk}(\beta)$ be its compactification. Notice that as $V$ is aspherical, configurations in $\ov{\mc M}{}_{k+1}^{\rm disk} (\beta)$ have only disks bubbles but not sphere bubbles. The evaluation of a $K$-equivalence class of quasidisks at the last boundary marked point is well-defined as a point in the quotient Lagrangian $L \subset X$. Hence there is a continuous map 
\beqn
\ev: \ov{\mc M}{}_{k+1}^{\rm disk}(\beta) \to L.
\eeqn

\begin{thm}[Blaschke product]\label{thm_Blaschke}
Let $u: {\mb D}^2 \to V$ be a holomorphic quasidisk. Then there exist $\theta_1, \ldots, \theta_N \in [0, 2\pi)$ and $(a_{i, k})_{k = 1, \ldots, d_i} \subset {\mb D}^2 \subset {\mb C}$ for $i = 1, \ldots, N$ such that
\beq\label{Blaschke_formula}
u(z) = \left( \sqrt{\tau_1} e^{i \theta_1} \prod_{k=1}^{d_1} \frac{ z- a_{1, k}}{ 1 - \ov{a_{1,k}} z}, \ldots, \sqrt{\tau_N} e^{i \theta_N} \prod_{k=1}^{d_N} \frac{ z- a_{N, k}}{ 1 - \ov{a_{N,k}} z}    \right).
\eeq
Moreover, the Maslov index of $u$ is $2 (d_1 + \cdots + d_N)$. 
\end{thm}

In particular, there are $N$ ``basic'' Maslov two disk classes $\beta_1, \ldots, \beta_N \in H_2(V, L)$ where each $\beta_i$ is represented by a quasidisk given as above with $d_j = \delta_{ij}$. These Maslov two classes form a basis of $H_2(V, \wh L)$. 

\begin{thm}
The moduli space ${\mc M}_{k+1}^{\rm disk}(\beta)$ is regular of dimension $n + 2 i (\beta) + k-2$ and the evaluation map 
\beqn
\ev: {\mc M}_{k+1}^{\rm disk}(\beta) \to L
\eeqn
is a smooth submersion.
\end{thm}
\begin{proof}
See \cite[Section 6]{Cho_Oh}.
\end{proof}

A consequence is that each stratum of the compactification $\ov{\mc M}{}_{k+1}^{\rm disk}(\beta)$ is regular. To be more precise, let $\Gamma$ denote a ribbon tree representing the combinatorial type of a nodal disk (with $k$ inputs and $1$ output) with each vertex labelled by a disk class whose sum is equal to $\beta$. Then there is a stratum ${\mc M}_\Gamma^{\rm disk}\subset \ov{\mc M}{}_{k+1}^{\rm disk}(\beta)$. 

\begin{cor}
Each stratum ${\mc M}{}_\Gamma^{\rm disk} \subset \ov{\mc M}_{k+1}(\beta)$ is regular and the evaluation map $\ev: {\mc M}_\Gamma^{\rm disk} \to L$ is a submersion.
\end{cor}
\begin{proof}
See \cite[Corollary 6.2]{Woodward_toric}.
\end{proof}

\subsubsection{Treed holomorphic quasimaps}

The idea of treed holomorphic disks goes back to Cornea--Lalonde \cite{Cornea_Lalonde_2005, Cornea_Lalonde}. We recall the adaptation by Woodward \cite{Woodward_toric} in order to define the quasimap $A_\infty$ algebras. Throughout our discussion, we fix a smooth perfect Morse function $f_L: L \to {\mb R}$ defined over the Lagrangian torus $L \subset X$, which has exactly $2^n$ critical points.


Given a treed disk $C = S \cup T$ of type $\Gamma$, suppose we have a domain-dependent perturbation $f$ of the Morse function $f_L: L \to {\mb R}$ parametrized by points $t$ on the tree part $T$, a {\bf treed holomorphic quasimap} on $C$ is a collection of objects 
\beqn
\Big( (u_v)_{v \in V_\Gamma^{\rm finite}}, (x_e)_{e \in E_{\Gamma}} \Big)
\eeqn
where for each finite vertex $v \in V_\Gamma^{\rm finite}$, we assign a smooth map $u_v: S_v \to V$ satisfying 
\beqn
\ov\partial u_v  = 0,\ u_v(\partial S_v) \subset \wh L,
\eeqn
$x_e: I_e \to L$ is a smooth map satisfying 
\beqn
x_e'(t) + \nabla f(x_e(t)) = 0;
\eeqn
moreover, the matching condition requires 1) for each node joining a boundary point $z$ of some surface component $S_v$ and a finite end of an edge $e$, the value of $x_e(z)$ lies in the same $K$-orbit as the value of $u_v(z)$; 2) for each infinite vertex $v\in V_\Gamma^\infty$ joining two (semi-)infinite edges $e_1$ and $e_2$, the limits of $x_{e_1}$ and $x_{e_2}$ at the corresponding infinities agree. Here to ensure the convergence of the maps $x_e$, we require that the perturbation $f$ is supported away from the infinities.

Two treed holomorphic quasimaps are regarded as {\bf equivalent} if after identifying domains, the maps on corresponding surfaces parts are $K$-equivalent (recall $K$ is the gauge group). 

To define the $A_\infty$ structure (or other structures) one would like to regularize the moduli spaces of equivalence classes of treed holomorphic quasimaps and their boundaries. One first needs to use coherent systems of perturbations to describe such moduli spaces. 

\subsubsection{Perturbations for the $A_\infty$ algebra}

To achieve transversality relevant for defining the $A_\infty$ algebra, we only need to perturb the Morse function on edges. Hence for a given {\it minimal} metric type $\Gamma$, a domain-dependent perturbation can be viewed as a map 
\beqn
P_\Gamma: \ov{\mc{UT}}_\Gamma \to C^\infty(L).
\eeqn
We require any such perturbation to vanish near infinities, i.e., vanish near the closed subset 
\beqn
\ov{\mc{UT}}{}_\Gamma^\infty \subset \ov{\mc{UT}}_\Gamma
\eeqn
corresponding to positions of vertices at infinity. Notice that if $\Gamma$ is not necessarily stable, a perturbation $P_{\Gamma^{\rm min}}$ for the minimal form is enough to determine the treed holomorphic map on any treed disks $C$ of type $\Gamma$. Indeed, on any infinite edges of $C$ (if any) the negative gradient flow equation is taken for the unperturbed Morse function $f_L$.

In order to establish the $A_\infty$ relation, we also need to require that, if $\Gamma$ degenerates to $\Pi$, then the restriction of $P_\Gamma$ to the stratum $\ov{\mc{UT}}_\Pi\subset \ov{\mc{UT}}_\Gamma$ must agree with the perturbation $P_\Pi$ which have been chosen for the minimal domain type $\Pi$. Hence we need to construct a {\it coherent} system of perturbations indexed for all minimal domain types $\Gamma$. To use the Sard--Smale theorem to prove that generic perturbations are regular, we also need to specify the neighborhood of $\ov{\mc{UT}}{}_\Gamma^\infty$ where we require the perturbation to vanish; such choices of neighborhoods need also be coherent. 

Another complexity in this procedure is that we need to work with unstable domains (as in \cite{Woodward_toric}, see also \cite{Abouzaid_plumbing}), unlike the cases of \cite{Charest_Woodward_2015}\cite{Woodward_Xu}\cite{Venugopalan_Woodward_Xu} where domains are always stable. Here we give a different way of writing Woodward's perturbation scheme for unstable trees (see Section \cite[Section 3]{Woodward_toric}). Given a minimal domain type $\Gamma$, an {\bf indexing function} is a map $\vec{n}: V_\Gamma^{\rm finite} \to {\mb Z}_{\geq 0}$, whose values are denoted by $n_v$, satisfying that $n_v \geq 1$ when $v$ is an unstable vertex. One should regard the values of $\vec{n}$ as one half of the Maslov indices of disk components. We consider perturbations which depend also on such indexing functions.

\begin{defn}\label{defn_coherent_perturbation}
A {\bf coherent family of domain-dependent perturbations} is a collection of continuous maps
\beqn
P_{\Gamma, \vec{n}}^{\rm qd}: \ov{\mc{UT}}_\Gamma \to C^\infty(L)
\eeqn
indexed by all minimal domain types $\Gamma$ and all indexing functions $\vec{n}: V_\Gamma^{\rm finite} \to {\mb Z}_{\geq 0}$ satisfying the following conditions.
\begin{enumerate}


    \item For $\Gamma$ the tree with a single vertex, no input, and one output, the Morse function on the outgoing edge is the unperturbed function $f_L$. 

    \item When $\Gamma$ degenerates to $\Pi$, there is a canonical surjective map $\rho: V_\Pi^{\rm finite} \to V_\Gamma^{\rm finite}$. Hence any indexing function 
    $\vec{n}_\Pi: V_\Pi \to {\mb Z}_{\geq 0}$ induces a partition $\vec{n}_\Gamma: V_\Gamma \to {\mb Z}_{\geq 0}$. We require that 
    \beqn
    P_{\Gamma, \vec{n}_\Gamma}^{\rm qd}|_{\ov{\mc U}_\Pi} = P_{\Pi, \vec{n}_\Pi}^{\rm qd}.
    \eeqn

    \item When $\Gamma$ is broken with unbroken components $\Gamma_1,\ldots, \Gamma_s$, the partition $\vec{n}$ on $\Gamma$ is defined by assembling the partitions $\vec{n}_1, \ldots, \vec{n}_s$ on $\Gamma_1, \ldots, \Gamma_s$. Then $P_{\Gamma, \vec{n}}^{\rm qd}$ should be naturally induced from $P_{\Gamma_i, \vec{n}_i}$. 
\end{enumerate}
\end{defn}

\subsubsection{Compactification and transversality}

Let $\Gamma$ be a possibly unstable, non-minimal domain type. A {\bf map type} over $\Gamma$, denoted by ${\bm \Gamma}$, assigns to each finite vertex $v_\alpha \in V_\Gamma^{\rm finite}$ a disk class $\beta_v$ (with nonnegative Maslov index) and to each vertex at infinity $v_\beta \in V_\Gamma^\infty$ a critical point $x_\beta \in {\rm crit} f_L$. A map type ${\bm \Gamma}$ induces an indexing function $\vec{n}$ on the minimal form $\Gamma^{\rm min}$ by setting $n_v$ to be half of the Maslov index of $\beta_v$ and adding together if several vertices are connected by finite edges of length zero. Then use the perturbation $P_{\Gamma^{\rm min}, \vec{n}}^{\rm qd}$ to define a moduli space ${\mc M}_{\bm \Gamma}$ of treed holomorphic disks. The topology of ${\mc M}_{\bm\Gamma}$ is defined in the usual way. 

Given a perturbation, the moduli space ${\mc M}_{\bm \Gamma}$ is the zero locus a Fredholm section on certain Banach manifold. We say that the moduli space ${\mc M}_{\bm \Gamma}$ is regular if the Fredholm section is transverse (it is independent of the corresponding Sobolev completions of the space of smooth maps). We say that a coherent system of perturbations is {\bf regular} if all moduli spaces ${\mc M}_{\bm \Gamma}$ are regular.

Now we consider possible degenerations of treed holomorphic disks. In general, a sequence of treed holomorphic disks of a fixed map type ${\bm \Gamma}$ can converge to a limit by breaking an edge, shrinking an edge to zero, or bubbling off holomorphic disks. Notice that because $V$ is a vector space and we do not have interior markings, there cannot be any sphere bubbles in the limit. The notion of convergence is standard and its definition is omitted here. As the perturbation system is coherent, any limiting object (of a possibly different map type ${\bm \Pi}$) is also a treed holomorphic disk defined using a corresponding perturbation $P_{\Pi^{\rm min}, \vec{n}}^{\rm qd}$, hence an element in ${\mc M}_{\bm \Pi}$. We denote 
\beqn
\ov{\mc M}_{\bm \Gamma}:= \bigsqcup_{{\bm \Pi} \preceq {\bm \Gamma}} {\mc M}_{\bm \Pi}
\eeqn
where by abuse of notation, $\preceq$ is the natural partial order among map types induced from the notion of convergence.

\begin{prop}
There exists a coherent system of perturbation data such that every moduli space ${\mc M}_{\bm \Gamma}$ is regular. 
\end{prop}

\begin{proof}
The proof is an inductive construction with respect to the partial order $\Pi \preceq \Gamma$ among minimal domain types and the indexing function $\vec{n}$. First one can easily check, by the Blaschke formula Theorem \ref{thm_Blaschke} that the specification of item (1) in Definition \ref{defn_coherent_perturbation} can make the relevant configurations transverse. Then once regular perturbations on all boundary strata of $\mc{UT}_\Gamma$ have been fixed, one can use the Sard--Smale theorem to find regular extensions to the interior. See details in \cite[Corollary 6.2]{Woodward_toric}. \end{proof}

Now we consider the compactification of moduli spaces. A map type ${\bm \Gamma}$ is called {\bf essential} if it is unbroken and has no boundary edges of length zero. Given a collection ${\bm x} = (x_1, \ldots, x_k; x_\infty)$ of critical points of the Morse function $f_L$, for $i = 0, 1$, let 
\beqn
{\mc M}^{qd}(x_1, \ldots, x_k; x_\infty)_i := \bigcup_{{\bm \Gamma}} {\mc M}_{\bm \Gamma}
\eeqn
where the union is taken over all essential map types of index $i$ whose vertices at infinities are labelled by ${\bm x}$. 

\begin{lemma}
If $i = 0$, the moduli space ${\mc M}^{qd}(x_1, \ldots, x_k; x_\infty)_0$ is discrete and has finitely many points below any given energy bound. If $i = 1$, the compactified moduli space $\ov{\mc M}{}^{qd}(x_1, \ldots, x_k; x_\infty)_1$ is a 1-dimensional (topological) manifold with boundary, which is compact below any given energy bound. 
\end{lemma}

\begin{proof}
For the zero-dimensional moduli space, the claimed finiteness follows from the compactness argument and the transversality. For the one-dimensional moduli space, the fact that it is a 1-dimensional manifold with boundary follows from the transversality, compactness, as well as the standard gluing construction. 
\end{proof}

Moreover, the moduli spaces are all oriented. The orientation depends on choices of orientations of unstable manifolds of critical points of $f_L$ and the orientations of moduli spaces of quasidisks; the latter depends on the orientation of the Lagrangian torus and the spin structure, which we fix from the beginning. Notice that these choices can be made independent of the position ${\bf u}\in {\rm Int} P$ in the interior of the moment polytope.

\subsubsection{Quasimap Fukaya $A_\infty$ algebra}

We would like to define a (family of) cohomologically unital $A_\infty$ algebra(s) over $\Lambda_{\ov{\mb Q}}$ from the moment Lagrangian tori. Given a Lagrangian torus $L = L({\bf u}) \subset X$, a {\bf local system} on $L$ is a homomorphism 
\beqn
{\bf y}: H_1(L; {\mb Z}) \to \exp(\Lambda_{0, \ov{\mb Q}}).
\eeqn
Introduce the notation $\L = (L, {\bf y})$. We denote the corresponding bulk-deformed $A_\infty$ algebra of $\L$ by ${\mc F}_{\mf b}(\L)$, which is defined as follows. First, the underlying ${\mb Z}_2$-graded vector space is 
\beqn
\qcf_{\mf b}^\bullet ( \L; \Lambda_{\ov{\mb Q}}):= {\rm Span}_{\Lambda_{\ov{\mb Q}}} {\rm crit} f_L  \cong (\Lambda_{\ov{\mb Q}})^{2^n}
\eeqn
where the degree of a critical point $x \in {\rm crit} f_L$ is $|x| = n - {\rm index} (x)\ {\rm mod}\ 2$. Given critical points $x_1, \ldots, x_k$, define 
\beq\label{composition_defn}
m_k (x_k, \ldots, x_1) = \sum_{x_\infty} (-1)^\heartsuit \left( \sum_{ [u]\in {\mc M}^{qd}(x_1, \ldots, x_k; x_\infty)_0} {\mf b}([u]) T^{E([u])} {\bf y}^{\partial[u]} \epsilon([u])  \right) x_\infty.
\eeq
We explain the terms below.
\begin{enumerate}
\item The sign $\heartsuit$ is defined as 
\beq\label{heartsuit}
\heartsuit:= \sum_{i=1}^k i |x_i|\in {\mb Z}_2.
\eeq

\item For each disk $u$ with boundary on $\wh L$, as $\wh L$ does not intersect the bulk, there is a well-defined topological intersection number 
\beqn
{\mf b}([u]):= \prod_{j=1}^N c_j^{u \cap V_j}
\eeqn
which only depends on $K$-equivalence class $[u]$. Notice that if $c_j \in {\mb Z}[{\bf i}]$, so is ${\mf b}(u)$. 

\item $E([u])\in {\mb R}$ is the energy of $[u]$.

\item ${\bf y}^{\partial [u]} \in \exp (\Lambda_{0, \ov{\mb Q}})$ is the value of the local system ${\bf y}$ on the loop $\partial [u]\subset L$. 

\item $\epsilon([u]) \in \{\pm 1\}$ is determined by the orientation of the zero-dimensional moduli space.

\end{enumerate}
Similar to previous cases involving bulk deformations, the expression \eqref{composition_defn} is a legitimate element of $\qcf_{\mf b}^\bullet (\L; \Lambda_{\ov{\mb Q}})$. Extending linearly, one obtains a linear map 
\beqn
m_k: \qcf_{\mf b}^\bullet (\L; \Lambda_{\ov{\mb Q}} )^{\otimes k} \to \qcf_{\mf b}^\bullet (\L; \Lambda_{\ov{\mb Q}}).
\eeqn
Notice that when $k = 0$, this is a linear map 
\beqn
m_0: \Lambda_{\ov{\mb Q}} \to \qcf_{\mf b}^\bullet (\L; \Lambda_{\ov{\mb Q}}).
\eeqn

\begin{thm}[\cite{Woodward_toric}]\label{cohomological_unit} The collection of linear maps $m_0, m_1, \ldots$ defines a curved $A_\infty$ algebra structure on $\qcf_{\mf b}^\bullet (\L; \Lambda_{\mb Q})$, denoted by ${\mc F}_{\mf b}(\L)$. Moreover, if $x_{\max}$ is the unique maximal point of $f_L$, then ${\bf e} = x_{\rm min}$ is a cohomological unit of ${\mc F}_{\mf b}(\L)$, namely $m_1 ({\bf e}) = 0$ and
\beqn
(-1)^{|x|} m_2( {\bf e}, x) = m_2( x, {\bf e}) = x,\ \forall x\in \qcf_{\mf b}^\bullet(\L; \Lambda_{\ov{\mb Q}}).
\eeqn
\end{thm}

\begin{proof}
See \cite[Theorem 3.6]{Woodward_toric} for the case without bulk deformation. One can verify that the case with bulk deformation can be proved in the same way.
\end{proof}

\begin{rem}
The $A_\infty$ algebra can be defined over ${\mb Z}$ as long as the bulk deformation has integer coefficients, though we do not need such a fact in our discussion.
\end{rem}

\subsubsection{Potential function and nontrivial Floer cohomology}

Although the quasimap Fukaya algebra is only cohomologically unital, one can still define the potential function. 

\begin{prop}
For the quasimap $A_\infty$ algebra $\qcf(\L; \Lambda_{\ov{\mb Q}})$, $m_0(1)$ is a multiple of ${\bf e}$.
\end{prop}

\begin{proof}
See \cite[Proposition 3.7]{Woodward_toric} for the case with ${\mf b}= 0$. When we use a nontrivial (small) bulk deformation, as we only change the weights in counting but do not modify the perturbation method, the same proof goes through.
\end{proof}

\begin{defn}
Define $W_{\mf b}(u): H_1(L( {\bf u} ); \exp (\Lambda_{0, \ov{\mb Q}}) ) \to \Lambda$ by 
\beqn
m_0(1) = W_{\mf b}(u)({\bf y}) {\bf e}
\eeqn
and call it the {\bf potential function} of the brane $\L = (L({\bf u}), y)$. By abuse of terminology, we also call $W_{\mf b}$ the bulk-deformed potential function of the Lagrangian $L({\bf u})$ or the toric manifold.

Let $({\mb C}^*)^n \cong X^* \subset X$ be the complement of toric divisors. Choose a trivialization
\beqn
\tau_X: {\rm Int}P \times T^n \to X^*
\eeqn
which is unique up to isotopy, which induces a well-defined trivialization
\beqn
\bigsqcup_{u \in {\rm Int} P} H_1(L({\bf u}); \exp (\Lambda_{0, \ov{\mb Q}})) = {\rm Int} P \times (\exp (\Lambda_{0, \ov{\mb Q}}))^n.
\eeqn
The bulk-deformed {\bf quasimap disk potential} 
of the toric manifold $X$ is defined by 
\beqn
\begin{split}
W_{\mf b}: {\rm Int} P \times (\exp \Lambda_0)^n & \to \Lambda\\
             ({\bf u}, {\bf y}) & \mapsto W_{\mf b}({\bf u})({\bf y}).
\end{split}
\eeqn
\end{defn}

Now we can define the quasimap Floer cohomology. By the $A_\infty$ relation, for any $x \in \qcf(\L; \Lambda_{\ov{\mb Q}})$, 
\beqn
m_1(m_1(x)) + (-1)^{\|x\|} m_2( m_0(1), x) + m_2(x,m_0(1)) = 0.
\eeqn
By Theorem \ref{cohomological_unit}, the last two terms cancel. Hence $m_1^2 = 0$. Hence one can define the ${\mf b}$-deformed {\bf quasimap Floer cohomology} of the brane $\L$ to be 
\beqn
\qhf_{\mf b}^\bullet (\L; \Lambda_{\ov{\mb Q}}):= {\rm ker} m_1/ {\rm im} m_1.
\eeqn

Following \cite{Cho} \cite{Cho_Oh} \cite{Woodward_toric}, to find nontrivial Floer cohomology, one needs to establish a version of the divisor equation. Recall that $ L \cong (S^1)^n$ with $H_1 (L; {\mb Z}) \cong {\mb Z}^n$. The perfect Morse function $f_L$ has exactly $n$ critical points of Morse index $1$, whose homology classes are identified with the $n$ standard generators of $H_1(L; {\mb Z})$. If $x_1, \ldots, x_n$ are these generators, then any local system ${\bf y}$ is determined by the values
\beqn
y_1 = {\bf y}(x_1), \ldots, y_n = {\bf y}(x_n).
\eeqn

\begin{thm}\label{thm_divisor_eqn}
If $x$ is a generator of $H_1(L; {\mb Z})$, then 
\beqn
m_1(x) = \partial_x W_{\mf b}({\bf u})(y_1, \ldots, y_n)
\eeqn
\end{thm}

\begin{proof}
In the absence of bulk deformation, this is established in \cite[Section 3.6]{Woodward_toric}, which also carries over in our case.
\end{proof}

Lagrangian branes with nontrivial Floer cohomology can be identified with critical points of the potential function.

\begin{thm}(cf. \cite[Theorem 6.6]{Woodward_toric})
If ${\bf y} = (y_1, \ldots, y_n)$ is a critical point of $W_{\mf b}({\bf u})$, then the Floer cohomology of $\L({\bf u}) = (L({\bf u}),{\bf y})$ is isomorphic to $H^\bullet( L({\bf u}); \Lambda_{\ov{\mb Q}})$. 
\end{thm}

\begin{proof}
The case with ${\mf b} = 0$ is given by \cite[Theorem 6.6]{Woodward_toric}. When we have a nonzero small bulk deformation, it is still a consequence of the divisor equation (Theorem \ref{thm_divisor_eqn}).
\end{proof}

\subsection{Critical points of the Givental--Hori--Vafa potential}

In this subsection we study various properties of the deformed Givental--Hori--Vafa potential which arises from disk counting in gauged linear sigma model.

We first recall the expression of the Givental--Hori--Vafa potential in terms of the data of the moment polytope and explain its relation with the quasimap disk potential. Let $\Delta\subset {\mb R}^n$ be the moment polytope of $X$, described by 
\beqn
\Delta = \Big\{ u \in {\mb R}^n\ |\ l_j(u) = \langle u, v_j \rangle - \lambda_j \geq 0,\ j = 1, \ldots, N \Big\}.
\eeqn
Here $v_j = (v_{j, 1}, \ldots, v_{j, n}) \in {\mb Z}^n$, $j = 1, \ldots, N$ are the inward normal vectors of each codimension 1 face of $\Delta$ coming from the toric data and $\lambda_j \in {\mb R}$. The {\bf Givental--Hori--Vafa potential} of $X$ (or rather its moment polytope) is the element 
\beqn
W_0  = \sum_{j=1}^N T^{-\lambda_j} y^{v_j}:= \sum_{j=1}^N T^{-\lambda_j} y_1^{v_{j, 1}}\cdots y_n^{v_{j, n}}\in \Lambda[y_1, \ldots, y_n, y_1^{-1}, \ldots, y_n^{-1}]. 
\eeqn
More generally, given any small bulk deformation ${\mf b} = \sum_{j=1}^N \log c_j V_j$, the deformed Givental--Hori--Vafa potential is defined to be
\beqn
W_{\mf b} = \sum_{j=1}^N c_j T^{-\lambda_j} y^{v_j}.
\eeqn
Without loss of generality, we assume that the origin $0 \in {\mb R}^n$ is contained in the interior of $\Delta$. Hence all $\lambda_j$ are positive.

\begin{defn}\label{defn:log-derivative}
A point ${\bm \eta} = (\eta_1, \ldots, \eta_n) \in (\Lambda \setminus \{0\})^n$ is called a {\bf critical point} of $W_{\mf b}$ if 
\beqn
\left( y_1 \frac{\partial W_{\mf b}}{\partial y_1} \right) (\eta_1, \ldots, \eta_n) = \cdots = \left( y_n \frac{\partial W_{\mf b}}{\partial y_n} \right) (\eta_1, \ldots, \eta_n) = 0.
\eeqn
A critical point ${\bm \eta}$ is called {\bf nondegenerate} if 
\beqn
\det \left( \eta_i \eta_j \frac{\partial^2 W_{\mf b}}{\partial y_i \partial y_j}({\bm \eta}) \right) \neq 0.
\eeqn
$W_{\mf b}$ is called a {\bf Morse function} if all the critical points are nondegenerate.
\end{defn}

Observe that the Givental--Hori--Vafa potential is very similar to the quasidisk potential; the latter has a dependence on $u\in {\rm Int}\Delta$. Indeed, the disk potential of the Lagrangian $L(u)$ with a local system ${\bm y}\in (\exp(\Lambda_0))^n$ is 
\beqn
W_{\mf b}(T^{u_1} y_1, \ldots, T^{u_n} y_n).
\eeqn
This is proved by \cite[Corollary 6.4]{Woodward_toric} in the absence of bulk deformations, and the bulk-deformed version follows from the same argument by the Blaschke formula. 

Hence a critical point of $W_{\mf b}$ corresponds to a Floer nontrivial Lagrangian if the valuation of the coordinates of the critical point is in the interior of the moment polytope. On the other hand, in view of mirror symmetry, the Jacobian ring of the Givental--Hori--Vafa potential, or formally the ring of functions on the critical locus, is closely to related to the quantum cohomology under mirror symmetry. However, their ranks agree only in the Fano case. In general, certain critical points fall outside the moment polytope and do not correspond to cohomology classes of the toric manifold.

\begin{example}
Consider the $n$-th Hirzebruch surface $F_n$ ($n \geq 1$) whose moment polytope is 
\beqn
\Delta = \left\{ u =  (u_1, u_2) \in {\mb R}^2\ \left| \ \begin{array}{c} l_1(u) = u_1 \geq 0,\\ l_2(u) = u_2 \geq 0,\\ l_3(u) = 1-\alpha - u_2 \geq 0,\\ l_4 (u) = n -u_1 - n u_2 \geq 0.
\end{array}\right. \right\}
\eeqn
Here $\alpha\in (0, 1)$ is a parameter. The (undeformed) Givental--Hori--Vafa potential is 
\beqn
W_0 (y_1, y_2) = y_1 + y_2 + T^{1-\alpha} y_2^{-1} + T^n y_1^{-1} y_2^{-n}.
\eeqn
The equations for critical points are 
\begin{align*}
&\ y_1 = T^n y_1^{-1} y_2^{-n},\ &\ y_2 = T^{1-\alpha} y_2^{-1} + n T^n y_1^{-1} y_2^{-n}.
\end{align*}
Assume $n$ is even to simplify notations. Solving $y_1$ one obtains
\beqn
y_1 = \pm T^{\frac{n}{2}} y_2^{-\frac{n}{2}}
\eeqn
and hence 
\beq\label{Hirzebruch_surface_critical_point}
y_2 = T^{1-\alpha} y_2^{-1} \pm n T^{\frac{n}{2}} y_2^{-\frac{n}{2}} \Longrightarrow y_2^{\frac{n}{2}-1}( y_2^2 - T^{1-\alpha} ) = \pm T^{\frac{n}{2}}. 
\eeq
Each of the two equations has $\frac{n}{2}+1$ roots, providing $n+2$ critical points, much larger than the rank of homology (which is $4$).

Notice that there are two solutions to \eqref{Hirzebruch_surface_critical_point} of the form 
\beqn
y_2 = \pm T^{\frac{1-\alpha}{2}} + {\rm higher\ order\ terms}.
\eeqn
They give 4 critical points whose ``tropical'' positions are inside the moment polytope $\Delta$. There are also $n-2$ roots of \eqref{Hirzebruch_surface_critical_point} whose valuations are 
\beqn
\frac{\frac{n}{2}-(1-\alpha)}{ \frac{n}{2}-1}> 1-\alpha.
\eeqn
They correspond to critical points which are outside the moment polytope. This ends the example. 
\end{example}

\begin{defn}
We say that a critical point ${\bf \eta} = (\eta_1, \ldots, \eta_n)$ of $W_{\mf b}$ is {\bf inside the moment polytope} $\Delta$ if  
    \beqn
    \vec{\mf v}_T({\bm \eta}) = ({\mf v}_T(\eta_1), \ldots, {\mf v}_T(\eta_n)) \in {\rm Int} \Delta \subset {\mb R}^n.
    \eeqn
Denote by
    \beqn
    {\rm Crit}_X W_{\mf b} \subset {\rm Crit} W_{\mf b}
    \eeqn
    the set of critical points of $W_{\mf b}$ that are inside the moment polytope of $X$. 
\end{defn}

\begin{prop}\label{prop_same_rank}
Let ${\mf b}$ be an arbitrary small bulk deformation. When $W_{\mf b}$ is a Morse function, one has 
\beqn
\# {\rm Crit}_X W_{\mf b} = {\rm dim} H_\bullet(X).
\eeqn
\end{prop}

\begin{proof}
We use a result of Fukaya {\it et. al.} \cite[Theorem 2.8.1 (2)]{FOOO_mirror}. First, Fukaya {\it et. al.} defined their bulk-deformed Lagrangian Floer disk potential $\mf{PO}_{\mf b}$ by counting (stable) holomorphic disks inside the toric manifold (using $T^n$-equivariant Kuranishi structures). For our bulk-deformed Givental--Hori--Vafa potential function $W_{\mf b}$, their theorem shows that there exists a bulk deformation ${\mf b}'$ and a ``change of coordinate'' $y\mapsto y'$ such that
\beqn
W_{\mf b}(y') = \mf{PO}_{{\mf b}'}(y).
\eeqn
Notice that the change of coordinate does not change the Morse property and the tropical positions of the critical points. Hence one has 
\beqn
\# {\rm Crit}_X (W_{\mf b}) = \# {\rm Crit}_X (\mf{PO}_{{\mf b}'}).
\eeqn
On the other hand, by \cite[Theorem 1.1.3]{FOOO_mirror}, this number of critical points coincides with the rank of homology.
\end{proof}

Lastly we prove the following fact.

\begin{thm}\label{thm_good_bulk}
There exists a small bulk deformation ${\mf b} = \sum_{j=1}^N \log c_j V_j$ with $c_j \in {\mb Z}[{\bf i}]$ such that $W_{\mf b}$ is a Morse function and all critical values are distinct. 
\end{thm}

\begin{proof}
We first show that the statement is true for generic ${\mf b}$ with complex coefficients. First, we relate the Givental--Hori--Vafa potential to a complex Laurent polynomial by evaluation at $T = t$ for some complex number $t$. To consider convergence issue, introduce 
\beqn
\Lambda^{\rm conv}_{0, \ov{\mb Q}}:= \left\{ \sum_{i=1}^\infty a_i T^{\lambda_i} \in \Lambda_{0, \ov{\mb Q}} \ |\ \sum_{i=1}^\infty |a_i| |t|^{\lambda_i}\ {\rm converges\ for\ } |t|\leq \epsilon\ {\rm for\ some\ }\epsilon>0 \right\}.
\eeqn
Let $\Lambda^{\rm conv}_{\ov{\mb Q}}$ be its field of fractions. By \cite[Proposition 8.5]{FOOO_toric_1}, $\Lambda_{\ov{\mb Q}}^{\rm conv}$ is algebraically closed. On the other hand, critical points of $W_{\mf b}$ are solutions to algebraic equations with coefficients in $\Lambda_{\ov{\mb Q}}^{\rm conv}$, as the convergence holds due to the fact that $W_{\mf b}$ has only finitely many terms. Hence critical points are in $(\Lambda_{\ov{\mb Q}}^{\rm conv})^n$.

On the other hand, if we regard $T$ as a complex number, then by Kouchnirenko's theorem \cite{Koushnirenko_1975}, there is a proper analytic subset $S \subset {\mb C}^n$ (which in particular has positive codimension) such that when 
\beqn
c(t) = (c_1 t^{-\lambda_1}, \ldots, c_N t^{-\lambda_N}) \notin S
\eeqn
the function $W_{\mf b}^t:=\sum_{j=1}^N c_j t^{-\lambda_j} y^{v_j}$ has finitely many critical points and the number of them is bounded by $n!$ times the volume of the Newton polytope of this Laurent polynomial (which only depends on the moment polytope). As proved by Iritani \cite[Proposition 3.10]{Iritani_2009}, we can also guarantee that all critical points are nondegenerate. Now take a generic point $(c_1, \ldots, c_N)$ \footnote{Within this proof, being generic means being in the complement of a proper complex analytic subset.} so that $c(1) \notin S$. We claim that such a point satisfies our requirement. 

Indeed, the map 
\beqn
c: {\mb C} \setminus (-\infty, 0] \to {\mb C}^n
\eeqn
is an analytic map. Hence the complement of $c^{-1}(S)$ contains points arbitrarily close to $0$. We first show that the number of critical points of $W_{\mf b}$ is no greater than Kouchnirenko's bound, temporarily denoted by $N_\Delta$. Indeed, if there are $N_\Delta +1$ critical points, then as the coordinates of them are in $\Lambda_{\ov{\mb Q}}^{\rm conv}$, we can evaluate them at $T = t$ with $|t|$ sufficiently small and $c(t)\notin S$, obtaining more critical points of $W_{\mf b}^t$ than possible. Similarly, as we can evaluate critical points at $|t|$ small, all critical points have to be nondegenerate. 

Lastly, we prove that for generic ${\mf b}$ all critical values of $W_{\mf b}$ are distinct. First notice that the complex monomials $W_1, \ldots, W_N$ separate points, i.e., given $y', y'' \in ({\mb C}^*)^n$, $y' \neq y''$, for some $W_j$, $W_j(y') \neq W_j(y'')$. This is because a subset of $n$ monomials among $W_1, \ldots, W_N$ are coordinates on the torus of $y_1, \ldots, y_n$. Now consider the universal critical locus
\beqn
\widetilde{\rm Crit} W:= \big\{ (c_1, \ldots, c_N, y_1, \ldots, y_n)\ |\ dW_{\mf b}(y_1, \ldots, y_n) = 0 \big\}.
\eeqn
Over the nondegenerate locus it is a smooth $N$-dimensional complex manifold and $c_1, \ldots, c_N$ are local parameters. Given a nondegenerate $c_1, \ldots, c_N$, let $y^{(1)}, y^{(2)}$ be two different critical points. Suppose $W_j(y^{(1)}) \neq W_j(y^{(2)})$. Then deforming $c$ along $(c_1, \ldots, c_j + s, \ldots, c_N)$ and let the two critical points deform as $y^{(1)}(s)$, $y^{(2)}(s)$. Then 
\beqn
\frac{d}{ds} \left( W_s(y^{(1)}(s)) - W_s(y^{(2)}(s)) \right) = W_j(y^{(1)}) - W_j(y^{(2)}) \neq 0.
\eeqn
This means that the locus of $c$ where two critical values coincide is cut out transversely. 

Now we have shown that for generic complex ${\mf b}$, $W_{\mf b}$ satisfies the requirement. As the set of such complex ${\mf b}$ is open and dense, one can actually find ${\mf b}$ such that $c_j \in {\mb Q}[\sqrt{-1}]$. Then by rescaling one can find the desired bulk deformation.
\end{proof}

\begin{defn}\label{defn:convenient}
A bulk-deformation ${\mf b} = \sum_{j=1}^N \log c_j V_j$ with $c_j \in {\mb Z}[{\bf i}]$ is called {\bf convenient} if $W_{\mf b}$ is a Morse function and all critical values are distinct. 
\end{defn}

\subsection{Homotopy units}\label{subsection_homotopy_unit}

The $A_\infty$ algebra constructed using our perturbation scheme only has cohomological units. In order to establish strict unitality one needs the system of perturbations to satisfy an additional property with respect to the operation of forgetting any boundary inputs and stabilize. This is difficult to achieve (in contrast to the case of \cite{FOOO_toric_1}). Here we use a typical method of constructing a homotopy unit which appeared in \cite{FOOO_Book}\cite{Ganatra_thesis}\cite{Sheridan_2016}\cite{Charest_Woodward_2015}\cite{Woodward_Xu}\cite{Venugopalan_Woodward_Xu} etc.

\begin{defn}\cite[Section 4.3]{Sheridan_2016}
Let $(A, {\bf e})$ be a cohomological unital $A_\infty$ algebra over $\Lambda_{\mb K}$. A {\bf homotopy unit structure} on $(A, \e)$ is an $A_\infty$ structure on the $\Lambda_{\mb K}$-module 
\beqn
A^+ = A \oplus \Lambda_{\mb K} \f [1] \oplus \Lambda_{\mb K} \e^+
\eeqn
such that the $A_\infty$ composition maps on $A^+$ restrict to the $A_\infty$ composition maps on $A$, $m_1(\f) = \e^+ - \e$, and such that $\e^+$ is a strict unit, i.e.
\begin{align*}
&\ (-1)^{|a|} m_2(\e^+, a) = m_2(a, \e^+) = a,\ &\ m_k(\cdots, \e^+, \cdots) = 0\ \forall k \neq 2.
\end{align*}
\end{defn}

To construct a homotopy unit, one needs to include a collection of extra moduli spaces. Consider {\bf weighted ribbon trees} $\Gamma$ whose vertices at infinity $v\in V_\Gamma^\infty$ are either {\bf unweighted} or {\bf weighted}. We require that when $v$ is an output or a breaking, it must be unweighted. Each weighted boundary input carries an additional parameter $\rho\in [0, 1]$. Therefore a moduli space of weighted metric ribbon trees has additional parameters from weighted inputs. We require that the perturbations $P_{\Gamma, \vec{n}}^{\rm disk}$ on any (minimal) tree $\Gamma$ also depend on these parameters. Moreover, we require that
\begin{enumerate}
    \item When all inputs are unweighted, the perturbation on this tree coincides with the perturbation we have chosen to define the cohomologically unital Fukaya algebra ${\mc F}_{\mf b} (\L)$.

    \item For each weighted input, when the parameter $\rho = 0$, the perturbation on this tree agrees with the perturbation for the tree $\Gamma'$ obtained by changing this weighted input to an unweighted input. 

    \item For each weighted input $v\in V_\Gamma^\infty$, when the parameter $\rho = 1$, the perturbation $P_{\Gamma, \vec{n}}^{\rm disk}$ on this tree agrees with the perturbation obtained by pulling back a perturbation $P_{\Gamma', \vec{n}'}^{\rm disk}$ via the forgetful map. Here $\Gamma'$ is defined as follows. Suppose $v$ is attached to a finite vertex $v'$. If $n_{v'} > 0$ or after forgetting $v$, $v'$ is still stable, then $\Gamma'$ is just obtained by $\Gamma$ by removing $v$; if $n_{v'} = 0$ and $v'$ becomes unstable after removing $v$, then $\Gamma'$ is obtained from $\Gamma$ by removing $v$ and contracting $v'$ to the next adjacent finite vertex. See Figure \ref{figure_forget_weighted} for illustration of this operation.

    \end{enumerate}

\begin{figure}[h]
    \centering
    \includegraphics{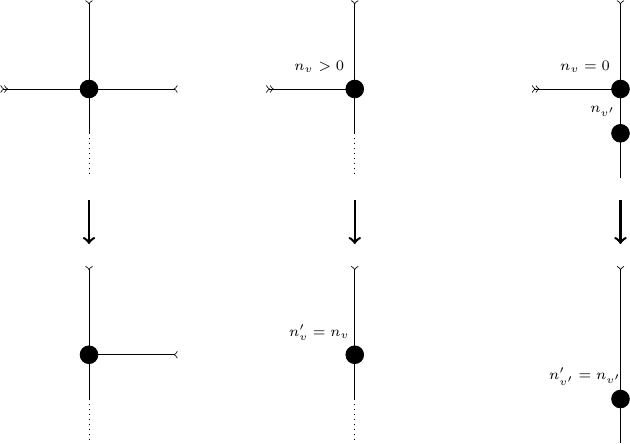}
    \caption{Forgetting a weighted input. }
    \label{figure_forget_weighted}
\end{figure}

Now we need to define the additional composition maps $m_k^+$ on $A^+$ when involves the new generators $f$ and $e^+$, and prove the $A_\infty$ relation for this enlarged set of compositions. We first define 
\beqn
m_k^+(\cdots, \e^+, \cdots)
\eeqn
according to the requirement of strict unit. Then we need to define $m_k^+$ for variables being either the original generators of $A$ or the element $\f$. To define this, we require that the incoming edges corresponding to weighted inputs converge to the unique maximal point of the Morse function $f_L: L \to {\mb R}$, and count $0$-dimensional moduli spaces. A consequence of the fact that all quasidisks have positive Maslov index is that
\beqn
m_k^+(\f, \cdots, \f) = 0\ \forall k \geq 2.
\eeqn

We need to verify the $A_\infty$ relation for all $m_k^+$. Recall that the $A_\infty$ relation reads
\beqn
\sum_{j=0}^k \sum_{i=0}^{k-j} (-1)^{\maltese_1^i} m_{k-j+1}^+ (x_k, \cdots, m_j^+ (x_{i+j+1}, \ldots, x_{i+1}), x_i, \ldots, x_1) = 0.
\eeqn
We only needs to verify for the case when all variables are generators of $A^+$. When all of them are old generators of $A$, this is the same as the original $A_\infty$ relation for $m_k$; when some variable is $\e^+$, this can be verified from the requirement that $\e^+$ satisfies the equations for a strict unit. Now assume that all variables are either old generators or $\f$. Consider $1$-dimensional moduli spaces with this fixed sequence of inputs and consider its boundary strata. In addition to the strata corresponding to boundary edge breakings, additional boundary strata corresponding to parameters $\rho$ on weighted inputs turn to $0$ or $1$. These strata correspond to the terms $m_k^+( \cdots, m_1^+(\f), \cdots)$ in the $A_\infty$ relation. Hence the $A_\infty$ relation for $m_k^+$ is verified. We summarize the above discussion as follows.

\begin{prop}
There exists a homotopy unit structure on the cohomologically unit $A_\infty$ algebra ${\mc F}_{\mf b} (\L)$. Denote the corresponding strictly unital $A_\infty$ algebra by ${\mc F}_{\mf b}^+(\L)$. Moreover, if we denote the element whose coboundary relates $\e$ and $\e^+$ by $\f_{\L}$, then one has 
\beqn
m_k^+\Big( \underbrace{\f_{\L}, \ldots, \f_{\L}}_{k} \Big) = 0,\ \forall k \geq 2.
\eeqn
\end{prop}

\subsubsection{Canonical weakly bounding cochain}

Recall that a weakly bounding cochain is an odd element $b \in {\mc F}_{\mf b}^+ (\L)$ solving the weak Maurer--Cartan equation 
\beqn
\sum_{k\geq 0} m_k^+(b, \cdots, b) \in \Lambda \e^+.
\eeqn
In general, worrying about convergence, we require that $b$ has a positive Novikov valuation. In our case, we only use a special weakly bounding cochain.

\begin{defn}
The {\bf canonical} weakly bounding cochain of the strictly unital $A_\infty$ algebra ${\mc F}_{\mf b}^+(\L)$ is 
\beqn
b_{\L} = W_{\mf b} \f_{\L}.
\eeqn
\end{defn}

We check that, by the fact that $m_k^+( \f_{\L}, \cdots,  \f_{\L} ) = 0$ for $k \geq 2$ and $m_1^+ ( \f_{\L}) = \e_{\L}^+ - \e_{\L}$, one has 
\beqn
\sum_{k \geq 0} m_k^+(b_{\L}, \cdots, b_{\L}) = m_0^+(1) + m_1^+( W_{\mf b} \f_{\L}) = W_{\mf b} \e_{\L} + W_{\mf b} (\e_{\L}^+ - \e_{\L}) = W_{\mf b} \e_{\L}^+.
\eeqn
Hence indeed $b_{\L}$ is a weakly bounding cochain.

Now we can define the flat $A_\infty$ algebra ${\mc F}_{\mf b}^\flat(\L)$ with compositions being (for $k \geq 1$) 
\beqn
m_k^\flat (x_k, \ldots, x_1) = \sum_{l_0, \ldots, l_k \geq 0}  m_{k+l_0 + \cdots + l_k}^+ \Big( \underbrace{ b_{\L}, \ldots, b_{\L}}_{l_k}, x_k, \cdots, x_1, \underbrace{ b_{\L}, \ldots, b_{\L}}_{l_0} \Big). 
\eeqn
In particular, $m_1^\flat \circ m_1^\flat = 0$ and the cohomology of ${\mc F}_{\mf b}^\flat(\L)$ agrees with the quasimap Floer cohomology $\qhf_{\mf b}^\bullet(\L; \Lambda_{\ov{\mb Q}})$.

\subsubsection{Multiplicative structure}

We need to identify the multiplicative structures on the quasimap Floer cohomology. The second composition $m_2^\flat$ on ${\mc F}_{\mf b}^\flat (\L)$ induces a multiplication on $\qhf_{\mf b}^\bullet( \L; \Lambda_{\ov{\mb Q}})$. 

\begin{prop}\label{prop_ring_structure}
When ${\bf y}$ is a critical point of $W_{\mf b}({\bf u})$ and the Hessian of $W_{\mf b}({\bf u})$ is nondegenerate at ${\bf y}$, i.e.
\beqn
\det \left( \frac{\partial^2 W_{\mf b}({\bf u})}{\partial x_i \partial x_j} ({\bf y}) \right) \neq 0,
\eeqn
the quasimap Floer cohomology algebra $\qhf_{\mf b}^\bullet(\L; \Lambda_{\ov{\mb Q}})$ is isomorphic to a Clifford algebra over $\Lambda_{\ov{\mb Q}}$ associated to a nondegenerate quadratic form on an $n$-dimensional space.
\end{prop}

Note that the above nondegeneracy condition coincides with the one from Definition \ref{defn:log-derivative} because we are considering Laurent polynomials. The computation of the ring structure is carried out in a similar situation in \cite{Venugopalan_Woodward_Xu}. Here we only sketch it. The key of the computation is to establish another divisor equation
\beq\label{divisor_equation_2}
m_2^\flat (x_i,x_j) + m_2^\flat (x_j, x_i) = \frac{\partial^2 W_{\mf b}}{\partial x_i \partial x_j} \e.
\eeq
on cohomology. When the corresponding critical point of $W_{\mf b}$ is nondegenerate, it follows that the Floer cohomology is isomorphic to a Clifford algebra induced from the Hessian of the critical point.

\begin{rem}
We explain why the divisor equation \eqref{divisor_equation_2} fails on the chain level if one uses the naive way of perturbation. Consider $X = {\mb P}^1$. Fix a torus action. The (undeformed) potential function is
\beqn
W = T^u y + T^{1-u} \frac{1}{y}.
\eeqn
The two terms come from the contribution of two disks, one through the north pole and the other through the south pole. If the divisor equation \eqref{divisor_equation_2} holds, then there should be two configurations with two inputs labelled by the index $1$ critical point, however, once the perturbation is chosen, one can only see one configurations exist in the moduli space. This is because the perturbation is not symmetric with respect to flipping the two incoming semi-infinite edges.
\end{rem}

\begin{proof}[Proof of Proposition \ref{prop_ring_structure}]

Once the divisor equation \eqref{divisor_equation_2} is established, the calculation of the ring structure follows immediately. Hence we only explain how to achieve the divisor equation following the same idea as \cite{Venugopalan_Woodward_Xu}. Notice that the $A_\infty$ structure is independent of the perturbation up to homotopy equivalence. Hence the ring structure on the Floer cohomology is independent of the perturbation. Now we broaden the class of perturbations by considering multi-valued ones in order to achieve some symmetry, and use such perturbations to establish Equation \eqref{divisor_equation_2} on the chain level. A multi-valued perturbation is just a (finite) multi-set of perturbations on each tree. We consider a coherent family of multi-valued perturbations which still satisfy Definition \ref{defn_coherent_perturbation}. We say that a multi-valued perturbation is symmetric, if, when restricted to the tree $\Gamma_0$ with two inputs, one output, and one finite vertex, the perturbation $P_{\Gamma_0, \vec{n}}$ (where $\vec{n}$ on the only finite vertex is $1$, corresponding to Maslov index two disks) is invariant under the ${\mb Z}_2$-action on the universal tree $\ov{\mc{UT}}_{\Gamma_0}$ induced by switching the two incoming semi-infinite edges.

One can follow the same inductive argument to construct a symmetric coherent system of multi-valued perturbations and achieve transversality. Now when defining the counts, we need to count for each member of the multi-valued perturbation and then take an average. This still defines an $A_\infty$ algebra and it is homotopy equivalent to any one defined using single-valued perturbations, provided that we work over the rationals. Moreover, for any two critical points $x_i, x_j$ of Morse index $n-1$, the divisor equation \eqref{divisor_equation_2} holds.  For details, see \cite[Lemma 5.12]{Venugopalan_Woodward_Xu}.
\end{proof}

\subsubsection{Hochschild cohomology}

Now consider the Hochschild cohomology of the $A_\infty$ algebra ${\mc F}_{\mf b}^\flat ({\bm L})$. 

\begin{prop}
When $\L$ corresponds to a nondegenerate critical point of $W_{\mf b}$, one has
\beqn
\hh^\bullet( {\mc F}_{\mf b}^\flat({\bm L})) \cong \Lambda_{\ov{\mb Q}}
\eeqn
where the Hochschild cohomology is generated by the identity ${\bm 1}_{{\mc F}^\flat_{\mf b}({\bm L})}$.
\end{prop}

\begin{proof}
We know that the cohomology of ${\mc F}_{\mf b}^\flat({\bm L})$ is isomorphic to a Clifford algebra over $\Lambda_{\ov{\mb Q}}$. This proposition follows from Proposition \ref{HH_computation_1}.
\end{proof}

\begin{rem}
When the bulk-deformation ${\mf b}$ is convenient, we can formally define the quasimap Fukaya category as the disjoint union of the $A_\infty$ algebras ${\mc F}_{\mf b}^\flat(\L)$ for $\L$ corresponding to all critical points of $W_{\mf b}$ inside the moment polytope. However, what we need is only the direct sum of these Hochschild cohomology.
\end{rem}

\section{Open string theory II. Closed-open maps}\label{section10}

In this section, we prove Theorem \ref{thm:main-2}. It is the consequence of the following theorem.

\begin{thm}\label{thm_CO}
Let ${\mf b}$ be a convenient bulk deformation ${\mf b}$.

\begin{enumerate}
    \item There is an isomorphism of $\Lambda_{\ov{\mb Q}}$-algebras
\beqn
{\rm CO}_\bb: \vhf_\bullet^{\mf b} (V; \Lambda_{\ov{\mb Q}}) \to  \bigoplus_{\L \in {\rm Crit}_X W_{\mf b}} \hh^\bullet ({\mc F}_{\mf b}^\flat (\L)) \cong (\Lambda_{\ov{\mb Q}})^{{\rm Crit}_X W_{\mf b}}.
\eeqn

\item The operator on $\vhf_\bullet^{\mf b}( V; \Lambda_{{\mb  Z}[{\bf i}]} )$ defined by the pair-of-pants product with the (bulk-deformed) first Chern class (see Definition \ref{defn_first_Chern_class}) has distinct eigenvalues in $\Lambda_{\ov{\mb Q}}$.
\end{enumerate}
\end{thm}

\begin{rem}
A closed-open map on the level of Floer cohomology, also in the setting of vortex Floer theory, was constructed in \cite{Wu_Xu}. The method of using quilted objects to prove the multiplicative property was learned from Woodward, see \cite{Venugopalan_Woodward_Xu}.
\end{rem}

\subsection{Moduli spaces for the closed-open map}

\subsubsection{Based trees and closed-open domain types}

Recall our conventions about trees and ribbon trees given in the last section. To model curves with spherical components or Floer cylinders, we consider a broader class of trees called {\it based trees}. A {\bf based tree} is a pair $(\Gamma, \uds\Gamma)$ where $\uds\Gamma$ is a subtree with a ribbon structure containing the root $v_{\rm root}$ and adjacent semi-infinite edge. In a based tree, vertices in $V_{\uds\Gamma}$ are called a {\it boundary vertex}, and other vertices are called {\it interior vertices}. Similarly, an edge is either an interior edge or a boundary edge. A {\bf metric based tree} is a based tree $\Gamma$ together with a metric on its base $\uds\Gamma$. 

Now specify domains responsible for the definition of the closed-open map on the chain level. Consider based trees with a distinguished interior vertex at infinity $v_{\rm Ham}^\infty \in V_\Gamma^\infty\setminus V_{\uds\Gamma}$. For each such tree $\Gamma$, let $v_\ham \in V_{\uds\Gamma}^{\rm finite}$ be the distinguished vertex in the base $\uds\Gamma$ which is closest to $v_\ham^\infty$. We also assume that such trees always have exactly one boundary output $v_{\rm out}$. We call such a tree $\Gamma$ (with a metric type on the base $\uds\Gamma$) a {\bf closed-open domain type}. We say $\Gamma$ is minimal if its base is minimal (see Subsection \ref{subsec:trees}), i.e., the base has no finite edges of length zero or infinite edges. For a minimal $\Gamma$, the base $\uds\Gamma$ has a moduli space $\mc{MT}_{\uds\Gamma}^{\rm CO}$ and universal tree $\mc{UT}_{\uds\Gamma}^{\rm CO}$. One also has a compactification (see Section \ref{subsec:trees}) denoted by $\ov{\mc{MT}}{}_{\uds\Gamma}^{\rm CO}$. 

Given a closed-open domain $\Gamma$, a closed-open domain of type $\Gamma$ is a treed disk $C = S \cup T$ which has a distinguished ``parametrized component'' $C_\ham$ corresponding to the vertex $v_\ham$ which has a nonempty boundary. See Figure \ref{figure:CO_domain} for an illustration of a closed-open domain.

\begin{figure}[h]
    \centering
    \includegraphics{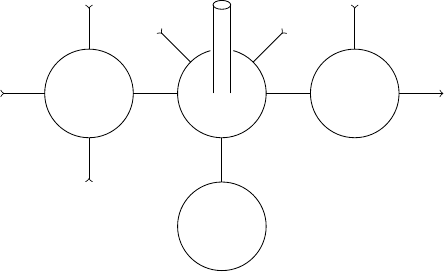}
    \caption{A closed-open domain. The component with a cylindrical end is the component $C^\ham$.}
    \label{figure:CO_domain}
\end{figure}

We define a type of ``mixed equation'' on closed-open domains. Fix an admissible bulk-avoiding pair $(\wh H, \wh J)$ for which the bulk-deformed vortex Floer chain complex $\vcf_\bullet^{\mf b} (\wh H, \wh J; \Lambda_{\ov{\mb Q}})$ is defined. Let $C= S \cup T$ be a closed-open domain with distinguished component $C_\ham$. Because there is at least one boundary output, $C_\ham$ together with the interior puncture and boundary nodes is stable. Hence can always identify $C_\ham \cong {\mb D} \setminus \{0\} \cong (-\infty, 0]\times S^1$ and equip it with the cylindrical metric. Using a cut-off function supported in $(-\infty, -1]$, one can homotope the pair $(\wh H, \wh J)$ with the pair $(0, J_V)$ where $J_V$ is the standard complex structure on the vector space $V \cong {\mb C}^N$, giving rise to a domain-dependent pair $(\wh H_z, \wh J_z)$ for $z \in C_{\rm Ham}$. Given the above data, we consider tuples
\beqn
\Big( (u_v)_{v\in V_\Gamma}, (x_e)_{e\in E_\Gamma} \Big)
\eeqn
where 
\begin{enumerate}
    \item For each vertex $v$ belong to the path connecting $v_{\rm Ham}^\infty$ and $v_\ham$ (not included), $u_v = [u_v, \xi_v, \eta_v]$ is a gauge equivalence class of solutions to the vortex equation
    \begin{align*}
    &\ \partial_s u_v + {\mc X}_{\xi_v} + \wh J_t( \partial_t u_v + {\mc X}_{\eta_v} - X_{\wh H_t}(u_v)) = 0,\ &\ \partial_s \eta_v - \partial_t \xi_v + \mu(u_v) = 0.
    \end{align*}

    \item For $v = v_{\rm Ham}$, $u_v = [u_v, \xi_v, \eta_v]$ is a gauge equivalence class of solutions to 
       \begin{align}\label{CO_equation}
    &\ \partial_s u_v + {\mc X}_{\xi_v} + \wh J_z ( \partial_t u_v + {\mc X}_{\eta_v} - X_{\wh H_z}(u_v)) = 0,\ &\ \partial_s \eta_v - \partial_t \xi_v + \mu(u_v) = 0.
    \end{align}
    Moreover, $u_v$ satisfies the Lagrangian boundary condition
    \beq\label{CO_boundary_condition}
    u_v (\partial C_\ham) \subset \wh L.
    \eeq

    \item For all other $v$, $u_v$ is a $K$-orbit of quasidisk with boundary in $\wh L$.

    \item For each edge $e \in E_\Gamma$, $x_e$ is a (perturbed) negative gradient line/ray/segment of the Morse function $f_L: L \to {\mb R}$.

    \item These objects must have finite energy and must satisfy the obvious matching condition at interior and boundary nodes.
\end{enumerate}

The finite energy condition forces the component $u_v$ whose domain $C_v$ has the distinguished input $v_\ham^\infty$ to converge to an equivariant 1-periodic orbit of $\wh H$. Given a closed-open domain type $\Gamma$, a {\bf closed-open map type} over $\Gamma$, denoted by ${\bm \Gamma}$, consists of topological types of objects for each component. A closed-open map type is called {\bf essential} if there is no interior node and all finite boundary edges have positive length and there is no breaking. 

\subsubsection{Transversality}

Given a closed-open domain type $\Gamma$, a domain-dependent perturbation consists of a domain-dependent smooth function $f_\Gamma$ depending on positions on the universal tree $\ov{\mc{UT}}_\Gamma$ and a domain-dependent almost complex structure $\wh J^{\rm CO}$ depending only on positions on the component $C_\ham \cong (-\infty, 0]\times S^1$. In other words, we keep using the standard complex structure over disk components without interior marked point. As before, the perturbation function $f_\Gamma$ also depends on a function $\vec{n}: V_{\uds\Gamma}^{\rm finite} \setminus \{ v_{\rm Ham}\} \to {\mb Z}_{\geq 0}$. To achieve transversality, one can first fix $\wh J^{\rm CO}$ which is equal to the given $\wh J_t$ near $-\infty$. 

Next we need to extend the perturbation we have chosen to define the (bulk-deformed) quasimap $A_\infty$ algebra of $L$. Notice that for any closed-open domain type $\Gamma$, the base $\uds\Gamma$ has a distinguished finite vertex $v_{\ham}$. The tree $\Gamma$ degenerates to another tree $\Pi$ which has an unbroken component $\Pi'$ that does not contain the distinguished vertex. For such unbroken components $\Pi'$, the domain-dependent perturbation has been chosen as before to define the $A_\infty$ structure. Hence we look for a system of domain-dependent perturbations
\beqn
P_{\Gamma, \vec{n}}^{\rm CO}: \ov{\mc{UT}}_{\uds\Gamma} \to C^\infty(L)
\eeqn
which respect similar conditions as Definition \ref{defn_coherent_perturbation}. We omit the complete definition here. Moreover, we require that, once $\Gamma$ has an unbroken component $\Gamma'$ which does not contain $v_{\rm Ham}$, the perturbation on this component agrees with the existing one chosen before. 

Now we consider relevant moduli spaces. Given a closed-open map type ${\bm \Gamma}$. Let $\vec{n}: V_\Gamma^{\rm finite}\setminus \{v_\ham\} \to {\mb Z}_{\geq 0}$ be the function whose value on $v$ is half of the Maslov index of the disk class $\beta_v$ contained in the data ${\bm\Gamma}$. The moduli space ${\mc M}_{\bm \Gamma}^{\rm CO}$ is the space of solutions to the mixed equation described above for the complex structure $\wh J_z$ in \eqref{CO_equation}, and the negative gradient flow equation with the Morse function $f_L$ perturbed by $P_{\Gamma, \vec{n}}^{\rm CO}$. Then as before, one can find a coherent family of perturbations making all such moduli spaces regular. We omit the details. 

Furthermore, one can incorporate the perturbations used for defining the homotopy units. For this we allow that the inputs of an closed-open domain type to be weighted or unweighted and require similar properties of perturbations on domains with weighted inputs as in Subsection \ref{subsection_homotopy_unit} (the almost complex structure $\wh J^{\rm CO}$ is independent of the weighting parameters $\rho$).

\subsection{The closed-open map}

Having regularized all relevant moduli spaces, we define the relevant counts for the closed-open maps. A closed-open map type ${\bm \Gamma}$ is called {\bf essential} if it is stable, has no breakings and no sphere bubbles, 
no boundary edges of length zero. Given a $k+1$-tuple of generators ${\bf x} = (x_1, \ldots, x_k; x_\infty)$\footnote{Notice that among $x_1, \ldots, x_k$ some of them could be the weighted element $\f$.}, an equivariant 1-periodic orbit ${\mf x}$ of the bulk-avoiding Hamiltonian $\wh H$ 
and a disk class $\beta$, denote by
\beqn
{\mc M}_\beta^{\rm CO}({\mf x}, {\bf x})_i,\ i = 0, 1
\eeqn
the union of moduli spaces ${\mc M}_{\bm \Gamma}^{\rm CO}$ of essential closed-open map types ${\bm \Gamma}$ whose boundary inputs/output are labelled by ${\bf x}$, whose (only) interior input $v_{\rm Ham}^\infty$ is labeled by ${\mf x}$, and whose total disk class is $\beta$, and whose virtual dimension is $i$. Given $E \geq 0$, let 
\beqn
{\mc M}_\beta^{\rm CO}({\mf x}, {\bf x})_i^{\leq E}\subset {\mc M}_\beta^{\rm CO}({\mf x}, {\bf x})_i
\eeqn
be the subset of configurations whose (analytic) energy is at most $E$. 

It is standard to prove the following theorem.

\begin{thm}
\begin{enumerate}
\item ${\mc M}{}_\beta^{\rm CO}({\mf x}, {\bf x})_i$ is an oriented topological manifold of dimension $i$.

\item For all $E\geq 0$, ${\mc M}{}_\beta^{\rm CO}({\mf x}, {\bf x})_0^{\leq E}$ is a finite set. 

\item For all $E \geq 0$, ${\mc M}{}_\beta^{\rm CO}({\mf x}, {\bf x})_1^{\leq E}$ is compact up to at most one 1) interior breaking, 2) boundary breaking, 3) bubbling of holomorphic disks, or 4) the length of a finite boundary edge shrinks to zero.

\item By the standard gluing construction and identifying fake boundary strata, one can compactify the $1$-dimensional moduli space to $\ov{\mc M}{}_\beta^{\rm CO}({\mf x}, {\bf x})_1$ which is an oriented topological 1-manifold with boundary whose cut-off at any energy level $E$ is compact. 
\end{enumerate}
\end{thm}

Now given a local system ${\bf y}$, denote the brane with this local system by $\L = (L, {\bf y})$. We define a count
\beqn
n_{\L, {\mf b}}^{\rm CO} (\beta, {\mf x}, {\bf x}) = \sum_{[{\mf u}] \in {\mc M}_\beta^{\rm CO}({\mf x}, {\bf x})_0} \exp \left( \sum_{j=1}^N \log c_j\ [{\mf u}] \cap V_j \right) T^{E(\beta)} {\bf y}^{\partial \beta} \epsilon([{\mf u}])  \in \Lambda_{\ov{\mb Q}}
\eeqn
where ${\mf b} = \sum_{j=1}^N \log c_j V_j$.
By Gromov compactness one has the following result.
\begin{lemma}
$n_{\L, {\mf b}}^{\rm CO}(\beta, {\mf x}, {\bf x})$ converges in $\Lambda_{\ov{\mb Q}}$.
\end{lemma}

Then define a sequence of linear map 
\beqn
\widetilde {\rm CO}{}_{\L, {\mf b}}^n: \vcf_\bullet^{\mf b}(V; \Lambda_{\ov{\mb Q}} ) \to {\rm Hom}_{\Lambda_{\ov{\mb Q}}} \left( {\mc F}_{\mf b}^+(\L)^{\otimes n}, {\mc F}_{\mf b}^+(\L) \right), n = 0, 1, \ldots
\eeqn
by 
\beqn
\widetilde {\rm CO}{}_{\L,{\mf b}}^n ({\mf x}) (x_n, \ldots, x_1) = \sum_{x_\infty}  n_{\L,{\mf b}}^{\rm CO} ({\mf x}, {\bf x})  x_\infty
\eeqn
and linear extension. 

We use the canonical weakly bounding cochain $b_\L$ to turn it into a chain map. Define 
\beqn
{\rm CO}_{\L, {\mf b}}^n: \vcf_\bullet^{\mf b}(V; \Lambda_{\ov{\mb Q}} ) \to {\rm Hom}_{\Lambda_{\ov{\mb Q}}} \left( {\mc F}_{\mf b}^+(\L)^{\otimes n}, {\mc F}_{\mf b}^+(\L) \right), n = 0, 1, \ldots
\eeqn
by 
\beqn
{\rm CO}_{\L,{\mf b}}^n ({\mf x}) (x_n, \ldots, x_1) = \sum_{l_n, \ldots, l_0} \widetilde{\rm CO}{}_{\L, {\mf b}}^{n+l_0+ \cdots + l_n} \left( \underbrace{b_\L, \ldots, b_\L}_{l_n}, x_n, \cdots, x_1, \underbrace{b_\L, \ldots, b_\L}_{l_0} \right).
\eeqn
The whole sequence $\{{\rm CO}_{\L, {\mf b}}^n\}_{n= 0, \ldots}$ is then a linear map 
\beqn
{\rm CO}_{L, {\mf b}}: \vcf_\bullet^{\mf b}(V; \Lambda_{\ov{\mb Q}}) \to CC^\bullet ( {\mc F}_{\mf b}^+(\L)).
\eeqn

\begin{prop}
${\rm CO}_{\L, {\mf b}}$ is a chain map.
\end{prop}

\begin{proof}
We analyze the boundary of 1-dimensional moduli spaces ${\mc M}{}_\beta^{\rm CO}({\mf x}, x)_1$. Given any map type ${\bm \Gamma}$ contributing to this moduli space, the true boundaries of ${\mc M}_{\bm \Gamma}^{\rm CO}$ consists of configurations where either there is exactly one interior breaking (at an equivariant 1-periodic orbit) or exactly one boundary breaking (see Figure \ref{Figure_CO_chain}).

\begin{figure}[h]
    \centering
    \includegraphics[scale = 0.9]{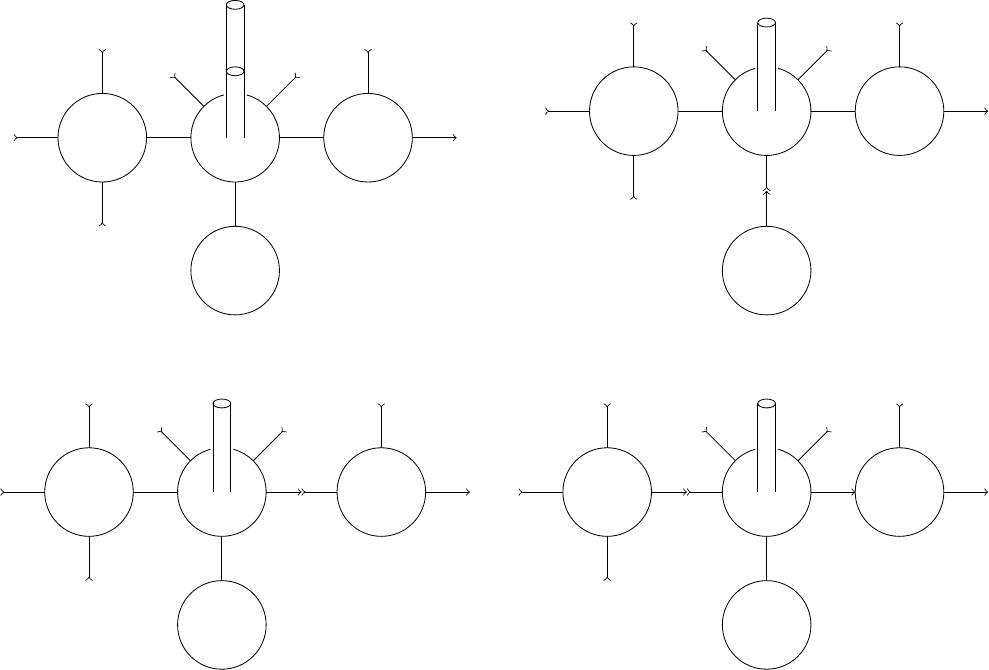}
    \caption{True boundaries of a 1-dimensional moduli space. The pictures represent the case when the weakly bounding cochain is zero and the insertions are all variables of the Hochschild cochains. One can draw the picture for general cases by arbitrarily inserting weakly bounding cochains on the boundary.}
    \label{Figure_CO_chain}
\end{figure}

The configurations with interior breakings contribute to the composition ${\rm CO}_{\L, {\mf b}} \circ \delta_{\vcf}$ (the upper left in Figure \ref{Figure_CO_chain}). On the other hand, there are three types of configurations with boundary breakings, described as follows. 
\begin{enumerate}

\item The first (corresponding to the upper right in Figure \ref{Figure_CO_chain}) is where the breaking separates off a treed disk with no interior puncture or boundary insertions except for an arbitrary number of the weakly bounding cochain $b$. As we have
\beqn
\sum_{k \geq 0} m_k^+(b, \ldots, b) = W_{\mf b} \e^+.
\eeqn
Such configuration contributes by a multiple of the counting of a closed-open moduli space with a boundary insertion $\e^+$, which vanishes by the forgetful property of the perturbation. 

\item The second (corresponding to the lower left in Figure \ref{Figure_CO_chain}) is where the interior puncture and the output are separated by the breaking. This kind of broken configuration contributes to the Gersternhaber product $m^\flat \circ {\rm CO}_{\L, {\mf b}}(-)$ (up to a sign).

\item The third (corresponding to the lower right in Figure \ref{Figure_CO_chain}) is where the interior puncture and the output are not separated by the breaking. This kind of broken configuration contributes to the Gernstenhaber product ${\rm CO}_{\L, {\mf b}}(-) \circ m^\flat$ (up to a sign).
\end{enumerate}
Therefore, up to sign verifications which we skip here, ${\rm CO}_{\L, {\mf b}}$ is a chain map.
\end{proof}

Standard TQFT type argument shows that up to chain homotopy the closed-open map is well-defined, i.e., independent of the pair $(\wh H, \wh J)$ defining the vortex Floer chain complex and independent of the choice of all relevant perturbations.

There is another map on the cohomology level which we also need. Namely, if we do not use any boundary inputs, by counting treed vortices over closed-open domains one can obtain a linear map 
\beq\label{CO2}
{\rm CO}_{\L, {\mf b}}^0: \vhf_\bullet^{\mf b} (V; \Lambda_{\ov{\mb Q}}) \to \qhf^\bullet_{\mf b} (\L; \Lambda_{\ov{\mb Q}}).
\eeq
It was firstly defined in \cite{Wu_Xu} in a slightly different way. Here we can easily generalize to the bulk-deformed case. Moreover, this map sends the identity ${\bm 1}_{\mf b}^\glsm$ to the identity in the Lagrangian Floer cohomology.

Summing over all Floer-nontrivial Lagrangian branes, we define the {\bf closed-open map} 
\beqn
{\rm CO}_{\mf b}:= \bigoplus_{\L \in {\rm Crit}_X W_{\mf b}} {\rm CO}_{\L, {\mf b}} : \vhf_\bullet^{\mf b}(V;\Lambda_{\ov{\mb Q}}) \to \bigoplus_{\L \in {\rm Crit}_X W_{\mf b}} \hh^\bullet( {\mc F}_{\mf b}^+ (\L)).
\eeqn

\subsection{The closed-open map is multiplicative}

Now we establish the following important property of the closed-open map.

\begin{thm}\label{thm_CO_multiplicative}
The map ${\rm CO}_{\L, {\mf b}}$ is multiplicative and maps the unit to the unit. 
\end{thm}

We use an analogue of ``quilted'' moduli spaces to prove the multiplicativity, in the same way as in \cite[Section 3.6]{Venugopalan_Woodward_Xu}.

\begin{defn}[Balanced marked disks and balanced treed disks] \hfill
\begin{enumerate}
    \item A stable marked disk $S \cong {\mb D}$ with two interior markings $z', z''\in {\rm Int} S$ and $k+1$ boundary markings $\uds z = (z_0, \ldots, z_k)$ is called {\bf balanced} if $z', z'', z_0$ lies on a circle in ${\mb D}$ tangent to $\partial {\mb D}$ at $z_0$.

    \item A treed disk with two interior leaves $z', z''$, $k$ boundary inputs and one boundary output is called {\bf balanced} if the following conditions are satisfied.
    \begin{enumerate}

    \item $z', z''$ are contained in the same spherical component. 
    
    \item If $z', z''$ are contained in the same disk component $S_v$. Let $z_0'\in \partial S_v$ be the boundary node connecting $S_v$ to the output. Then $(S_v, z', z'', z_0')$ is a balanced marked disk.

    \item If $z', z''$ are contained in two different disk components, $S_{v'}$ and $S_{v''}$ respectively. Let $e_1, \ldots, e_l$ be the unique path connecting $v'$ and $v''$ in the tree, then 
    \beqn
    \sum_{i=1}^l \pm {\bm l} (e_i) = 0
    \eeqn
    where the sign is positive resp. negative if the edge $e_i$ is oriented toward resp. against the output. We call the unique path $e_1, \ldots, e_l$ the {\bf bridge}.
    \end{enumerate}
\end{enumerate}
\end{defn}

Consider any stable domain type $\Gamma$ with two interior inputs, $k$ boundary inputs and one boundary output. Consider the moduli space ${\mc M}_\Gamma^{\rm balanced}$ of balanced treed disks of type $\Gamma$. The list of codimension one boundary strata is different from the unbalanced case, as the balanced condition cuts down the dimension by 1. See Figure \ref{figure:CO_balanced}. 

\begin{figure}[h]
    \centering
    \includegraphics[scale = 0.5]{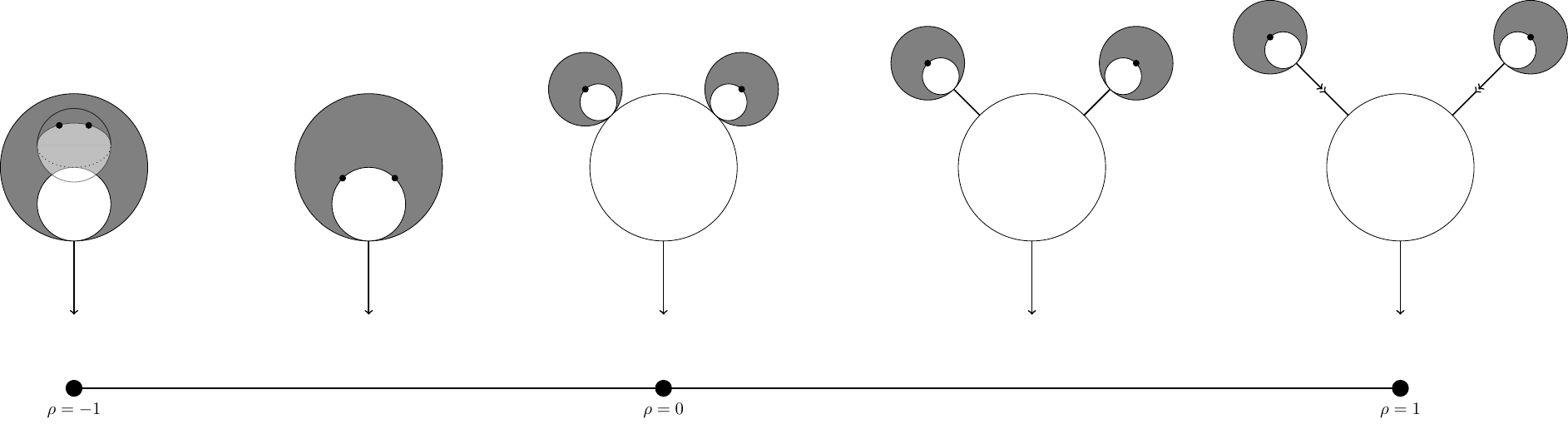}
    \caption{The moduli space of balanced treed disks with two interior inputs and one boundary inputs. This moduli space is parametrized by one variable $\rho \in [-1, 1]$.}
    \label{figure:CO_balanced}
\end{figure}

Notice that a real boundary ${\mc M}_\Pi^{\rm balanced} \subset \partial {\mc M}_\Gamma^{\rm balanced}$ could be the product of several other moduli spaces whose types may have either one interior input or zero interior inputs. We have chosen surface metrics with cylindrical ends for stable closed-open domains (with one interior inputs); hence we can extend the choices to a family of surface metrics with cylindrical ends for the moduli space of stable closed-open domains with two interior inputs. We omit the details.

Now we can consider the following mixed equation for domains with two interior cylindrical ends. Choose two bulk-avoiding admissible pairs $(\wh H_t', \wh J_t')$ and $(\wh H_t'', \wh J_t'')$. Turn on the Hamiltonian perturbation on cylindrical ends. Consider the mixed equation similar to that for the closed-open map. We can extend the existing perturbation to this new type of moduli spaces to achieve transversality.

\begin{proof}[Proof of Theorem \ref{thm_CO_multiplicative}]
Choose two Floer cycles ${\mf x}_1$ and ${\mf x}_2$. We only need to show that  
\beq\label{CO_multiplicative_chain}
{\rm CO}_{\mf b}( {\mf x}_1 \ast_{\mf b} {\mf x}_2) - {\rm CO}_{\mf b}({\mf x}_1) \star {\rm CO}_{\mf b}({\mf x}_2) \in {\rm Im} \delta_{\CC}.
\eeq
As one can choose perfect Morse functions on toric manifolds, we can assume that ${\mf x}_1$ and ${\mf x}_2$ are two single equivariant 1-periodic orbits.  

Consider 1-dimensional moduli spaces of treed disks with two cylindrical ends labelled by ${\mf x}_1$ and ${\mf x}_2$ and arbitrary boundary output $x_\infty$ and inputs
\beqn
\underbrace{b_\L, \ldots, b_\L}_{j_k}, x_k, \cdots, x_1, \underbrace{b_\L, \ldots, b_\L}_{j_0}.
\eeqn
We call $x_k, \ldots, x_1$ {\bf regular inputs}. Consider the true boundaries of such moduli spaces. {\it a priori} There are five types of them, listed as below. We count their contributions (weighted by the bulk deformation), whose sum should be zero.
\begin{enumerate}
    \item Breaking of Floer cylinders at one interior input. As ${\mf x}_1$ and ${\mf x}_2$ are cycles, the contribution of this type of boundary points is zero.

    \item Two cylindrical ends merge together to form a pair of pants. The contribution of this type of boundary is 
    \beqn
    {\rm CO}_{\mf b}( {\mf x}_1 \ast_{\mf b} {\mf x}_2).
    \eeqn

    \item One boundary edge not belonging to the bridge breaks and the piece broken off is not a disk without regular input. The contribution of this type of boundary is a Hochschild coboundary. 

    \item One boundary edge not belonging to the bridge breaks and the piece broken off is a disk without regular input. The broken off piece sums to a multiple of the strict unit $e_\L^+$. By the property of the perturbation data, the contribution of this type of boundary is zero. 

    \item A pair of boundary edges belonging to the bridge break. The contribution of this type of boundary is the Yoneda product 
    \beqn
    {\rm CO}_{\mf b}({\mf x}_1) \star {\rm CO}_{\mf b}({\mf x}_2).
    \eeqn    
\end{enumerate}
Therefore, one obtains \eqref{CO_multiplicative_chain}.

Now we prove the unitality. By the choice of the small bulk deformation, the Hochschild cohomology of the quasimap Fukaya category is semisimple and splits as the direct sum of 1-dimensional pieces. Moreover, each piece is the Hochschild cohomology of the $A_\infty$ algebra ${\mc F}_{\mf b}^+(\L)$, which is linearly spanned by the identity. Hence we only need to prove that the linear map \eqref{CO2} sends the identity ${\bm 1}_{\mf b}^\glsm \in \vcf_\bullet^{\mf b}(V; \Lambda_{\ov{\mb Q}} )$ to the identity element of $\qhf_{\mf b}(\L)$. This verification can be found in \cite[Theorem 6.11]{Wu_Xu} (this verification does not need to consider the homotopy unit and weakly bounding cochains).
\end{proof}

\subsection{The Kodaira--Spencer map}

To prove the first item of Theorem \ref{thm_CO}, it remains to show that the closed-open map is a linear isomorphism. Proposition \ref{prop_same_rank} shows that the domain and the codomain of ${\rm CO}_{\mf b}$ have the same rank
\beqn
{\rm dim}_{\Lambda_{\ov{\mb Q}}} \vhf_\bullet^{\mf b}(V; \Lambda_{\ov{\mb Q}}) = {\rm dim} H^\bullet(X) = \# {\rm Crit}_X W_{\mf b}.
\eeqn
Hence we only need to show that ${\rm CO}_{\mf b}$ is either injective or surjective. 

Following \cite{FOOO_mirror}, we define another closed-open type map which we call the {\it Kodaira--Spencer map} at ${\mf b}$, denoted by
\beqn
\mf{ks}_{\mf b}: \Lambda_{\ov{\mb Q}}[{\bf z}_1, \ldots, {\bf z}_N] \to (\Lambda_{\ov{\mb Q}})^{{\rm Crit}_X W_{\mf b}}.
\eeqn
It is formally the derivative of the bulk-deformed potential function taken at the bulk ${\mf b}$ evaluated at critical points of the Morse function. We only need to use the standard complex structure to define this map. 

\subsubsection{Moduli spaces of quasidisks with tangency conditions}

We go toward the definition of the Kodaira--Spencer map. Fix a Lagrangian $L = L({\bf u})$ for a moment. Let $I = (\alpha_1, \ldots, \alpha_N)$ be a multiindex of nonnegative integers, which defines a monomial
\beqn
{\bf z}^I = {\bf z}_1^{\alpha_1} \cdots {\bf z}_N^{\alpha_N}.
\eeqn
Consider a holomorphic disk $u: ({\mb D}, \partial {\mb D}) \to (V, \wh L)$, which can be classified by Theorem \ref{thm_Blaschke}. We write $u = (u_1, \ldots, u_N)$ in coordinates. We say that $u$ satisfies the $I$-tangency condition at $z\in {\rm Int}{\mb D}$ if $u_i$ vanishes to the order of $\alpha_i$ at $z$, for all $i = 1, \ldots, N$. In particular, when $\alpha_i = 0$, there is no restriction to $u_i$. Given a multiindex $I$ and a disk class $\beta$, denote the moduli space of quasidisks with boundary in $\wh L$ (with one output)  satisfying the $I$-tangency condition at the origin by 
\beqn
{\mc M}_{I,1}^{qd}(\beta).
\eeqn
Its virtual dimension is 
\beqn
{\rm dim}^{\rm vir} {\mc M}_{I,1}^{qd}(\beta) = n + m(\beta) - 2|I| -2.
\eeqn

\begin{rem}
We can put the above moduli space into an infinitely dimensional Banach space where we can specify the tangency conditions for arbitrary maps with sufficiently high regularity. For example, using the setup of Cieliebak--Mohnke \cite[Section 6]{Cieliebak_Mohnke}. Hence we can examine whether the moduli space of quasididks subject to tangency conditions is regular or not.
\end{rem}

\begin{prop}\label{prop_tangency_regularity}
Suppose $\beta = \sum_{j=1}^N d_j \beta_j$ with $d_j \in {\mb Z}$. Then ${\rm dim} {\mc M}_{I, 1}^{qd}(\beta) \neq \emptyset$ only if $d_j \geq \alpha_j$ for all $j$. Moreover, the moduli space ${\mc M}_{I, 1}^{qd}(\beta)$ is smooth and the evaluation map at the boundary marking is a submersion.
\end{prop}

\begin{proof}
By Theorem \ref{thm_Blaschke}, the $j$-th coordinate of the map $u$ of the form \eqref{Blaschke_formula} needs to vanish at least to the order $\alpha_j$ at the origin. Hence $d_j \geq \alpha_j$. 

To prove the regularity of the moduli space ${\mc M}_{I, 1}^{qd}(\beta)$, one only needs to prove the regularity of the corresponding moduli space of holomorphic disks in $V$ with boundary in $\wh L$ (before quotienting the $K$-action) as the $K$-action is free. Since the complex structure on $V\cong {\mb C}^N$ is the standard one, and the tangency condition is imposed on each coordinate independently, one only needs to prove the Fredholm regularity for the $N= 1$ case. In this case, we consider holomorphic disks in ${\mb C}$ with boundary contained in the unit circle, which also vanish to a given order $k$ at the origin. Choose $p>2$ and $m$ sufficiently large, so that one has the Sobolev embedding $W^{m, p} \hookrightarrow C^k$ in dimension two. 

Now fix the disk class $\beta$. Consider the Banach space $W(\beta)$ of maps from $({\mb D}, \partial {\mb D})$ to $({\mb C}, S^1)$ of regularity $W^{m+1, p}$. Let $W_0(\beta) \subset W(\beta)$ be the subspace of maps which vanish at $0$ to the order $k+1$. Let $E(\beta) \to W(\beta)$ be the Banach space bundle, whose fiber over $u$ is the space of $(0, 1)$-forms of regularity $W^{m, p}$, and let $E_0(\beta)\subset E(\beta)$ be the subbundle of those forms which vanish at $0$ to the order $k$. Suppose $u_0: {\mb D} \to {\mb C}$ is a holomorphic disk in $W_0(\beta)$. Then there is a commutative diagram (see \cite[Section 6]{Cieliebak_Mohnke})
\beqn
\xymatrix{ T_{u_0} W(\beta) \ar[r]^{F}  & E(\beta)|_{u_0} \\
           T_{u_0} W_0(\beta) \ar[r]_{F_0} \ar[u] & E_0(\beta)|_{u_0} \ar[u] }
\eeqn
where $F$ resp. $F_0$ is the standard Cauchy--Riemann operator, restricted to corresponding Banach spaces. One needs to prove that $F_0$ is surjective. Notice that by Cho--Oh's theorem \cite[Theorem 6.1]{Cho_Oh}, $F$ is surjective. Hence for each $\eta_0 \in E_0(\beta)|_{u_0}$, there exists $\xi \in T_{u_0}W(\beta)$ such that $F(\xi) = \eta_0$. One only needs to modify $\xi$ to some $\xi_0 \in T_{u_0} W_0(\beta)$ with $F(\xi) = F(\xi_0)$. Indeed, as $u_0$ vanishes up to order $k+1$ at the origin, the disk class $\beta$, which is only the degree of the map $u_0$, is at least $k+1$. Then by the Blaschke formula \eqref{Blaschke_formula}, one can easily deform $u_0$ by $k+1$-jet data. Such deformations are in the kernel of $F$. Hence we can obtain the desired $\xi_0$. This proves the Fredholm regularity of the moduli spaces.

The fact that the evaluation map at the output is a submersion onto $L$ follows easily from the Blaschke formula.
\end{proof}

\subsubsection{The derivative of the potential}

Now we can define the Kodaira--Spencer map. For each critical point $\L \in {\rm Crit}_X W_{\mf b}$ (lying inside the moment polytope), we will define a linear map 
\beqn
\wt{\mf{ks}}_{\L, {\mf b}}: \Lambda_{\ov{\mb Q}}[{\bf z}_1, \ldots, {\bf z}_N] \to \qcf_\bullet^+ (\L; \Lambda_{\ov{\mb Q}})
\eeqn
using the counts of certain zero-dimensional moduli spaces. It will turn out that the value of this map is always a multiple of the unique maximum $\e_\L = x_{\rm max}\in {\rm Crit} f_L$, hence descends to a map 
\beqn
\wt{\mf{ks}}_{\L, {\mf b}}: \Lambda_{\ov{\mb Q}}[{\bf z}_1, \ldots, {\bf z}_N] \to \qhf_\bullet^{\mf b}(\L; \Lambda_{\ov{\mb Q}}).
\eeqn
We define the coefficients to be ${\mf{ks}}_{\L, {\mf b}}$, i.e., 
\beqn
\wt{\mf{ks}}_{\L, {\mf b}}({\bf z}^I) = \mf{ks}_{\L, {\mf b}}( {\bf z}^I) [\e_\L].
\eeqn

We first fix a multiindex $I = (\alpha_1, \ldots,\alpha_N)$. Denote 
\beqn
\beta^I = \alpha_1 \beta_1 + \cdots + \alpha_N \beta_N \in H_2(V, \wh L).
\eeqn
For each disk class $\beta \in H_2(V, \wh L)$ and each critical point $x \in {\rm Crit} f_{L({\bf u})}$ of the Morse function $f_{L({\bf u})}: L({\bf u}) \to {\mb R}$, consider the moduli space 
\beqn
{\mc M}_{I, 1}^{qd}(\beta; x)
\eeqn
where we require that the output converges to the critical point $x$. Proposition \ref{prop_tangency_regularity} implies that this moduli space is regular. Moreover, 
\beqn
{\mc M}_{I, 1}^{qd}(\beta; x) \neq \emptyset\ {\rm and}\ {\rm dim}{\mc M}_{I, 1}^{qd}(\beta; x) = 0 \Longrightarrow \beta = \beta^I\ {\rm and}\ x = x_{\rm max}.
\eeqn
Moreover, in this case, the moduli space has exactly one point because of the Blaschke formula. We count the unique element weighted by the bulk deformation and the local system. 

\begin{rem}
{\it A priori} we should consider treed holomorphic disks with one boundary output and one interior marking with certain tangency condition. It is similar to the case of proving that $m_0$ is a multiple of ${\bm e}_{\bm L}$ that one can prove for zero-dimensional moduli spaces, only those treed disks with one disk component contribute. 
\end{rem}

The count of the above moduli spaces (with a single point) defines the Kodaira--Spencer map. More explicitly, define 
\beqn
\wt{\mf{ks}}_{\L, {\mf b}}: \Lambda_{\ov{\mb Q}}[{\bf z}_1, \ldots, {\bf z}_N] \to \qhf_{\mf b}^\bullet( \L; \Lambda_{\ov{\mb Q}})
\eeqn
by 
\beqn
\wt{\mf{ks}}_{\L, {\mf b}}({\bf z}^I)= {\mf b}^I T^{E(\beta^I)} {\bf y}^{\partial \beta^I} [\e_\L] = \mf{ks}_{\L, {\mf b}}({\bf z}^I) [\e_\L].
\eeqn
Here for ${\mf b} = \sum_{j=1}^N \log c_j V_j$, the notation ${\mf b}^I$ denotes the quantity
\beqn
c_1^{\alpha_1} \cdots c_N^{\alpha_N},
\eeqn
which is the exponential of the intersection number between the above unique quasidisk in ${\mc M}_{I, 1}^{qd}(\beta^I; x_{\rm max})$ and the bulk ${\mf b}$.

The Kodaira--Spencer map takes a very simple form. Recall that we have written 
\beqn
W_{\mf b} = W_{{\mf b}, 1} + \cdots + W_{{\mf b}, N} = c_1 W_1 + \cdots + c_N W_N.
\eeqn

\begin{prop}\label{prop_ks_computation}
For each multiindex $I$, one has 
\beq\label{ks_formula}
\mf{ks}_{\mf b}( {\bf z}^I ) = W_{\mf b}^I:= W_{{\mf b}, 1}^{\alpha_1} \cdots W_{{\mf b}, N}^{\alpha_N}.
\eeq
\end{prop}
\begin{proof}
The calculation is carried out in a straightforward way. The area of a disk in class $\beta^I$ is 
\beqn
E(\beta^I) = \alpha_1 l_1({\bf u}) + \cdots + \alpha_N l_N({\bf u}).
\eeqn
The contribution of the local system is 
\beqn
{\bf y}^{\partial \beta^I} = \prod_{j=1}^N (y_1^{v_{j, 1}} \cdots y_n^{v_{j, n}})^{\alpha_j}.
\eeqn
Hence the formula \eqref{ks_formula} follows.
\end{proof}

Define the Kodaira--Spencer map by
\beqn
\mf{ks}_{\mf b}:= \bigoplus_{\L \in {\rm Crit}_X W_{\mf b}} \mf{ks}_{\L, {\mf b}}.
\eeqn

\begin{thm}\label{thm_ks_surjective}
The Kodaira--Spencer map $\mf{ks}_{\mf b}$ is surjective.
\end{thm}

\begin{proof}
By \cite[Lemma 3.12]{FOOO_toric_2}, the monomials $W_1, \ldots, W_N$ generate (over the ring $\Lambda_{0, \ov{\mb Q}}$) the ring 
\beqn
\Lambda_0^P \llangle y_1^\pm, \ldots, y_n^\pm \rrangle
\eeqn
which is a ring of formal Laurent series satisfying a particular valuation condition determined by the moment polytope $P$. Let ${\bm \eta}_1, \ldots, {\bm \eta}_s$ be the critical points of $W_{\mf b}$ inside the moment polytope. Using the notion of convergent Novikov field $\Lambda_{\ov{\mb Q}}^{\rm conv}$, we see that for $T = t$ being a sufficiently small nonzero complex number, ${\bm \eta}_1^t, \ldots, {\bm \eta}_s^t$ are distinct points in $({\mb C}^*)^n$. Then there exist $s$ complex Laurent polynomials
\beqn
F_1, \ldots, F_s \in {\mb C}[y_1, \ldots, y_n, y_1^{-1}, \ldots, y_n^{-1}]
\eeqn
such that the matrix $\left[ F_a({\bm \eta}_b^t) \right]_{1\leq a, b \leq s}$ is invertible. Regard $F_1, \ldots, F_s$ as Laurent polynomials with Novikov coefficients, we see the determinant of the matrix
\beqn
\det \left[ F_a({\bm \eta}_b) \right]_{1 \leq a, b \leq s} \neq 0 \in \Lambda_{\ov{\mb Q}}.
\eeqn
The above is still true if we replace $F_a$ by $T^A F_a$ for any $A\in {\mb R}$. On the other hand, for $A$ sufficiently large, $T^A F_a \in \Lambda_0^P \llangle y_1^\pm, \ldots, y_n^\pm \rrangle$. This implies that, the restriction of $\mf{ks}_{\mf b}$ to the finite-dimensional subspace spanned by $T^A F_a$ is subjective due to the generation property of the monomials $W_1, \dots, W_N$. Hence $\mf{ks}_{\mf b}$ is also surjective.
\end{proof}

\subsection{A quantum Kirwan map}

The set of small bulk-deformations is contained in the larger set of equivariant cohomology upstairs. Classically, there is the Kirwan map 
\beqn
\kappa^{\rm classical}: H_K^\bullet(V) \to H^\bullet(X).
\eeqn
In principle, by incorporating vortices one can define a quantization of the Kirwan map. This has been pursued by Ziltener \cite{Ziltener_book} in the symplectic setting and worked out by Woodward \cite{Woodward_15} in the algebraic setting. Here we define a variant of the quantum Kirwan map, denoted by 
\beq\label{qkirwan}
\kappa_{\mf b}: \Lambda_{{\mb Z}[{\bf i}]}[{\bf z}_1, \ldots, {\bf z}_N] \to \vhf_\bullet^{\mf b} (V; \Lambda_{{\mb Z}[{\bf i}]})
\eeq
such that the image of the unit $1$ is the identity ${\bm 1}_{\mf b}^\glsm$. 

We define the above map by imposing tangency conditions at the origin of the cigar. Fix a regular bulk-avoiding admissible pair $(\wh H_\infty, \wh J_\infty)$ which defines a bulk-deformed vortex Floer complex $\vhf_\bullet^{\mf b}( \wh H, \wh J; \Lambda_{{\mb Z}[{\bf i}]})$. Consider a domain-dependent almost complex structure $\wh J$ (resp. Hamiltonian perturbation $\wh H$) parametrized by points on the cigar $\Sigma^{\rm cigar} \cong {\mb C}$ which is equal to the standard almost complex structure $\wh J_V$ (resp. vanishes) in a specified neighborhood of $0\in \Sigma^{\rm cigar}$ and which agrees with $\wh J_\infty$ (resp. $\wh H_\infty$) near infinity. Consider the vortex equation with the data $(\wh H, \wh J)$ on the cigar. Any finite energy solution should converge to a critical point of ${\mc A}_{H_\infty}$. Moreover, as the almost complex structure is standard near $0$, one can impose the tangency condition corresponding to $I$ at the origin. Such a tangency condition is gauge invariant. Then for each critical point ${\mf x} \in {\rm Crit} {\mc A}_{H_\infty}$, there is a moduli space
\beqn
{\mc M}{}_I^{\rm cigar}({\mf x}) \subset {\mc M}{}^{\rm cigar}({\mf x}).
\eeqn
By using domain-dependent perturbations, one can achieve transversality for such a moduli space. Then one has 
\beqn
{\rm dim} {\mc M}_I^{\rm cigar}({\mf x}) = {\rm dim} {\mc M}{}^{\rm cigar}({\mf x}) - 2 |I|.
\eeqn
On the other hand, as the Hamiltonian is bulk-avoiding, each solution has well-defined topological intersection numbers with $V_j$. Then define 
\beqn
\kappa_{\mf b}({\bf z}^I) = \sum_{{\mf x}\atop {\rm dim}{\mc M}_I^{\rm cigar}({\mf x}) = 0 } \left( \sum_{[{\mf u}] \in {\mc M}_I^{\rm cigar} ({\mf x})} \left( \prod_{j=1}^N c_j^{[{\mf u}] \cap V_i} \right) \epsilon ([{\mf u}]) \right)  {\mf x}.
\eeqn

\begin{thm}[Properties of the bulk-deformed quantum Kirwan map]\hfill
\begin{enumerate}
\item The element $\kappa_{\mf b}({\bf z}^I)$ is a legitimate element of $\vcf_\bullet^{\mf b} (\wh H_\infty, \wh J_\infty; \Lambda_{{\mb Z}[{\bf i}]})$ and is $\partial^{\mf b}$-closed. Moreover, its homology class is independent of the choice of perturbation and its corresponding element in $\vhf_\bullet^{\mf b}(V; \Lambda_{{\mb Z}[{\bf i}]})$ is well-defined.    

\item $\kappa_{\mf b}(1) = {\bm 1}_{\mf b}^\glsm$.

\end{enumerate}
\end{thm}

\begin{proof}
The first conclusion follows from the standard argument and the second one follows from the definition of ${\bf 1}_{\mf b}^\glsm$.
\end{proof}

We define another element in the vortex Floer homology which can be viewed as the first Chern class in the bulk-deformed Hamiltonian Floer homology, or the image of thee first Chern class under the bulk-deformed PSS map. Recall that the first Chern class of a toric manifold is naturally represented by the union of toric divisors. Upstairs, they are the union of all coordinate hyperplanes. 

\begin{defn}\label{defn_first_Chern_class}
The {\bf ${\mf b}$-deformed first Chern class} is the element
\beqn
\kappa_{\mf b}({\bf z}_1 + \cdots + {\bf z}_N) \in \vhf_\bullet^{\mf b}(V; \Lambda_{{\mb Z}[{\bf i}]})
\eeqn
Denote the operator on $\vhf_\bullet^{\mf b}(V; \Lambda_{{\mb Z}[{\bf i}]})$ defined by the pair-of-pants product with the ${\mf b}$-deformed first Chern class by 
\beq
{\mb E}_{\mf b}: \vhf_\bullet^{\mf b}(V; \Lambda_{{\mb Z}[{\bf i}]}) \to \vhf_\bullet^{\mf b}(V; \Lambda_{{\mb Z}[{\bf i}]}).
\eeq
\end{defn}

\subsection{The commutative diagram}

We prove the following proposition.

\begin{prop}\label{prop_commutative_diagram}
When the bulk deformation ${\mf b}$ is convenient, the following diagram commutes.
\beq\label{commutative_diagram}
\vcenter{\xymatrix{   \Lambda_{\ov{\mb Q}} [{\bf z}_1, \ldots, {\bf z}_N] \ar[rr]^-{\mf{ks}_{\mf b}}  \ar[d]_{\kappa_{\mf b}}  & & 
  (\Lambda_{\ov{\mb Q}})^{{\rm Crit}_X W_{\mf b}} \ar[d] \\ \vhf_{\mf b} (V; \Lambda_{\ov{\mb Q}}) \ar[rr]_{{\rm CO}_{\mf b}}  & &  \displaystyle \bigoplus_{\L \in {\rm Crit}_X W_{\mf b}} HH^\bullet( {\mc F}_{\mf b}^\flat( \L))    }}
\eeq
Here the right vertical arrow is the natural identification induced by the individual isomorphisms $\hh^\bullet( {\mc F}_{\mf b}^\flat( \L)) \cong \Lambda_{\ov{\mb Q}}$.
\end{prop}

\begin{proof}
We turn on Hamiltonian perturbations on disks to construct a homotopy between the Kodaira--Spencer map and the closed-open map composed with the quantum Kirwan map. Fix a critical point of $W_{\mf b}$ lying in the interior of the moment polytope with the corresponding Lagrangian brane $\L = (L({\bf u}), {\bf y})$. We claim that the following diagram commutes.
\beq\label{commutative_diagram_2}
\vcenter{ \xymatrix{   \Lambda_{\ov{\mb Q}} [{\bf z}_1, \ldots, {\bf z}_N] \ar[rrr]^-{\wt{\mf{ks}}_{\L, {\mf b}}}  \ar[d]_{\kappa_{\mf b}}  & & & 
  \qhf_{\mf b}^\bullet( \L; \Lambda_{\ov{\mb Q}})  \ar[d]  \\ \vhf_{\mf b} (V; \Lambda_{\ov{\mb Q}}) \ar[rrr]_{{\rm CO}_{\L, {\mf b}}^0}  & & & \qhf_{\mf b}^\bullet( \L; \Lambda_{\ov{\mb Q}})   }   }
\eeq
Once this is established, it follow that the image ${\rm CO}_{\L, {\mf b}}^0 \circ \kappa_{\mf b}$ is contained in the line spanned by the identity element of $\qhf_{\mf b}^\bullet(\L;\Lambda_{\ov{\mb Q}})$. Hence on the chain level, one has 
\beqn
{\rm CO}_{\L, {\mf b}}^0 ( \kappa_{\mf b}({\bf z}^I)) - 
 \mf{ks}_{\L, {\mf b}}({\bf z}^I) \e_\L^+  \in {\rm Im} (m_1^\flat).
\eeqn
As the Hochschild cohomology of $\L$ is spanned by the identity element, it follows that the diagram \eqref{commutative_diagram} also commutes. 

Now we prove that \eqref{commutative_diagram_2} commutes. Consider closed-open domains with one interior marking. Define a 1-parameter family of equations parametrized by $\nu \in [0, 1]$ such that when $\nu = 0$, the equation is the quasidisk equation with tangency condition at the marking. When $\nu$ is positive, we stretch a neighborhood of the interior marking and turn on a Hamiltonian perturbation by a bulk-avoiding pair $(\wh H, \wh J)$. We always require the tangency condition at the interior marking. As for boundary insertions, we only allow the boundary inputs to be labelled by the canonical weakly bounding cochain $b_\L$, while the boundary output can be labelled by any critical point of $f_L$. We can consider such moduli spaces with the tangency condition corresponding to multiindex $I$, total disk class $\beta$, and the output labelled by $x \in {\rm Crit} f_L$. 

One can use similar arguments as before to regularize relevant moduli spaces using perturbations which naturally extend existing perturbations defining the $A_\infty$ structure, the closed-open map, and the quantum Kirwan map. Then by counting elements in zero-dimensional moduli spaces, one can define a linear map 
\beqn
R_\L: \Lambda_{\ov{\mb Q}}[{\bf z}_1, \ldots, {\bf z}_N] \to \qcf^\bullet( \L; \Lambda_{\ov{\mb Q}}) \subset \qcf^\bullet( \L; \Lambda_{\ov{\mb Q}})^+.
\eeqn

Now we consider boundaries of 1-dimensional moduli spaces. There are the following types of boundary strata.
\begin{enumerate}
    \item The boundary at $\nu = 0$. This side of the boundary consists of points in zero-dimensional moduli spaces used to define the Kodaira--Spencer map. The contribution of these boundary points is equal to $\wt{\mf{ks}}_{\L, {\mf b}}$.

    \item The boundary at $\nu = 1$. This side of the boundary consists of configurations having exactly one interior breaking at certain equivariant 1-periodic orbit of the Hamiltonian $\wh H$. As the perturbation extends the perturbations chosen for the closed-open map and the quantum Kirwan map, the contribution of these boundary points is equal to 
    \beqn
    {\rm CO}_{\L, {\mf b}}^0 \circ \kappa_{\mf b}.
    \eeqn

    \item Boundary points at $\nu \in (0, 1)$. These configurations have exactly one boundary breakings. There are two possibilities. First, the interior puncture and the boundary output are in the same unbroken component. In this case, the other unbroken component is a treed quasidisk with only boundary insertions being the canonical weakly bounding cochain $b_\L$. As the perturbation satisfies the forgetful property when one input is unweighted (the strict unit $\e^+$), the contribution of this kind of boundary points is zero. Second, the interior puncture and the boundary output are in two different unbroken component. The contribution of such configurations is 
    \beqn
    m_1^\flat ( R_\L( {\bf z}^I))
    \eeqn
    which is exact.     
\end{enumerate}
Therefore, it follows that on the chain level, one has 
\beqn
\wt{\mf{ks}}_{\L, {\mf b}} ({\bf z}^I) - {\rm CO}_{\L, {\mf b}}^0( \kappa_{\mf b}({\bf z}^I)) \in {\rm Im} (m_1^\flat).
\eeqn
Hence on the cohomology level the diagram \eqref{commutative_diagram_2} commutes.
\end{proof}

Because the Kodaira--Spencer map is surjective, this finishes the proof of item (1) of Theorem \ref{thm_CO}. 

\subsection{Quantum multiplication by the first Chern class}

Now we prove item (2) of Theorem \ref{thm_CO}. We prove the following theorem.

\begin{thm}\label{thm_CO_spectrality}
When ${\mf b}$ is a convenient small bulk deformation, the operator ${\mb E}_{\mf b}$ on $\vhf_\bullet^{\mf b} (V; \Lambda_{\ov{\mb Q}})$ has an eigenspace decomposition 
\beqn
\vhf_\bullet^{\mf b} (V; \Lambda_{\ov{\mb Q}}) = \bigoplus_{\L \in {\rm Crit}_X (W_{\mf b})} \vhf_\bullet^{\mf b} (V; \Lambda_{\ov{\mb Q}})_{W_{\mf b}(\L)}.
\eeqn
\end{thm}

\begin{proof}
By item (1) of Theorem \ref{thm_CO}, $\vhf_\bullet^{\mf b} (V; \Lambda_{\ov{\mb Q}})$ is semisimple. Hence ${\mb E}_{\mf b}$ is diagonalizable, and let the eigenvalues be $\lambda_1, \ldots, \lambda_m$. Now take an eigenvalue $\lambda = \lambda_i$ and a critical point $\L = (L({\bf u}), {\bf y})\in {\rm Crit}_X W_{\mf b}$. We consider the restriction of ${\rm CO}_{\L, {\mf b}}^0: \vhf_\bullet^{\mf b}(V;\Lambda_{\ov{\mb Q}}) \to \qhf_{\mf b}^\bullet(\L; \Lambda_{\ov{\mb Q}})$ to the $\lambda$-eigenspace. We prove that this map is nonzero only when $\lambda$ coincides with the critical value. Consider closed-open domains with two interior markings, one boundary output, and arbitrarily many boundary inputs (to be labelled by the canonical weakly bounding cochain of $\L$). We distinguish the two interior markings. The first one is $v_{\ham}$, which will be labelled by an equivariant 1-periodic orbit. The second one is denoted by $v_{\rm Chern}$, which will be labelled by components of the equivariant toric divisor. Given any such closed-open domain $C = S \cup T$ where $S$ is the surface part and $T$ is the tree part, the marking corresponding to $v_{\rm Ham}$ becomes a puncture while the marking corresponding to $v_{\rm Chern}$ is denoted by $z_{\rm Chern} \in {\rm Int} S$. 

We would like to include one more constraints on the position of $v_{\rm Chern}$. In the same way as defining the closed-open map, there is a distinguished component $C_\ham$ of such domains $C = S \cup T$. Because the domain $C$ has a distinguished output, we can identify $C_\ham$ with ${\mb D}\setminus \{0\}$ canonically such that the boundary node on $C_{\rm Ham}$ leading towards the output is the point $1 \in S^1 \cong \partial {\mb D}$.
Define the {\bf offset angle} of $z_{\rm Chern}$ as follows.

\begin{enumerate}
    \item If $z_{\rm Chern}$ is in a cylindrical component, it does not have an offset angle.

    \item If $z_{\rm Chern}$ is on $C_\ham \cong {\mb D} \setminus \{0\}$, then the offset angle is the angular coordinate of $z_{\rm Chern}$.

    \item If $z_{\rm Chern}$ is not on $C_\ham$ or any cylindrical component, then there is a unique boundary node on $C_\ham$ connecting $C_\ham$ to $z_{\rm Chern}$. The offset angle is the angular coordinate of this boundary node. 
\end{enumerate}

We fix $\theta \in S^1 \setminus \{1\}$ and only consider closed-open domains described as above such that the offset angle of $z_{\rm Chern}$ is equal to $\theta$ or does not have an offset angle. Consider the same equation defining the closed-open maps on such domains with possibly different perturbation data, where on the cylindrical end one has the Hamiltonian perturbation by a regular bulk-avoiding pair $(\wh H, \wh J)$, and along the boundary one imposes the Lagrangian boundary condition from $\L$.  

We analyze the true boundaries of 1-dimensional such moduli spaces. We assume that the cylindrical end is labelled by a cycle $a$ in $\vcf_\bullet^{\mf b}(V; \Lambda_{\ov{\mb Q}})$. The true boundary components corresponding to configurations which have exactly one breaking, either an interior one or a boundary one. See Figure \ref{Figure_CO_spectrality}.

\begin{figure}[h]
    \centering
    \includegraphics[scale=0.85]{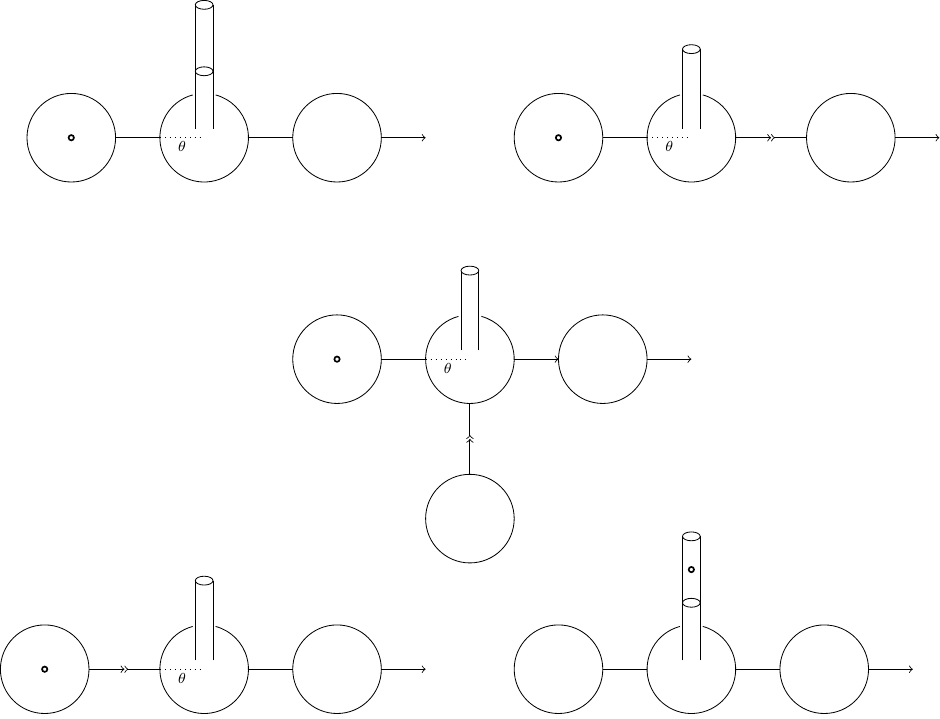}
    \caption{Boundary of 1-dimensional moduli spaces with one special interior marking.}
    \label{Figure_CO_spectrality}
\end{figure}

\begin{enumerate}

\item The breaking is interior and the special marking $z_{\rm Chern}$ is not on a cylindrical component. The sum of this kind of contributions is zero as the interior input is a cycle. Note that as we are counting treed holomorphic disks, the line segment connecting the component on which $z_{\rm Chern}$ lies and $C_{\rm Ham}$ is not meant to be a breaking.

\item The breaking is boundary at the offset angle $1 \in S^1$ (which is different from $\theta$) hence separates $C_\ham$ and the output. The sum of this kind of configuration is a coboundary in $\qcf_{\mf b}(\L)$, which is zero in cohomology.

\item The breaking is boundary at an offset angle different from $1 \in S^1$ and $\theta$ hence does not separate $C_\ham$ and the output. The disk bubble contributes to a multiple of the strict unit $\e_\L^+$. Hence by the forgetful property of the perturbation data, the contribution of such configurations is zero.

\item The breaking is boundary at the specified offset angle $\theta$ which separates the special marking $z_{\rm Chern}$ and the component $C_\ham$. The disk bubble always has Maslov index 2, hence the interior constraint imposed at $z_{\rm Chern}$ gives a factor $1$. Hence the disk bubble contributes to $W_{\mf b}(\L) e_\L^+$. However, as the offset angle is fixed, there are such rigid configurations, and the counting is equal to 
\beqn
W_{\mf b}(\L) \cdot {\rm CO}_{\L, {\mf b}}^0 (a).
\eeqn

\item The breaking is interior and the special marking $z_{\rm Chern}$ is on the cylindrical component that breaks off. This kind of configuration contributes to 
\beqn
{\rm CO}_{\L, {\mf b}}^0 ({\mb E}_{\mf b}(a)) = \lambda\cdot {\rm CO}_{\L, {\mf b}}^0 (a),
\eeqn
due to the appearance of the pair-of-pants product in the upper component.
\end{enumerate}
The analysis above shows that in cohomology, one has 
\beqn
\lambda\cdot {\rm CO}_{\L, {\mf b}}^0 (a) = W_{\mf b}(\L) \cdot {\rm CO}_{\L, {\mf b}}^0 (a).
\eeqn
Hence if $\lambda \neq W_{\mf b}(\L)$, the map ${\rm CO}_{\L, {\mf b}}^0$ vanishes on this eigenspace. 

On the other hand, the linear map 
\beqn
\bigoplus_{\L \in {\rm Crit}_X W_{\mf b}} {\rm CO}_{\L,{\mf b}}^0: \vhf_\bullet^{\mf b}(V; \Lambda_{\ov{\mb Q}}) \to \bigoplus_{\L\in {\rm Crit}_X W_{\mf b}} \qhf_{\mf b}^\bullet (\L; \Lambda_{\ov{\mb Q}})
\eeqn
is injective, because when we take the component generated by the identity elements of $\qhf_{\mf b}^\bullet( \L; \Lambda_{\ov{\mb Q}})$, it descends to the isomorphism ${\rm CO}_{\mf b}$ onto the direct sum of the Hochschild cohomology. Therefore, one has
\beq\label{eqn_spectrality}
{\rm Spec} ({\mb E}_{\mf b}) \subset W_{\mf b} ({\rm Crit}_X W_{\mf b}).
\eeq
On the other hand, for each critical point $\L \in {\rm Crit}_X (W_{\mf b})$, the closed-open map ${\rm CO}_{\L, {\mf b}}^0$ is unital hence nonzero. This implies that $W_{\L, {\mf b}} \in \Lambda$ is also an eigenvalue of ${\mb E}_{\mf b}$. Hence \eqref{eqn_spectrality} is an identity. As when ${\mf b}$ is convenient, all critical values are distinct, it follows that all eigenspaces of ${\mb E}_{\mf b}$ are 1-dimensional.
\end{proof}

\bibliography{reference}
\bibliographystyle{amsalpha}

\end{document}